\documentclass[11pt,oneside,reqno,a4paper]{paper}
\usepackage[utf8]{inputenc}
\usepackage{amsmath,amsthm}
\usepackage{lmodern}
\usepackage[T1]{fontenc}        
\usepackage[numbers]{natbib}  

\usepackage{graphicx}
\usepackage{url}
\usepackage{amssymb}
\usepackage{mathtools}
\usepackage{bbold}
\usepackage[all]{xy}
\usepackage{xcolor}
\usepackage{mdframed}
\usepackage{makeidx} 
\usepackage{float}
\usepackage{xy}
\usepackage{placeins}

\usepackage{tabularx}
\newcolumntype{C}{>{\centering\arraybackslash}X}
\usepackage{hyphenat}
\usepackage{booktabs}
\usepackage{pdfpages}
\usepackage{array}
\usepackage{hyperref}
\usepackage{datetime}
\usepackage{enumitem}

\allowdisplaybreaks

\title{On the Convergence of the Mesh-Free Discretization of a Non-Linear Bond-Based Peridynamic Model with a Bond-Breakage Criterion}
\author{Lukas Pflug\thanks{FAU Competence Center Scientific Computing, Friedrich-Alexander University of Erlangen-Nuremberg, Martensstr. 5a, 91058 Erlangen, Germany}\and Michael Stingl\thanks{Department of Continuous Optimization, Friedrich-Alexander University of Erlangen-Nuremberg, Cauerstraße 11, 91058 Erlangen, Germany}\and Max Zetzmann\thanks{Department of Continuous Optimization, Friedrich-Alexander University of Erlangen-Nuremberg, Cauerstraße 11, 91058 Erlangen, Germany, Corresponding author. E-mail: max.zetzmann@fau.de}}
\date{October 2026}

\newcounter{remarkcnt}
\counterwithin{remarkcnt}{section}
\newcommand{\dd}{\ \mathbf{d}}

\newcommand{\qedwhite}{\hfill \ensuremath{\Box}}

\renewcommand{\exp}{\operatorname{exp}}
\newcommand{\diam}[0]{\operatorname{diam}}

\newcommand{\vol}[0]{\operatorname{vol}}

\newcommand{\id}[0]{\mathrm{id}}
\newcommand{\img}[0]{\operatorname{img}}

\newcommand{\scalarlaw}[0]{g}

\newcommand{\dist}[0]{\text{dist}}

\renewcommand{\b}[1]{\mathbf{#1}} 
\newcommand{\ca}[1]{\mathcal{#1}}

\newcommand{\supp}[0]{\text{ supp}}

\makeatletter
\renewcommand\paragraph{\@startsection{paragraph}{4}{\z@}%
{-3.25ex\@plus -1ex \@minus -.2ex} {1.5ex \@plus .2ex} {\normalfont\normalsize\bfseries}}
\makeatother

\newcommand{\Pk}[0]{\ca{P}_\kappa}
\newcommand{\Kk}[0]{\ca{K}_\kappa}
\newcommand{\Pkn}[0]{\ca{P}_{\kappa_n}}
\newcommand{\Kkf}[1]{\ca{K}_{\kappa_{#1}}}
\newcommand{\Map}[0]{\operatorname{Map}}
\newcommand{\Meas}[0]{\operatorname{Meas}}

\newcommand{\theoskip}[0]{\medskip}
\newcommand{\theopreskip}[0]{\medskip}
\newcommand{\IP}[0]{\operatorname{IP}}
\newcommand{\EOC}[0]{\mathrm{EOC}}

\DeclareMathOperator*{\esssup}{ess\,sup}

\newenvironment{theorem}[1][]{\refstepcounter{remarkcnt}\par\theopreskip\noindent\textbf{Theorem \theremarkcnt}
\if\relax\detokenize{#1}\relax
\else
(\textit{#1})
\fi}{\par\theoskip}

\newenvironment{lemma}[1][]{\refstepcounter{remarkcnt}\par\theopreskip\noindent\textbf{Lemma \theremarkcnt}
\if\relax\detokenize{#1}\relax
\else
(\textit{#1})
\fi}{\par\theoskip}

\newenvironment{corollary}[1][]{\refstepcounter{remarkcnt}\par\theopreskip\noindent\textbf{Corollary \theremarkcnt}
\if\relax\detokenize{#1}\relax
\else
(\textit{#1})
\fi}{\par\theoskip}

\newenvironment{proposition}[1][]{\refstepcounter{remarkcnt}\par\theopreskip\noindent\textbf{Proposition \theremarkcnt}
\if\relax\detokenize{#1}\relax
\else
(\textit{#1})
\fi}{\par\theoskip}

\newenvironment{assumptions}[1][]{\refstepcounter{remarkcnt}\par\theopreskip\noindent\textbf{Assumptions \theremarkcnt}
\if\relax\detokenize{#1}\relax
\else
(\textit{#1})
\fi}{\par\theoskip}

\newenvironment{remark}[1][]{\refstepcounter{remarkcnt}\par\theopreskip\noindent\textbf{Remark \theremarkcnt}
\if\relax\detokenize{#1}\relax
\else
(\textit{#1})
\fi}{\par\theoskip}

\newenvironment{definition}[1][]{\refstepcounter{remarkcnt}\par\theopreskip\noindent\textbf{Definition \theremarkcnt}
\if\relax\detokenize{#1}\relax
\else
(\textit{#1})
\fi}{\par\theoskip}

\newcommand{\Linftyptw}[0]{L^\infty_{\text{ptw}}}
\newcommand{\boundkerinteg}[0]{\mathcal{L}}
\newcommand{\boundkerdiscrinteg}[0]{\mathcal{L}^{\text{discr}}}

\begin{document}
\maketitle

\begin{abstract}
We show the convergence of a commonly used mesh-free discretization of a dynamic non-linear bond-based peridynamic model with irreversible damage implemented via a bond-breakage criterion for a fixed horizon $\delta >0$. The model chosen here is a mathematically well-posed regularization of a discontinuous peridynamic operator often found in the literature, which assumes that fracture results from strain in tension only. We explicitly include the one-point quadrature used in this scheme in the analysis. The convergence is with respect to the natural norm of the solution space given by $C^1([0,T],L^\infty(\Omega)^d)$, and therefore $L^\infty(\Omega)^d$ in the spatial coordinates. This result heavily constrains the spatial regularity of the continuous solution which we prove afterwards. More specifically, for locally $L$-Lipschitz continuous input data, local Lipschitz bounds are derived for the displacement, regardless of damage. This is consistent with interpreting the peridynamic model used here as an approximation to a truly discontinuous fracture model with a length-scale parameter given by the horizon $\delta>0$. We note that we do not deal with the limit of the solutions as $\delta \searrow 0$. 
\end{abstract}

\section{Introduction}
\label{sec:introduction}
Peridynamics is a relatively modern approach to modeling continuum mechanical problems that has gained significant traction in the mechanical and fracture mechanics communities, due to its inherent compatibility with discontinuities in the displacement and the ease with which fracture criteria can be implemented. Its formulation circumvents typical problems of the PDE formulation of continuum mechanics for discontinuous fracture problems, since it no longer requires spatial weak differentiability of the displacement. The essential difference to the classical theory is the replacement of the spatial differential operators by a spatially non-local integral operator.\par

First posed by Silling in \cite{SillRef}, adapted to fracture in \cite{SillMeshFree} and further developed in \cite{Silling2007}, peridynamics
has evolved into a family of related models \cite{CPDSOURCE,CPDFRAC,Partmann2023,lipton2015cohesive,lipton2020nonlocal,Lipton2019DynamicBF}, including correspondence formulations \cite{Silling2007, JAVILI2025104563,Javili2025}, which directly incorporate classical constitutive models (see e.g. \cite{Partmann2024} for a comparison of their wave-propagation behavior). These models have been applied extensively to fracture mechanical modeling and simulation, see e.g.\ \cite{Ha2010,KHOSRAVANI2022103899,Bobaru2015Cracks,Willberg2023}. A review of the validation of peridynamics against experimental results can be found in \cite{Diehl2019}.\par

The reformulation of the balance equations in terms of a non-local operator required a new mathematical approach to establishing well-posedness. For the linear case without damage, this has been carried out for both the bond-based and the state-based equations in \cite{Mengesha2014,Meng2014BB,OnWellPsdLinPD}, and it motivated the formulation of a non-local vector calculus that was subsequently used to analyze the equations further \cite{nonlocveccalclong}. For non-linear models without damage, existence and well-posedness results have been established for several model classes in \cite{Emmrich2013WellposednessOT,BellidoWellposed,EmmWeakMeasu,CocliteWellpsdNONLINPD}. Once damage is included, the available results become considerably sparser. Results for non-linear models with {reversible} damage include \cite{Lipton2019DynamicBF,lipton2015cohesive,lipton2020nonlocal}. In this paper we focus on a model with {irreversible} damage that was shown to be well-posed in \cite{WP_FRAC}. Other models with irreversible damage include \cite{bromley2026blendedapproachevolvingphase, Lipton2018}. The bond-based model studied in \cite{WP_FRAC} is derived from the prototypical bond-breakage model of \cite{SillMeshFree} by relaxing the damage factor and the force scalar to be Lipschitz continuous. Similar regularization techniques have been employed in \cite{Emmrich2016} to obtain well-posedness for a more regular class of pairwise force functions $\b{f}(\zeta, \eta)$ continuous in $\eta$, wherein further, the damage variable was an integral of a shifted bond-stretch over time instead of a maximum as used in \cite{WP_FRAC, SillMeshFree}. 
\par

On the numerical side, the popular mesh-free discretization stemming from \cite{SillMeshFree} has been investigated numerically in e.g.\ \cite{ConvStudiesInMeshFree,BobarFlorConv1DPD,ImprovOnePointQuadAlgs}; this list is by no means exhaustive. Rigorous convergence proofs, however, are available only for a comparatively narrow set of models. In \cite{GuanGunzNonlocStabAcc}, the stability and accuracy of a fully discrete approximation of a time-dependent, {one-dimensional} linear non-local wave-equation was analyzed. Piecewise-linear finite elements were used for the spatial discretization, and convergence was shown under stability conditions of the fully discrete Newmark scheme. For linear peridynamic bond-based models without damage, convergence has been established for general Galerkin-Ritz schemes in \cite{DuConv}, deriving rates for specific approximations. The analysis is restricted to simple domains $\Omega$ and further characterizes the regularity of the solutions relative to the regularity of the kernel, defining suitable function spaces using Fourier analysis. A related line of work concerns the joint limit of the discretization parameter $\kappa\searrow 0$ and the horizon $\delta\searrow 0$: \cite{TianDuAnalysisApprox} studies and compares several discretization methods for a linear one-dimensional equation without damage deriving criteria for when different schemes converge and when they do not, both to the continuous non-local solution in the discretization limit $\kappa\searrow 0$ and to the local solution in the joint limit $\kappa = \delta /r \searrow 0$ for some $r>1$. Further,  \cite{XiaoQiangAsymptCompatRobDiscr} develops a general framework for asymptotically compatible
discretizations of parametrized linear nonlocal problems, including state-based peridynamics without damage. For linear models without damage, \cite{CHEN20111237} compares continuous and discontinuous finite element approximations of the static one-dimensional linear model for both smooth and discontinuous displacements. Treatment of the mesh-free discretization of the static equations for a linear state-based model with only weak assumptions on the input data, allowing discontinuities, is given in \cite{pflug2026convergencediscretizedlinearstatic}.\par
For non-linear models, convergence results have been obtained in \cite{FinElemApproxNonLoc,liphabbcohesivenumerical,JhaNumericalCO} for the models of \cite{lipton2015cohesive,lipton2020nonlocal,Lipton2019DynamicBF}. These results come with convergence rates, but the underlying models do not include irreversible damage{, and \cite{FinElemApproxNonLoc} treats a finite element rather than a mesh-free discretization}. Conversely, irreversible damage was incorporated within the family of models given in \cite{lipton2015cohesive,lipton2020nonlocal,Lipton2019DynamicBF}
in \cite{COSKUN2025106302, bromley2026blendedapproachevolvingphase}, which has the further advantages of properly accounting for compression and of modeling additional effects such as hysteresis. A rigorous analysis of the corresponding mesh-free discretization is, however, to the best of our knowledge, still open. \par
Furthermore, with the exception of \cite{pflug2026convergencediscretizedlinearstatic}, none of the results above using a piecewise constant approximation address the one-point quadrature that is actually used in mesh-free implementations. \par

Overall, rigorous work on non-linear models with irreversible damage is still incomplete. The contribution of this paper is a rigorous convergence analysis for the model of \cite{WP_FRAC}, in two and three space dimensions, including a result about the regularity of the continuous solution relative to the regularity of the input data. The latter result is closely related to the numerical convergence. This model was chosen because of its simplicity as well as its proximity to the original formulation in \cite{SillMeshFree} which has inspired many practically used models. Building on the well-posedness theory of \cite{WP_FRAC}, we prove convergence of the mesh-free discretization, including the one-point quadrature used in practice, first for the semi-discrete scheme (Theorem \ref{theo:convsemidiscr}) and then for the fully discrete scheme obtained with a central difference time discretization (Theorem \ref{theo:fullconv}). Convergence is obtained in $C^1([0,T],L^\infty(\Omega)^d)$, the norm naturally inherited from the well-posedness theory. To the best of our knowledge, this has not been shown before for the non-linear model with irreversible damage of \cite{WP_FRAC}. In contrast to \cite{liphabbcohesivenumerical,JhaNumericalCO}, which derive rates in $L^2(\Omega)^d$ under the assumption of H\"older continuous solutions, our convergence result requires no regularity of the solution and consequently provides no rate in general. For suitable (e.g. Lipschitz-continuous) input data, we obtain a rate of $O(\kappa + 1/N)$.\par 

The convergence of a piecewise constant  approximation scheme in $C^0([0,T],L^\infty(\Omega)^d)$ raises the question of how regular the solution of the continuous problem actually is in the presence of significant damage. We address this directly in Proposition \ref{prop:contcontsol}, where we show that the displacement inherits continuity and local Lipschitz continuity from the input data, with local Lipschitz bounds that depend explicitly on $\delta>0$ and diverge as $\delta\searrow 0$. The result follows from the regularity of the underlying non-local operator and has the consequence that, at fixed $\delta>0$, damage cannot create discontinuities that are not already present in the input data. This implies that damage in zones where the input data is regular must manifest through a steep slope instead of a jump discontinuity. This suggests interpreting the peridynamic model at fixed $\delta$ as a regular approximation of a sharp free-fracture problem, an interpretation consistent with more recent peridynamic-type models \cite{COSKUN2025106302,bromley2026blendedapproachevolvingphase} that are viewed as blended, phase-field-like approaches in which the damage field evolves directly through the constitutive equations. We emphasize that this remains an interpretation: identifying a limit of the solutions as $\delta\searrow 0$ is open for the present model.\par

Throughout, we are interested exclusively in the limit $\kappa\searrow 0$ ($m$-convergence in the terms of \cite{BobarFlorConv1DPD}) of the discretization parameter at fixed horizon $\delta>0$. The limit $\delta\searrow 0$ and the joint limit with $\kappa \searrow 0$ are outside the scope of this paper. The limit $\delta\searrow 0$ is of interest in both the mechanical and the mathematical community and has been addressed in \cite{SillConv} for smooth motions. Numerical work has been done in \cite{BobarFlorConv1DPD}, results on the convergence of solutions in an $L^2(\Omega)^d$ setting in \cite{Mengesha2014,lipton2015cohesive}, and work on the $\Gamma$-convergence of the corresponding energies is found in \cite{Bellido2015,Bellido2020BondBasedPD}, where the latter shows that bond-based models do not converge to most  hyperelastic energies under the assumption of frame-indifference and isotropy.

\vspace{0.33cm}

The paper is organized as follows. In Section \ref{sec:ContBondBreakage} we introduce the non-linear model with damage studied here and summarize the well-posedness results of \cite{WP_FRAC}. In Section \ref{sec:SemiDiscrSystem} we apply the mesh-free discretization scheme of \cite{SillMeshFree}, including quadrature, and prove convergence of the semi-discrete scheme in Theorem \ref{theo:convsemidiscr}. Section \ref{sec:ConvFullyDiscr} describes the fully discrete system obtained with a central difference scheme and establishes its convergence in Theorem \ref{theo:fullconv}. For suitable input data (e.g. Lipschitz-continuous) we derive an asymptotic  rate of $O(\kappa + 1/N)$ where $\kappa$ is the spatial discretization parameter and $N$ the number of time steps. Section \ref{sec:PropContSol} is devoted to the regularity of solutions of the continuous problem and contains Proposition \ref{prop:contcontsol}. An example problem is chosen and numerically studied in Section \ref{sec:numerics}, visualizing the results of Sections \ref{sec:ConvFullyDiscr} and \ref{sec:PropContSol}. Section \ref{sec:HomogDirich} extends the analysis to Dirichlet boundary conditions. The proofs are nearly identical to those in the prior sections, and are therefore omitted.

\section{The Continuous Bond-Breakage Model}
\label{sec:ContBondBreakage}
The model studied in this paper was defined in \cite{WP_FRAC}, which proved the well-posedness of the equations on $C^2([0,T],L^\infty(\Omega)^d)$. The model is a Lipschitz-relaxation of the popular bond-breakage criterion model described in \cite{SillMeshFree}. Note that the original operator as defined in \cite{SillMeshFree} is not continuous, making the well-posedness theory and numerical treatment difficult without a suitable relaxation. \par
To state the underlying equations as given by (\ref{eqn:conteq}) below, we first define their constituents. \par 
In this paper we adopt the notational convention to boldface time dependent functions, e.g. $\b{u}\in C^0([0,T], L^\infty(\Omega)^d)$, while leaving functions without time dependency, e.g. $u\in L^\infty(\Omega)^d$, in roman type to distinguish where the non-locality in time will be relevant. Furthermore, we denote the finite dimensional $p$-norm as $\|\cdot\|_p$.

\begin{definition}[Definitions for the Continuous Equations]\label{def:constcont}
\begin{enumerate}
    \item Let $\Omega \subset \mathbb{R}^d$ be an open and bounded domain with ${d}\in\{2,3\}$.
    \item Define the bond stretch function $\ca{S}: \Map(\Omega, \mathbb{R}^d)\times \Omega\times\Omega \rightarrow \mathbb{R}$, defined by
    \begin{align*}
         \ca{S}({u}, x',x) = \frac{\|{u}(x')-{u}(x) + x'-x\|_2}{\|x'-x\|_2} -1.
    \end{align*}
    \item For a given displacement $u\in L^\infty(\Omega)^d$ at a time $t$ and a given $x\in\Omega$, we will call a bond $x'-x$ with $x'\in\Omega$ \textit{sub-critical} if $\ca{S}(u,x',x) < S^-_1$. Note that this includes the degenerate case $\|{u}(x')-{u}(x)+x'-x\|_2 = 0$, corresponding to the motion mapping the points $x',x\in \Omega$ onto the same image.
    \item Define the direction function for the motion $e: \Map(\Omega, \mathbb{R}^d)\times \Omega\times\Omega \rightarrow \mathbb{R}^d$, given by
    \begin{align*}
        e({u}, x',x) = \frac{{u}(x')-{u}(x)+x'-x}{\|{u}(x')-{u}(x)+x'-x\|_2}
    \end{align*}
    if the bond $x'-x$ at $x$ is not degenerate and $0$ otherwise.
    \item A horizon $\delta>0$ and a measurable weighting function $\omega_\delta: (0,\infty) \rightarrow \mathbb{R}_{\geq 0}$, fulfilling
    \begin{equation}
        \label{eqn:weightingfunctionintegcond}
        \supp(\omega_\delta(\|\cdot\|_2)) \subset B_\delta(0), \ \ \ca{L} :=\int_{\mathbb{R}^d} \frac{\omega_\delta(\|\zeta\|_2)}{\|\zeta\|_2}\dd \zeta<\infty.
    \end{equation}
    \item A terminal time $0 <T <\infty$.
    \item Constants $-1<S^-_1<S^-_0<0<S^+_0<S^+_1<\infty$, and the relaxed force scalar and damage function,
    \begin{equation}
    \label{eqn:rfsbbfdef}
        \scalarlaw(s) = \begin{cases}
            0 & s < S^-_1 \\
            S^-_0\tfrac{s-S^-_1}{S^-_0 - S^-_1} & S^-_1 \leq s\leq S^-_0 \\
            s & S^-_0 \leq s \leq S^+_0\\
            S^+_0\tfrac{S^+_1-s}{S^+_1 - S^+_0} & S^+_0 < s < S^+_1\\
            0 & S^+_1 \leq s
        \end{cases},\hspace{0.5cm} \mu(s) = \begin{cases}
            1 & s \leq S^+_0 \\
            \tfrac{S^+_1-s}{S^+_1 - S^+_0} & S^+_0 < s < S^+_1\\
            0 & s \geq S^+_1.
        \end{cases}
    \end{equation}
    Note that both these functions are Lipschitz continuous with constant \\
    $\ell := \max\Big(1, \tfrac{\max(1, S^+_0)}{S^+_1-S^+_0}, \tfrac{|S^-_0|}{S^-_0-S^-_1}\Big)>0$. We further define $\scalarlaw^+:= \max_{s\in \mathbb{R}} |\scalarlaw(s)| = \max (|S^+_0|, |S^-_0|)$
    \item The force density $\b{f}: \Omega \times \Omega \times C^0([-T,0], L^\infty(\Omega)^d) \rightarrow \mathbb{R}^d$ defined for almost all $x$ and almost all $x'$, by
    \begin{equation}\label{eqn:defcontfinteg}
        \b{f}(x', x, \b{u}) = \omega_\delta(\|x'-x\|_2)\scalarlaw(\ca{S}(\b{u}(0), x', x))\mu\Big(\esssup_{-T\leq s\leq 0}\ca{S}([\Gamma\b{u}](s), x', x) \Big)e(\b{u}(0), x',x),
    \end{equation}
    where the essential supremum over $s\in [-T, 0]$ is evaluated over a jointly measurable representative $\Gamma\b{u}\in L^\infty([-T,0]\times \Omega)^d$ of $\b{u} \in C^0([-T,0], L^\infty(\Omega)^d)$ (see Remark \ref{rm:jointmesrepr}). 
    \item The above definitions then culminate in the non-local peridynamic operator\\ $F: C^0([-T, 0], L^\infty(\Omega)^d) \rightarrow L^\infty(\Omega)^d$ given by 
    \begin{equation}\label{eqn:defcontF}
        F(\b{u})(x) = \int_{\Omega} \b{f}(x',x, \b{u}) \dd x'.
    \end{equation}
    Note that by construction, $\scalarlaw(\ca{S}(u,x',x))=0$ for sub-critical bonds $x'-x$. Most specifically, $F(\b{u})$ does not depend on the value of $e$ in degenerate bonds. We will make use of this degree of freedom in many proofs later on.
\end{enumerate}
\end{definition}

\begin{remark}[Jointly Measurable Representatives]\label{rm:jointmesrepr}
    For every $\b{u} \in C^0([-T,0], L^\infty(\Omega)^d)$ we construct a $\Gamma \b{u}\in L^\infty([-T,0]\times\Omega)^d$ such that $[\Gamma \b{u}](t, \cdot) = \b{u}(t)$ a.e. $t\in [-T,0]$ in $L^\infty(\Omega)^d$. This is done such that the following diagram commutes, where unlabeled arrows represent canonical inclusions,
\begin{equation*}
    \xymatrix{
        C^0([-T,0], L^\infty(\Omega)^d) \ar@{_{(}->}[d]\ar@{_{(}->}[r]^{\Gamma} & L^\infty([-T,0] \times \Omega)^d  \ar@{_{(}->}[dd]\\
        C^0([-T,0], L^2(\Omega)^d) \ar@{_{(}->}[d] \\
        L^2([-T,0], L^2(\Omega)^d) \ar@{_{(}->>}[r]^{\Phi} & L^2([-T,0]\times \Omega)^d.
    }
\end{equation*}
The map $\Phi :L^2([-T,0], L^2(\Omega)^d) \simeq L^2([-T,0]\times \Omega)^d$ is constructed in \cite{anainbanachsp}, Proposition 1.2.24, and fulfills $[\Gamma \b{u}](t,\cdot) = \b{u}(t)$ for almost all $t\in [-T,0]$ in $L^2(\Omega)^d$ for any $\b{u}\in L^2([-T,0], L^2(\Omega)^d)$.\par
This in particular implies for a given $\b{u}\in C^0([-T,0], L^\infty(\Omega)^d)$, 
\begin{align*}
    \|\Gamma \b{u}\|_{L^\infty([-T,0]\times \Omega)^d}& = \esssup_{t,x}  \|[\Gamma \b{u}](t,x)\|_\infty\\
    &= \esssup_{x} \esssup_{t} \|[\Gamma \b{u}](t,x)\|_\infty  \\
    &= \esssup_{t} \esssup_{x} \|[\Gamma \b{u}](t,x)\|_\infty  \\
    &= \esssup_{t} \|[\Gamma \b{u}](t,\cdot)\|_{L^\infty(\Omega)^d } \\
    &= \esssup_{t} \|\b{u}(t)\|_{L^\infty(\Omega)^d} = \|\b{u}\|_{C^0([-T,0], L^\infty(\Omega)^d)}.  
\end{align*}
\end{remark}

To state the corresponding equations, we further introduce the notation $\b{a}_t$ for a $t \in [0,T]$ and a function $\b{a}:[0,T] \rightarrow X$ for any Banach space $X$, by defining
\begin{equation*}
    \b{a}_t : [-T, 0] \rightarrow X,\ s\mapsto \begin{cases}
        \b{a}(0) & s \leq -t \\
        \b{a}(s+t) & \text{else}.
    \end{cases}
\end{equation*}
The peridynamic problems studied in this paper are then given for $\b{u} \in C^2([0,T], L^\infty(\Omega)^d)$, by 
\begin{equation}
    \label{eqn:conteq}
    \begin{rcases}
        \rho\ddot{\b{u}}(t) = F(\b{u}_t) + \b{b}(t)\\ 
        \b{u}(0) = u_0 \\
        \dot{\b{u}}(0) = v_0
    \end{rcases}
\end{equation}
for the initial values $u_0, v_0 \in L^\infty(\Omega)^d$. The external body forces $\b{b}$ are assumed to be in $C^0([0,T], L^\infty(\Omega)^d)$. For simplicity, we also assume that the density $\rho$ is constant, without loss of generality, we assume $\rho \equiv 1$ everywhere.

\begin{remark}[About the Relaxation]\label{rm:relaxation} The functions $\scalarlaw$ and $\mu$ as defined in (\ref{eqn:rfsbbfdef}) are Lipschitz relaxations introduced in \cite{WP_FRAC} of the functions 
\begin{equation*}
    \scalarlaw(x) \equiv x \chi_{[-1, S^+_1]}(x), \ \ \mu(x) = \chi_{(-\infty, S^+_1)}(x)
\end{equation*}
as used in \cite{SillMeshFree}. The force scalar $\scalarlaw$ was introduced in \cite{WP_FRAC} to control the interaction of particles $x,x'$ with bond-stretch near $S_1^-$. More specifically, the cutoff of the force scalar $\scalarlaw$ at $S^-_1>-1$ is essential to the well-posedness proof given in \cite{WP_FRAC} as well as to the results stated in this paper. In total, the relaxation changes the original model in \cite{SillMeshFree} in two ways.
\begin{enumerate}
    \item The relaxation of $\mu$ and $\scalarlaw$ at $S_1^+$, $S_0^+$: This can also be interpreted mechanically as allowing the bond to weaken first before breaking, as noted in \cite{WP_FRAC, Emmrich2016}. \newline
    {The concept of allowing a continuous weakening before breaking is not new and has been noted in \cite{NIAZI2021102855} to be desirable for controlling the influence of the horizon on the strength and fracture energy of models without pre-cracks as given in this paper.}
    \item The relaxation of $\scalarlaw$ at $S_1^-$, $S_0^-$: This cutoff introduces physically undesirable behavior, weakening bond interactions if bonds are compressed beyond a stretch of $S^-_0$, making them reversibly fail below $S^-_1$. Note that by definition, $\ca{S}(u,x',x) \geq -1$ for all $x,x'\in\Omega$ and displacements $u$ and that $S^-_1>-1$ can be chosen arbitrarily close to $-1$. \newline
The fact that such a cutoff has to be introduced explicitly in the first place is a consequence of the physically undesirable behavior of the model under compression. The strain is capped at $-1$, limiting the material's ability to resist compression. Furthermore, the equations are such that no damage can occur from it. In practice this is mitigated either by penalization terms in the form of short-range forces counteracting inter-penetration \cite{SillMeshFree, MACEK20071169, PARKS2008777} acting on particles whose distance fall beneath a given threshold, or by reformulations of the bond strain itself \cite{TupekRadovitzky2014} solving the issue on the level of the continuous equations. Both remedies are constitutive rather than analytical in nature, and a consolidated treatment within the present class of models is, to our knowledge, not available. Therefore, we regard the question as outside the scope of this paper. \newline
We note that the necessity of such a modification is already visible from a mathematical view in the well-posedness theory of the model used here in \cite{WP_FRAC}. The cutoff of $\scalarlaw$ below $S^-_1$ is not merely a modeling convenience, rather, since the deformed bond direction degenerates as $\ca{S} \searrow -1$, with derivative of order $\|x'-x+u(x')-u(x)\|^{-1}$, a force scalar that does not vanish in this limit causes the proofs for the peridynamic operator to be Lipschitz continuous in the displacement $\b{u}$ as given here to fail, as the change of the direction of the bond-forces $e$ no longer linearly correlates with the norm of the displacement. Following this interpretation, the mathematical analysis of the bond-breakage model employed here senses that there is an issue with the constitutive behavior in compression from a purely formal analysis.
    
\end{enumerate}
We note that we are not investigating the limit as $S^-_0,S^-_1 \searrow -1$ or $S^+_0 \rightarrow S^+_1$ and instead view the relaxed model not as an approximation but as an alternative to the steep cutoff model used in \cite{SillMeshFree}. Again, use of such models with a continuous cutoff in $\mu$ have been found to be useful in practice by themselves, e.g.  \cite{NIAZI2021102855}. 
\end{remark}

We summarize the well-posedness of the problem (\ref{eqn:conteq}) in the following theorem taken from \cite{WP_FRAC}, Theorem 3. 

\begin{theorem}[Well-Posedness of the Continuous Equations]
    Let $u_0, v_0 \in L^\infty(\Omega)^d$, $\rho>0$, $\b{b} \in C^0([0,T], L^\infty(\Omega)^d)$ and $\omega_\delta: (0,\infty) \rightarrow \mathbb{R}_{\geq 0}$ be a weighting function fulfilling (\ref{eqn:weightingfunctionintegcond}) for some $\delta >0$. Then, the problem as posed in (\ref{eqn:conteq}) has a unique solution $\b{u} \in C^2([0,T], L^\infty(\Omega)^d)$.
\end{theorem}

\section{The Semi-Discrete System}
\label{sec:SemiDiscrSystem}
This section will describe the spatial discretization of the continuous model described in Section \ref{sec:ContBondBreakage}. We use a mesh-free approach including one-point quadrature for the semi-discrete system, realized by the following discrete analogs of the terms in Definition \ref{def:constcont}. 

\begin{definition}[Constituents of the Discrete Equation]\label{def:constdisc}
\begin{enumerate}
    \item We construct the points with respect to which we want to formulate the mesh-free discrete problems as well as perform one-point quadrature.
    \begin{enumerate}
        \item We subdivide $\mathbb{R}^d$ for $d\in\{2,3\}$ into cubes/squares $(B_i)_{i\in \mathbb{Z}^d}$ of equal side length $\kappa>0$, with centers $x_i = \kappa i + \vec{o}_\kappa$ with $i \in \mathbb{Z}^d$ and some $\vec{o}_\kappa \in \mathbb{R}^d$. 
        \item We define $\Omega_i := \Omega \cap B_i$ for $i\in \mathbb{Z}^d$ if $\vol(\Omega \cap B_i)> 0$, and define the resulting index set of all such $i$ as $I^\kappa$.
        \item For notational simplicity we define $V_i := \vol(\Omega_i)$.
    \end{enumerate}
    \item The discrete bond stretch function for $i\not=j$, given by $\ca{S}^{ij}: (\mathbb{R}^d)^{I_\kappa}\rightarrow \mathbb{R}$ with $\ca{S}^{ij}(\vec{u}^\kappa) = \tfrac{\|\vec{u}^\kappa_j-\vec{u}^\kappa_i+x_j-x_i\|_2}{\|x_j-x_i\|_2}-1$.
    \item The discrete direction function $e^{ij}: (\mathbb{R}^d)^{I_\kappa}\rightarrow \mathbb{R}^d$ with $e^{ij}(\vec{u}^\kappa) = \tfrac{\vec{u}^\kappa_j-\vec{u}^\kappa_i+x_j-x_i}{\|\vec{u}^\kappa_j-\vec{u}^\kappa_i+x_j-x_i\|_2}$ if $\vec{u}_j^\kappa -\vec{u}_i^\kappa -  x_i + x_j \not=0$ and $0$ otherwise. Again, we call a bond $\zeta^{ij} = x_j-x_i$ at $x_i$ for a given discrete displacement $\vec{u}^\kappa$ and base point $x_i$ sub-critical if $\ca{S}^{ij}(\vec{u}^\kappa)< S^-_1$ and degenerate if $\ca{S}^{ij}(\vec{u}^\kappa) = -1$. 
    \item We introduce the function space $L^{\infty}_{\mathrm{ptw}}(A, B)$ for measurable Banach spaces $A$ and $B$, by defining 
\begin{equation*}
    L^{\infty}_{\mathrm{ptw}}(A, B) := \{ f: A \rightarrow B : f \text{ measurable}, \|f\|_{\text{sup}} := \sup_{x\in A}\|f(x)\|_B < \infty\},
\end{equation*}
with the $\|\cdot\|_{\text{sup}}$ norm.  This norm turns the space $L^\infty_{\mathrm{ptw}}$ into a Banach space.\par

    \item The discrete bond density function $\b{f}^{ij}: \Linftyptw([-T,0], (\mathbb{R}^d)^{I^\kappa})\rightarrow \mathbb{R}^d$ and the corresponding discrete non-local operator $\vec{F}_\kappa: \Linftyptw([-T, 0], (\mathbb{R}^d)^{I_\kappa}) \rightarrow (\mathbb{R}^d)^{I_\kappa}$
    \begin{equation*}
        (\vec{F}_\kappa(\vec{\b{u}}^\kappa))_i = \sum_{i\not=j} V_j\underbrace{\omega_\delta(\|x_j-x_i\|_2)\scalarlaw(\ca{S}^{ij}(\vec{\b{u}}^\kappa(0)))\mu\Big(\sup_{-T\leq s\leq 0}\ca{S}^{ij}(\vec{\b{u}}^\kappa(s))\Big)e^{ij}(\vec{\b{u}}^\kappa(0))}_{=:\vec{\b{f}}^{ij}(\vec{\b{u}}^\kappa)}.
    \end{equation*}
    Again, the value of $\vec{F}_\kappa(\vec{\b{u}}^\kappa)$ does not depend on the value of $e^{ij}$ for sub-critical bonds including degenerate bonds due to $\scalarlaw(\ca{S}^{ij}(\vec{\b{u}}^\kappa(0))) = 0$. We will use this degree of freedom in the proofs below to eliminate the problem of degenerate and sub-critical bonds appearing in the discrete formulation.
    \end{enumerate}
\end{definition}
Combining the above yields the semi-discrete problem, given by
\begin{equation}
    \label{eqn:semdiscrprobl}
    \begin{rcases}
        \tfrac{d^2}{dt^2}\vec{\b{u}}^\kappa(t) = \vec{F}_\kappa((\vec{\b{u}}^\kappa)_t) + \vec{\b{b}}_\kappa(t) \\
        \ \ \vec{\b{u}}^\kappa(0) = \vec{u}_0^\kappa\\
        \tfrac{d}{dt}\vec{\b{u}}^\kappa(0) = \vec{v}_0^\kappa
    \end{rcases}
\end{equation}
with the discrete initial values $\vec{u}_0^\kappa,  \vec{v}_0^\kappa \in (\mathbb{R}^d)^{I_\kappa}$ and body force $\vec{\b{b}}_\kappa \in C^0([0,T], (\mathbb{R}^d)^{I_\kappa})$. \par\medskip
The structure of the remainder of this section is as follows. We will first prove the well-posedness of the discrete systems by following the same strategy as in the continuous case. Then, we will prove the convergence of the semi-discrete solutions $\vec{\b{u}}^\kappa$ to the continuous solution in Theorem \ref{theo:convsemidiscr}. \par \medskip
For the following analysis, we require the operators representing finite dimensional discrete data in the continuum and vice versa.
\begin{definition}[Piecewise Constant Embedding and Projection Operators]
\begin{align*}
    &\Kk: (\mathbb{R}^d)^{I^\kappa} \rightarrow L^\infty(\Omega)^d, \ \ \vec{u} \mapsto (x\mapsto \sum_{i\in I^\kappa} \vec{u}_i \chi_{\Omega_i}(x) ), \\
    &\Pk:  L^\infty(\Omega)^d \rightarrow (\mathbb{R}^d)^{I^\kappa}, {u}\mapsto \Big(\frac{1}{V_i} \int_{\Omega_i} {u}(x) \dd x\Big)_{i\in I^\kappa}
\end{align*}
\end{definition}

To obtain the convergence results, we need the assumptions below.

\begin{assumptions}[Assumptions for the Convergence]\label{ass:numconv}
    \begin{enumerate}
        \item $\omega_\delta: (0,\infty) \rightarrow \mathbb{R}_{\geq 0}$ measurable, non-increasing, vanishing for arguments greater than $\delta>0$, as well as the integrability condition,
        \begin{equation}\label{eqn:boundcond}
            \int_{B_\delta(0)} \frac{\omega_\delta(\|\zeta\|_2)}{\|\zeta\|_2}\dd \zeta=:\boundkerinteg < \infty.
        \end{equation}
        \item $\rho>0$ constant.
        \item For the chosen sequence of $\kappa_n \searrow 0 $, we require
        \begin{enumerate}
            \item $\|\Kkf{n} \vec{u}_0^{\kappa_n} - u_0\|_{L^\infty(\Omega)^d} \xrightarrow{n \rightarrow \infty} 0$, $\|\ca{K}_{\kappa_n} \vec{v}_0^{\kappa_n} - v_0\|_{L^\infty(\Omega)^d} \xrightarrow{n \rightarrow \infty} 0$.
            \item The discrete external body forces $\vec{\b{b}}_{\kappa_n}\in C^0([0,T], (\mathbb{R}^d)^{I^\kappa})$, have to fulfill 
        \begin{equation*}
            \|\Kkf{n}\vec{\b{b}}_{\kappa_n} - \b{b}\|_{C^0([0,T], L^\infty(\Omega)^d)} \xrightarrow[]{\kappa \searrow 0}0.
        \end{equation*}
        \end{enumerate}
    \end{enumerate}
\end{assumptions}

\begin{remark}[About the Assumptions, Part I]
    The $L^\infty(\Omega)^d$ norm inherited from the solution space, in which the convergence is therefore naturally stated in, causes the Assumptions \ref{ass:numconv} to be rather restrictive. Material parameters need to be sufficiently regular and the convergence of the numerical input data is required in the $C^0([0,T], L^\infty(\Omega)^d)$ or $L^\infty(\Omega)^d$ norm, respectively.\par
    In particular, if the input data is not continuous, the discontinuities have to lie within the closure of all boundaries of the discrete elements $\Omega_i$ for all $\kappa_n$ in the chosen sequence $\kappa_n\searrow 0$, imposing significant geometrical restriction on the discontinuity set. The offset $\vec{o}_\kappa$ as introduced in Definition \ref{def:constdisc} can be used to mitigate this issue for simple discontinuity sets in the input data, as we will do in the numerical example shown in Section \ref{sect:numex}. This is nonetheless undesirable, and alternatives trying to evade the $L^\infty(\Omega)^d$ norm altogether are discussed in Remark \ref{rm:abouttheassumptionspt2}.\par
    Continuous input data however does not pose a problem for most choices of deriving the numerical input data $\vec{u}_0, \vec{v}_0, \vec{\b{b}}_{\kappa}$ from the continuous $u_0, v_0, \b{b}$ such as choosing the value at the index $i \in I^\kappa$ to be the value of the continuous input data at $x_i$, i.e. for example $(\vec{u}_0)_i = u_0(x_i)$, or the projection of the continuous input data under $\Pk$. Lipschitz continuity of the input data would then result in a linear convergence of the discrete input data to the continuous one. \par
    This approach for choosing the numeric input data however requires careful handling of elements $\Omega_i$ with $x_i \not\in \Omega$. The solution for this is to continuously extend the input data to $\overline{B_{r}(\Omega)} := \{x \in \mathbb{R}^d: \dist(x, \Omega)\leq r\}$ for a sufficiently chosen $r > \sqrt{d}\kappa/2$.
    \end{remark}

Again, in the following, we will assume without loss of generality, that $\rho \equiv 1$, through rescaling of $\omega_\delta$ and the external body forces.\par
Beginning with the well-posedness proof, the main step will be to show the Lipschitz continuity of the discrete non-local operator $\vec{F}_\kappa$. From now on, let $(\mathbb{R}^d)^{I^\kappa}$ be endowed with the $\|\cdot\|_\infty$ norm, unless specified otherwise.\par\medskip

To this end, we need the following lemma.
\begin{lemma}[Uniform Bounds on Discrete Integrals]\label{lemm:unifbndinteg} For all $i\in I^\kappa$ and $x\in \Omega_i$,
\begin{align*}
        0\leq \sum_{j\not=i} \int_{\Omega_j}\frac{\omega_\delta(\|x_j-x_i\|_2)}{\|x'-x\|_2} \dd x' \leq \boundkerdiscrinteg,\\
        0\leq \sum_{j\not=i} V_j \frac{\omega_\delta(\|x_j-x_i\|_2)}{\|x_j-x_i\|_2}\leq \boundkerdiscrinteg,
\end{align*}
for a constant $\boundkerdiscrinteg>0$ for all $\delta/\sqrt{d}> \kappa>0$.
\end{lemma}

\begin{proof}
    This follows from 
    \begin{align}\label{eqn:bndnorm}
        \|x'-x\|_2 &\leq 2\tfrac{\sqrt{d}}{2}\kappa + \|x_i - x_j\|_2 \\
        &\leq (\sqrt{d} +1) \|x_i-x_j\|_2 \nonumber.
    \end{align}
    for all $x\in \Omega_i$, $x' \in \Omega_j$ with $i\not=j$, using the change of variables formula and Assumptions \ref{ass:numconv}, in particular, the monotonicity of $\omega_\delta$ and (\ref{eqn:boundcond}),
    \begin{align*}
        0&\leq \sum_{j\not=i} V_j \frac{\omega_\delta(\|x_j-x_i\|_2)}{\|x_j-x_i\|_2} \leq (\sqrt{d}+1)\sum_{j\not=i}\int_{\Omega_j} \frac{\omega_\delta(\|x_j-x_i\|_2)}{\|x'-x\|_2} \dd x' \\
        &\leq (\sqrt{d}+1)^d\int_{B_{(\sqrt{d}+1)\delta}(x)} \frac{\omega_\delta(\tfrac{1}{\sqrt{d}+1}\|x'-x\|_2)}{\frac{1}{\sqrt{d}+1}\|x'-x\|_2} \frac{1}{(\sqrt{d}+1)^d}\dd x'  \leq (\sqrt{d}+1)^d\ca{L}
    \end{align*}
\end{proof}

\begin{proposition}[Lipschitz Continuity of the Discrete Non-Local Operator]\label{prop:LScontdiscrrhs} \newline
    Let $\vec{\b{u}},\vec{\b{v}}\in \Linftyptw([-T,0],(\mathbb{R}^d)^{I^\kappa})$, then
    \begin{equation*}
        \|\vec{F}_\kappa(\vec{\b{u}}) - \vec{F}_\kappa(\vec{\b{v}})\|_{\infty} \leq L_{\vec{F}}\|\vec{\b{u}}-\vec{\b{v}}\|_{\Linftyptw([-T,0], (\mathbb{R}^d)^{I^\kappa})} 
    \end{equation*}
    with $L_{\vec{F}}>0$ independent of $\delta/\sqrt{d}> \kappa > 0$.
\end{proposition}

\begin{proof} This result follows similarly to how the Lipschitz continuity was obtained for the continuous operator $F$ in \cite{WP_FRAC}, Lemma 1. Since $\scalarlaw,\mu$ were not changed by the discretization, their Lipschitz constants are also unchanged. \par
Let $i\in I^\kappa$, then
\begin{align*}
    (\vec{F}_\kappa(\vec{\b{u}}) - \vec{F}_\kappa(\vec{\b{v}}))_i &=         \sum_{i\not=j} V_j{\omega_\delta(\|x_j-x_i\|_2)\scalarlaw(\ca{S}^{ij}(\vec{\b{u}} (0)))\mu\Big(\sup_{-T\leq s\leq 0}\ca{S}^{ij}(\vec{\b{u}} (s))\Big)e^{ij}(\vec{\b{u}} (0))} \\
   &\quad -\sum_{i\not=j} V_j{\omega_\delta(\|x_j-x_i\|_2)\scalarlaw(\ca{S}^{ij}(\vec{\b{v}} (0)))\mu\Big(\sup_{-T\leq s\leq 0}\ca{S}^{ij}(\vec{\b{v}} (s))\Big)e^{ij}(\vec{\b{v}} (0))}.
\end{align*}
To deal with these terms, we first find
\begin{equation}
\label{eqn:bndonsemdiscrS}    |\ca{S}^{ij}(\vec{\b{u}} (t))-\ca{S}^{ij}(\vec{\b{v}} (t))| \leq\tfrac{\sqrt{d}}{\|x_j-x_i\|_2}\Big(\|\vec{\b{u}} _i(t) - \vec{\b{v}} _i(t)\|_\infty + \|\vec{\b{u}} _j(t) - \vec{\b{v}} _j(t)\|_\infty\Big),
\end{equation}
as well as, assuming non-degenerate bonds for both $\b{u}(t)$ and $\b{v}(t)$,
\begin{equation}\label{eqn:bndonsemdiscre}
\begin{aligned}
    &\|e^{ij}(\vec{\b{u}}(t))-e^{ij}(\vec{\b{v}}(t))\|_\infty\\
    &\leq \frac{1}{\|\vec{\b{u}}_j(t)-\vec{\b{u}}_i(t)+x_j-x_i\|_2}(\|\vec{\b{u}}_j(t)-\vec{\b{u}}_i(t)- \vec{\b{v}}_j(t)+\vec{\b{v}}_i(t)\|_\infty)\\
    &\quad+\|\vec{\b{v}}_j(t)-\vec{\b{v}}_i(t)+x_j-x_i\|_\infty\frac{\|\vec{\b{u}}_j(t)-\vec{\b{u}}_i(t)- \vec{\b{v}}_j(t)+\vec{\b{v}}_i(t)\|_2}{\|\vec{\b{u}}_j(t)-\vec{\b{u}}_i(t)+x_j-x_i\|_2 \|\vec{\b{v}}_j(t)-\vec{\b{v}}_i(t)+x_j-x_i\|_2}\\
    &\leq \frac{2\sqrt{d}}{\|\vec{\b{u}}_j(t)-\vec{\b{u}}_i(t)+x_j-x_i\|_2} \Big(\|\vec{\b{u}}_j(t)- \vec{\b{v}}_j(t)\|_\infty+\|\vec{\b{v}}_i(t)-\vec{\b{u}}_i(t)\|_\infty\Big) \\
    &= \frac{2\sqrt{d}}{(\ca{S}^{ij}(\vec{\b{u}}(t))+1)\|x_j-x_i\|_2} \Big(\|\vec{\b{u}}_j(t)- \vec{\b{v}}_j(t)\|_\infty+\|\vec{\b{v}}_i(t)-\vec{\b{u}}_i(t)\|_\infty\Big).
\end{aligned}
\end{equation}
Note that for any $t\in [-T, 0]$ and $i\in I^\kappa$,
\begin{equation*}
    \|\vec{\b{u}}_j(t)- \vec{\b{v}}_j(t)\|_\infty \leq \|\vec{\b{u}}(t)- \vec{\b{v}}(t)\|_\infty \leq \|\vec{\b{u}}(t)- \vec{\b{v}}(t)\|_{\Linftyptw([-T,0], (\mathbb{R}^d)^{I^\kappa})}.
\end{equation*}
We now need to deal with sub-critical bonds in $\vec{\b{v}}$, which we do by defining $e^{ij}_{\vec{b}}: (\mathbb{R}^d)^{I_\kappa}\rightarrow \mathbb{R}^d$ by
\begin{align*}
    e^{ij}_{\vec{b}}(\vec{a}) = \begin{cases}e^{ij}(\vec{a}) & \text{if } \ca{S}^{ij}(\vec{a})\geq S^-_1\\
    e^{ij}(\vec{b}) & \text{else,}
    \end{cases}
\end{align*}
for $\vec{b}, \vec{a} \in (\mathbb{R}^d)^{I^\kappa}$.
With this, we define
\begin{equation*}
    (\vec{F}_{\kappa,\vec{\b{u}}^\kappa}(\vec{\b{a}}^\kappa))_i = \sum_{i\not=j} V_j{\omega_\delta(\|x_j-x_i\|_2)\scalarlaw(\ca{S}^{ij}(\vec{\b{a}}^\kappa(0)))\mu\Big(\sup_{-T\leq s\leq 0}\ca{S}^{ij}(\vec{\b{a}}^\kappa(s))\Big)e^{ij}_{\vec{\b{u}}^\kappa(0)}(\vec{\b{a}}^\kappa(0))}.
\end{equation*}
We then obtain $\vec{F}_\kappa(\vec{\b{a}}^\kappa) = \vec{F}_{\kappa, \vec{\b{u}}^\kappa}(\vec{\b{a}}^\kappa)$ for all $\vec{\b{a}}^\kappa, \vec{\b{u}}^\kappa \in C^0([-T,0], (\mathbb{R}^d)^{I^\kappa})$.\par
This then leads to, using $\scalarlaw^+:= \max(|S^-_0|,|S^+_0|)= \sup_{s\in \mathbb{R}} |\scalarlaw(s)|$,
\begin{align*}
        &\|(\vec{F}_\kappa(\vec{\b{u}}) - \vec{F}_\kappa(\vec{\b{v}}))_i\|_\infty=\|(\vec{F}_\kappa(\vec{\b{u}}) - \vec{F}_{\kappa, \vec{\b{u}}^\kappa}(\vec{\b{v}}))_i\|_\infty \\
        &\leq \sum_{i\not=j} V_j\omega_\delta(\|x_j-x_i\|_2)|\scalarlaw(\ca{S}^{ij}(\vec{\b{u}} (0)))-\scalarlaw(\ca{S}^{ij}(\vec{\b{v}} (0)))|\\
        &\quad+\scalarlaw^+\sum_{i\not=j} V_j\omega_\delta(\|x_j-x_i\|_2)\Big|\mu\Big(\sup_{-T\leq s\leq 0}\ca{S}^{ij}(\vec{\b{v}} (s))\Big)-\mu\Big(\sup_{-T\leq s\leq 0}\ca{S}^{ij}(\vec{\b{u}} (s))\Big)\Big|\\
        &\quad+\sum_{i\not=j} V_j\omega_\delta(\|x_j-x_i\|_2)|\scalarlaw(\ca{S}^{ij}(\vec{\b{u}} (0)))|\|e^{ij}(\vec{\b{u}} (0))-e^{ij}_{\vec{\b{u}}^\kappa(0)}(\vec{\b{v}} (0))\|_\infty \\
        \intertext{Note that $|\scalarlaw(\ca{S}^{ij}(\vec{\b{u}} (0)))|\|e^{ij}(\vec{\b{u}} (0))-e^{ij}_{\vec{\b{u}}^\kappa(0)}(\vec{\b{v}} (0))\|_\infty$ is $0$ both if $\ca{S}^{ij}(\vec{\b{v}}(0))\in [-1, S^-_1)$ or $\ca{S}^{ij}(\vec{\b{u}}(0))\in [-1, S^-_1)$, leading us to globally apply the bound (\ref{eqn:bndonsemdiscre}) with $\ca{S}^{ij}(\vec{\b{u}}(0))\geq S^{-}_1$,}
        &\leq \sum_{i\not=j} 2\ell\sqrt{d}V_j\tfrac{\omega_\delta(\|x_j-x_i\|_2)}{\|x_j-x_i\|_2}  \|(\vec{\b{u}}  - \vec{\b{v}} )(0)\|_\infty \\
        &\quad+ \sum_{i\not=j} 2\ell \scalarlaw^+\sqrt{d}V_j\tfrac{\omega_\delta(\|x_j-x_i\|_2)}{\|x_j-x_i\|_2} \Bigg(\sup_{-T\leq s\leq 0}\|(\vec{\b{u}}  - \vec{\b{v}} )(s)\|_\infty\Bigg) \\
        &\quad+ \sum_{i\not=j} 4 \scalarlaw^+\sqrt{d}V_j\tfrac{\omega_\delta(\|x_j-x_i\|_2)}{(S^-_1+1)\|x_j-x_i\|_2}  \|(\vec{\b{u}}  - \vec{\b{v}} )(0)\|_\infty \\
        &\leq 2\sqrt{d}(\ell + \ell \scalarlaw^+ + \tfrac{2\scalarlaw^+}{S^-_1+1})\boundkerdiscrinteg\|\vec{\b{u}}  - \vec{\b{v}} \|_{\Linftyptw([-T,0], (\mathbb{R}^d)^{I^\kappa})},
\end{align*}
where $\boundkerdiscrinteg>0$ comes from Lemma \ref{lemm:unifbndinteg}.
\end{proof}

We now reduce the order of the system as follows.
Define 
\begin{equation*}
    \vec{G}_\kappa: C^0([-T,0],(\mathbb{R}^d)^{I^\kappa})\times C^0([-T,0],(\mathbb{R}^d)^{I^\kappa}) \times [0,T] \rightarrow (\mathbb{R}^d)^{I^\kappa} \times (\mathbb{R}^d)^{I^\kappa},
\end{equation*}
by, for $\phi,\psi \in C^0([-T,0],(\mathbb{R}^d)^{I^\kappa})$, 
\begin{equation*}
    \vec{G}_\kappa((\vec{\phi}, \vec{\psi}), t) = \begin{pmatrix}
    \vec{\psi}(0) \\
    \vec{F}_\kappa(\vec{\phi}) + \vec{\b{b}}_\kappa(t)
    \end{pmatrix}.
\end{equation*}
We then have
\begin{equation}
\label{eqn:LSG}
    \|\vec{G}_\kappa(\vec{\Phi},t)-\vec{G}_\kappa(\vec{\Psi},t)\|_\infty\leq L_{\vec{G}} \|\vec{\Phi} -\vec{\Psi}\|_{C^0([-T,0],(\mathbb{R}^d)^{I^\kappa}\times (\mathbb{R}^d)^{I^\kappa})},
\end{equation}
with $L_{\vec{G}}=\max(1, L_{\vec{F}})$. Further, from the inequality \begin{align*}|\|\vec{G}_\kappa(\vec{\Phi},t)\|_{\infty}-\|\vec{G}_\kappa(0,t)\|_{\infty}| \leq \|\vec{G}_\kappa(\vec{\Phi},t)-\vec{G}_\kappa(0,t)\|_{\infty},
\end{align*}
one obtains
\begin{equation*}
    \|\vec{G}_\kappa(\vec{\Phi},t)\|_\infty \leq L_{\vec{G}}\|\vec{\Phi}\|_{C^0([-T,0],(\mathbb{R}^d)^{I^\kappa}\times (\mathbb{R}^d)^{I^\kappa})} + \|\vec{\b{b}}_\kappa(t)\|_\infty.
\end{equation*}
The reduced system to (\ref{eqn:semdiscrprobl}) is then
\begin{equation}
\label{eqn:semdiscreqred}
    \begin{rcases}
        \tfrac{d}{dt}\vec{\b{u}}_{Red}^\kappa(t) = \vec{G}_\kappa((\vec{\b{u}}_{Red}^\kappa)_t, t) \\
        \vec{\b{u}}_{Red}^\kappa(0) = \vec{u}_{0, Red}^\kappa
    \end{rcases} t \in [0,T].
\end{equation}
with $\vec{u}_{0, Red}^\kappa := (\vec{u}_0^\kappa, \vec{v}_0^\kappa)^\top$. \par
With this, we can formulate the well-posedness result for the discrete models.

\begin{proposition}[Well-Posedness of the Semi-Discrete Problem]\label{prop:WPsemicont}
Let Assumptions \ref{ass:numconv} hold, then the problem (\ref{eqn:semdiscreqred}) has a unique solution $\vec{\b{u}}^\kappa_{Red}\in C^1([0,T], (\mathbb{R}^{d})^{I^\kappa}\times (\mathbb{R}^{d})^{I^\kappa} )$.
\end{proposition}

\begin{proof} We repeat the proof of \cite{WP_FRAC}, Theorem 2. Beginning first by rewriting (\ref{eqn:semdiscreqred}) into an  integral equation, 
\begin{equation}
\label{eqn:semdiscrinteqred}
        \b{u}_{Red}^\kappa(t) = \int_{0}^t\vec{G}_\kappa((\b{u}_{Red}^\kappa)_s, s) \dd s + \vec{u}_{0, Red}^\kappa 
\end{equation}
for all $t\in [0,T]$. The claim now follows from a standard application of a Picard iteration on $C^0([0,T], (\mathbb{R}^d)^{I^\kappa} \times (\mathbb{R}^d)^{I^\kappa})$ defined by
\begin{equation*}
    \b{u}^0:\equiv \vec{u}^\kappa_{0,Red},\ \ \ \b{u}^n(t) := \int_{0}^t \vec{G}_\kappa((\b{u}^{n-1})_s, s) \dd s+\vec{u}^\kappa_{0,Red}\text{, for } n\geq 1.
\end{equation*}
This yields the claim through (\ref{eqn:LSG}) and
\begin{equation*}
    \|\Phi_s - \Psi_s\|_{\Linftyptw([-T,0], (\mathbb{R}^{d})^{I^\kappa}\times (\mathbb{R}^{d})^{I^\kappa})} \leq \|\Phi - \Psi\|_{C^0([0,T], (\mathbb{R}^{d})^{I^\kappa}\times (\mathbb{R}^{d})^{I^\kappa})}
\end{equation*}
dealing with the time dependence of the first argument for all $s\in [0,T]$. All these bounds then lead to the usual bounds, the latter obtained through induction,
\begin{equation*}
    \|\b{u}^1(t) -\b{u}^0(t) \|_{\infty} \leq Mt, \ \ \ \|\b{u}^{n}(t)-\b{u}^{n-1}(t)\| \leq ML^{n-1}_{\vec{G}}\tfrac{t^n}{n!}, \text{ for } n\geq 2,
\end{equation*}
with $M:= L_{\vec{G}}\|\vec{u}^\kappa_{0,Red}\|_\infty + \|\vec{\b{b}_\kappa}\|_{C^0([0,T],(\mathbb{R}^d)^{I^\kappa})}$. This proves that the $(\b{u}^n)_{n\in \mathbb{N}_0}$ define a Cauchy sequence whose limit has to fulfill (\ref{eqn:semdiscrinteqred}).
\end{proof}

\begin{remark}[On Quadrature Weights]\label{rm:partvolumeweights}
Practical simulations include quadrature weights $w_{ij}$ replacing the $\chi_{[0,\delta]}(\|x_i-x_j\|_2)$ terms in $\omega_\delta(\|x_i-x_j\|_2)$, allowing for a more desirable numerical convergence behavior (See \cite{ImprovOnePointQuadAlgs, Seleson2018, Scabbia2023, LIU2024109115, TRASK2019151}). Then, $\omega_\delta(\|x_i-x_j\|_2) = \chi_{[0,\delta]}(\|x_i-x_j\|_2)\bar{\omega}(\|x_i-x_j\|_2)$ for some non-increasing $\bar{\omega}: (0,\infty) \rightarrow \mathbb{R}_{\geq 0}$ with $\omega_\delta \equiv \bar{\omega}$ on $(0,\delta]$ becomes $\omega^{ij}_\delta = w_{ij}\bar{\omega}(\|x_i-x_j\|_2)$. Including these weights does not cause significant theoretical problems for the results in this paper, as long as they fulfill the following assumptions taken from \cite{pflug2026convergencediscretizedlinearstatic}, for all $i,j\in I^\kappa$,
\begin{assumptions}[Assumptions on Partial Area Weights]\label{ass:quadweights}
\begin{enumerate}
    \item If $\|x_i-x_j\|_2 \geq \delta + \sqrt{d}\frac{\kappa}{2}$, then $w_{ij} = 0$.
    \item If $\|x_i-x_j\|_2 \leq \delta - \sqrt{d}\frac{\kappa}{2}$, then $w_{ij} = 1$.
    \item $0\leq w_{ij} \leq 1$.
    \item $w_{ij} = w_{ji}$.
\end{enumerate}
\end{assumptions}
Using these assumptions, Lemma \ref{lemm:unifbndinteg} and Lemma \ref{lemm:discrweighterrbnd} below can easily be adapted to include these terms. Rigorous proofs are given in the Appendix in Lemma \ref{lemm:alt34weights} and \ref{lemm:alt39weights}, respectively.
Note that requirements 1 and 2 merely ensure that if the $B_i$ do not intersect the horizon, or if it is fully inside, implied by its distance to $x_i$ alone, the weights have to be chosen to reflect this. 
\end{remark}

\subsection{Convergence of the Semi-Discrete Systems}
We will now aim to prove the convergence of the solutions for (\ref{eqn:semdiscrprobl}) to the continuous solution given by the system (\ref{eqn:conteq}). 
The main result will be stated in Theorem \ref{theo:convsemidiscr} below.

To this end, we begin with a few lemmas.
\begin{lemma}[Auxiliary Bounds]\label{lemm:auxbnd} Let $u\in L^\infty(\Omega)^d$, $\vec{u}\in (\mathbb{R}^d)^{I^\kappa}$, $i,j \in I^\kappa$ with $i\not=j$, then for almost all $x\in \Omega_i$ and almost all $x' \in \Omega_j$,
\begin{align}
\label{eqn:bndonSdiff}
    |\ca{S}^{ij}(\vec{u}) - \ca{S}(u, x',x)| & \leq \begin{cases}
        \frac{\sqrt{d}}{\|x'-x\|_2}( 2\|\Kk\vec{u}-u\|_{L^\infty(\Omega)^d} + (\ca{S}^{ij}(\vec{u})+2)\|x_j-x_i-x'+x\|_\infty) \\
        \frac{\sqrt{d}}{\|x_j-x_i\|_2}( 2\|\Kk\vec{u}-u\|_{L^\infty(\Omega)^d} + (\ca{S}({u},x',x)+2)\|x_j-x_i-x'+x\|_\infty),
    \end{cases}
\end{align}
where case brackets without conditions on the right signify that all cases hold simultaneously.\par
Similarly, for $\vec{u}$ with $\vec{u}_i-\vec{u}_j + x_i -x_j\not= 0$ and $u(x')-u(x)+x'-x \not= 0$,
\begin{align}\label{eqn:bndresesem}
    |e({u},x',x) - e^{ij}(\vec{u})| \leq 
    \tfrac{2\sqrt{d}}{(\ca{S}({u},x',x)+1)\|x'-x\|_2}\Big(2 \|\Kk\vec{u} - u\|_{L^\infty(\Omega)^d} + \|x-x_i-x'+x_j\|_\infty\Big).
\end{align}
For $\b{u}\in C^0([-T,0], L^\infty(\Omega)^d)$, $\vec{\b{u}}\in C^0([-T,0], (\mathbb{R}^d)^{I^\kappa})$, we have for almost all $x,x'$,
\begin{equation}\label{eqn:BndResultMu}
\begin{aligned}
    &\Big|{\mu\Big(\esssup_{-T\leq s \leq 0} \ca{S}([\Gamma\b{u}](s), x', x)\Big)}_{} -{\mu\Big(\sup_{-T\leq s \leq 0} \ca{S}^{ij}(\vec{\b{u}}(s))\Big)}_{}\Big|\\
    &\leq \Big((\sqrt{d}+1)\tfrac{\sqrt{d}\ell}{\|x'-x\|_2}\Big)( 2\|\Kk\vec{\b{u}}-\b{u}\|_{C^0([-T,0],L^\infty(\Omega)^d)} + (S^+_1+2)\|x_j-x_i-x'+x\|_\infty) .
\end{aligned}
\end{equation}
Similarly, we have,
\begin{equation}\label{eqn:BndResultscalarlaw}
\begin{aligned}
    &|\scalarlaw(\ca{S}(u, x', x)) -\scalarlaw(\ca{S}^{ij}(\vec{{u}}))|\\
    &\leq \begin{cases}
        \frac{\ell\sqrt{d}}{\|x'-x\|_2}( 2\|\Kk\vec{u}-u\|_{L^\infty(\Omega)^d} + (S^+_1+2)\|x_j-x_i-x'+x\|_\infty) \\
        \frac{\ell\sqrt{d}}{\|x_j-x_i\|_2}( 2\|\Kk\vec{u}-u\|_{L^\infty(\Omega)^d} + (S^+_1+2)\|x_j-x_i-x'+x\|_\infty).
    \end{cases}\\
    &\leq\frac{\ell\sqrt{d}(\sqrt{d}+1)}{\|x'-x\|_2}( 2\|\Kk\vec{u}-u\|_{L^\infty(\Omega)^d} + (S^+_1+2)\|x_j-x_i-x'+x\|_\infty).
\end{aligned}
\end{equation}
\end{lemma}

\begin{proof} For the first claim, we see that
\begin{align*}
    |\ca{S}^{ij}(\vec{u}) - \ca{S}(u, x',x)| &= \Big|\tfrac{\|\vec{u}_j-\vec{u}_i+x_j-x_i\|_2}{\|x_j-x_i\|_2} - \tfrac{\|u(x')-u(x)+x'-x\|_2}{\|x'-x\|_2}\Big|\\
     &\leq \tfrac{1}{\|x_j-x_i\|_2}\Big|\|\vec{u}_j-\vec{u}_i+x_j-x_i\|_2- \|u(x')-u(x)+x'-x\|_2\Big|\\
     &\quad + \|u(x')-u(x)+x'-x\|_2 \Big|\tfrac{1}{\|x'-x\|_2} - \tfrac{1}{\|x_j-x_i\|_2}\Big| \\
     &\leq \tfrac{1}{\|x_j-x_i\|_2} \Big(\|\Kk\vec{u}(x')-u(x')\|_2 + \|\Kk\vec{u}(x)-u(x)\|_2 + \|x'-x_j-x+x_i\|_2 \\
     & \quad + \tfrac{\|u(x')-u(x)+x'-x\|_2}{\|x'-x\|_2} \|x'-x_j-x+x_i\|_2\Big) \\
     &\leq \tfrac{\sqrt{d}}{\|x_j-x_i\|_2} \Big(2\|\Kk\vec{u}-u\|_{L^\infty(\Omega)^d} + (\ca{S}(u, x',x) +2) \|x'-x_j-x+x_i\|_\infty\Big).
\end{align*}
The other bound is derived similarly. For the differences of the $e$ terms, assuming the bond $x'-x$ is non-degenerate for $u$ in $x$, similarly that the bond $x_j-x_i$ is non-degenerate for $\vec{u}$ in $x_i$, we have
\begin{equation}
\begin{aligned}
     \|e({u},x',x) - e^{ij}(\vec{u})\|_\infty &\leq \Big\| \tfrac{{u}(x')-{u}(x)+x'-x}{\|{u}(x')-{u}(x)+x'-x\|_2}-\tfrac{\vec{u}_j-\vec{u}_i+x_j-x_i}{\|\vec{u}_j-\vec{u}_i+x_j-x_i\|_2}\Big\|_\infty\\
     &\leq \tfrac{\|{u}(x')-\vec{u}_j\|_\infty + \|{u}(x)-\vec{u}_i\|_\infty+\|x'-x-x_j+x_i\|_\infty}{\|{u}(x')-{u}(x)+x'-x\|_2} \\
     & + \sqrt{d}\|\vec{u}_j-\vec{u}_i +x_j-x_i\|_\infty \tfrac{\|{u}(x')-\vec{u}_j\|_\infty + \|{u}(x)-\vec{u}_i\|_\infty+\|x'-x-x_j+x_i\|_\infty}{\|\vec{u}_j-\vec{u}_i +x_j-x_i\|_2 \|{u}(x')-{u}(x)+x'-x\|_2} \\
     &\leq 2\sqrt{d}\tfrac{2 \|{u} -\Kk \vec{u}\|_{L^\infty(\Omega)^d} + \|x'-x-x_j+x_i\|_\infty}{\|{u}(x')-{u}(x)+x'-x\|_2} \\
    &\leq \tfrac{2\sqrt{d}}{(\ca{S}({u},x',x)+1)\|x'-x\|_2}(2 \|{u} -\Kk \vec{u}\|_{L^\infty(\Omega)^d} + \|x'-x-x_j+x_i\|_\infty).
    \end{aligned}
\end{equation}
For the $\mu$ terms, we see that
\begin{align*}
&\Big|\underbrace{\mu\Big(\esssup_{-T\leq s \leq 0} \ca{S}([\Gamma\b{u}](s), x', x)\Big)}_{=:A} -\underbrace{\mu\Big(\sup_{-T\leq s \leq 0} \ca{S}^{ij}(\vec{\b{u}}(s))\Big)}_{=:B}\Big|\\
\intertext{on which we apply the bound  (\ref{eqn:bndonSdiff}) after using the Lipschitz continuity of $\mu$. The values of $A$ and $B$ are used to choose the fitting case to bound  either $\ca{S}(u,x',x)$ or $\ca{S}^{ij}(\vec{\b{u}})$ from above,}
&\leq \begin{cases}
    0 & A = B= 0\\
    \frac{\sqrt{d}\ell}{\|x'-x\|_2}( 2\esssup_{-T\leq s \leq 0}\|\Kk\vec{\b{u}}(s)-\b{u}(s)\|_{L^\infty(\Omega)^d} + (S^+_1+2)\|x_j-x_i-x'+x\|)  & B >0 \\
    \frac{\sqrt{d}\ell}{\|x_j-x_i\|_2}( 2\esssup_{-T\leq s \leq 0}\|\Kk\vec{\b{u}}(s)-\b{u}(s)\|_{L^\infty(\Omega)^d} + (S^+_1+2)\|x_j-x_i-x'+x\|)  &  A>0
\end{cases} 
\end{align*}
since $A=0 \Leftrightarrow \esssup_{-T\leq s \leq 0} \ca{S}([\Gamma\b{u}](s), x', x) \geq S^+_1$ as well as $B=0 \Leftrightarrow \sup_{-T\leq s \leq 0} \ca{S}^{ij}(\vec{\b{u}}(s)) \geq S^+_1$. We also made use of 
\begin{align*}
    \Big| \esssup_{-T\leq s \leq 0} \ca{S}([\Gamma\b{u}](s), x', x)  - \sup_{-T\leq s \leq 0} \ca{S}^{ij}(\vec{\b{u}}(s)) \Big| &\leq \esssup_{-T\leq s \leq 0} |\ca{S}(\b{u}(s), x', x) - \ca{S}^{ij}(\vec{\b{u}}(s))|
\end{align*}
implied by the continuity of $\b{u}$ and $\vec{\b{u}}$ in $t \in [0,T]$.
The inequality (\ref{eqn:BndResultMu}) then follows from  $\|x-x'\| \leq (\sqrt{d} +1)\|x_i-x_j\|$ for $i\not=j$ and $x\in\Omega_i$, $x'\in\Omega_j$. \par
We apply similar reasoning to obtain
\begin{align*}
&\Big|\underbrace{\scalarlaw(\ca{S}(u, x', x))}_{=:A} -\underbrace{\scalarlaw( \ca{S}^{ij}(\vec{u})}_{=:B})\Big|\\
\intertext{we apply the Lipschitz continuity of $g$, and again select the suitable branch from (\ref{eqn:bndonSdiff})}
&\leq \begin{cases}
    0 & A = B= 0\\
    \frac{\sqrt{d}\ell}{\|x'-x\|_2}( 2\esssup_{-T\leq s \leq 0}\|\Kk\vec{\b{u}}(s)-\b{u}(s)\|_{L^\infty(\Omega)^d} + (S^+_1+2)\|x_j-x_i-x'+x\|)  & B \not=0 \\
    \frac{\sqrt{d}\ell}{\|x_j-x_i\|_2}( 2\esssup_{-T\leq s \leq 0}\|\Kk\vec{\b{u}}(s)-\b{u}(s)\|_{L^\infty(\Omega)^d} + (S^+_1+2)\|x_j-x_i-x'+x\|)  &  A\not=0.
\end{cases} 
\end{align*}
\end{proof}

We need the following lemmas to prove the next proposition.

\begin{lemma}[Shell Bounds for the Kernel]\label{lemm:shellbounds}
    Let $\omega_\delta: (0,\infty) \rightarrow \mathbb{R}_{\geq 0}$ be as required in Assumptions \ref{ass:numconv}. We then obtain for any $s > 0$,
    \begin{enumerate}
        \item If $r \geq 2s$ and $|l - r| \leq s$, we have
        \begin{equation*}
            |\omega_\delta(l) - \omega_\delta(r)| \leq \omega_\delta(r - s) - \omega_\delta(r + s)
        \end{equation*}
        \item We have
        \begin{equation*}
            \int_{\|\zeta\|_2 \geq 2s} \omega_\delta(\|\zeta\|_2 - s) - \omega_\delta(\|\zeta\|_2 + s) \dd \zeta \leq 3\cdot2^{d-1} s \boundkerinteg
        \end{equation*}
        \item For all $r>0$, 
        \begin{equation*}
            \int_{B_r(0)} \omega_\delta(\|\zeta\|_2) \dd \zeta \leq r \boundkerinteg.
        \end{equation*}
    \end{enumerate}
\end{lemma}
\begin{proof}
    Claim 1 follows directly from $\omega_\delta$ non-increasing as required in Assumptions \ref{ass:numconv}, as $l,r \in [r-s, r+s]$. \newline
    Claim 2 follows by using polar coordinates and using $c_d := \vol_{d-1}(S^{d-1})$,
    \begin{align*}
        &\int_{\|\zeta\|_2 \geq 2s} \omega_\delta(\|\zeta\|_2 - s) - \omega_\delta(\|\zeta\|_2 + s) \dd \zeta \\
        &= \int_{2s}^\infty \int_{S^{d-1}} (\omega_\delta(r -s) - \omega_\delta(r+s)) r^{d-1}\dd r \dd \mathcal{H}^{d-1}(\zeta) \\
        &= c_d \int_{2s}^\infty (\omega_\delta(r -s) - \omega_\delta(r+s)) r^{d-1}\dd r
        \intertext{splitting this integral and shifting the domains by $s$ via substitution yields}
        &= c_d \int_{s}^\infty \omega_\delta(t) (t+s)^{d-1}\dd t  - c_d \int_{3s}^\infty \omega_\delta(t) (t-s)^{d-1}\dd t \\
        &=  c_d \int_{s}^{3s} \omega_\delta(t) (t+s)^{d-1}\dd t  + c_d \int_{3s}^\infty \omega_\delta(t) ((t+s)^{d-1} -(t-s)^{d-1})\dd t 
    \end{align*}
    We then obtain for $t\in [s, 3s]$, that $(t+s)^{d-1} \leq 2^{d-1} t^{d-1} \leq 3 \cdot 2^{d-1} s t^{d-2}$, further for $t\geq 3s$, using the mean value theorem
    \begin{align*}
        (t+s)^{d-1} -(t-s)^{d-1} \leq 2(d-1) s (t+s)^{d-2} \leq  2(d-1) s (4/3)^{d-2} t^{d-2}
    \end{align*}
    using $d\geq 2$ and $s/t \leq 1/3$.  Both these inequalities then yield
    \begin{align*}
        &c_d \int_{s}^{3s} \omega_\delta(t) (t+s)^{d-1}\dd t  + c_d \int_{3s}^\infty \omega_\delta(t) ((t+s)^{d-1} -(t-s)^{d-1})\dd t \\
        &\leq s \max(2(d-1)(4/3)^{d-2}, 3\cdot 2^{d-1}) c_d\int_s^\infty \omega_\delta(t)   t^{d-2}\dd t \\
        &\leq 3\cdot 2^{d-1} s \int_{\|\zeta\|\geq s}\frac{\omega_\delta(\|\zeta\|_2)}{\|\zeta\|_2} \dd \zeta \leq 3\cdot 2^{d-1} s  \boundkerinteg
    \end{align*}
    Claim 3 follows immediately from $\|\zeta\|_2 \leq r$ for $\zeta \in B_r(0)$ and therefore $1 = \|\zeta\|_2/\|\zeta\|_2 \leq r/\|\zeta\|_2$ in the integrand for almost all $\zeta \in B_r(0)$.
\end{proof}

\begin{lemma}[Error Bound on the Discrete Weighting Function]\label{lemm:discrweighterrbnd} Let $i\in I^\kappa$, then, for all $x\in \Omega_i$,
\begin{equation*}
   \sum_{i\not=j}\int_{\Omega_j} |\omega_\delta(\|x_j-x_i\|_2)-\omega_\delta(\|x'-x\|_2)| \dd x' \leq D_\kappa,
\end{equation*}
with $0<D_\kappa< \infty$ independent of $x$ and $i$ as well as $ O(\kappa) \ni D_\kappa \searrow 0$ as $\frac{\delta}{2\sqrt{d}}> \kappa \searrow 0$.
\end{lemma}
\begin{proof} Set  $s:=\sqrt{d}\kappa$, $r(x'):=\|x'-x\|_2$ and $l(x'):=\|x_i-x_{j(x')}\|_2$, where $j(x')$ is the unique $j\in  \mathbb{Z}^d$ such that $x' \in B_j$. Then, Lemma \ref{lemm:shellbounds} yields using $|r(x') - l(x')| \leq s$ and $U := \Omega\backslash \Omega_i$,
\begin{align*}
    &\sum_{i\not=j}\int_{\Omega_j} |\omega_\delta(\|x_j-x_i\|_2)-\omega_\delta(\|x'-x\|_2)| \dd x' \\
    &\leq\int_{U}|\omega_\delta(l(x'))-\omega_\delta(r(x'))| \dd x' \\
    &\leq\int_{\|x'-x\|\geq 2s}|\omega_\delta(l(x'))-\omega_\delta(r(x'))| \dd x'\\
    &\quad + \int_{x'\in U, \|x'-x\|\leq 2s}|\omega_\delta(l(x'))-\omega_\delta(r(x'))| \dd x'
\end{align*}
We then have 
\begin{align*}
    &\int_{\|x'-x\|\geq 2s}|\omega_\delta(l(x'))-\omega_\delta(r(x'))| \dd x' \\
    &\leq \int_{\|x'-x\|\geq 2s}\omega_\delta(r(x')-s)-\omega_\delta(r(x')+s) \dd x' \\
    &\leq 3\cdot 2^{d-1} s \boundkerinteg
\end{align*}
as well as
\begin{align*}
    &\int_{x'\in U, \|x'-x\|\leq 2s}|\omega_\delta(l(x'))-\omega_\delta(r(x'))| \dd x' \\
    &\leq \int_{x'\in U, \|x'-x\|\leq 2s}\omega_\delta(l(x'))+\omega_\delta(r(x')) \dd x' \\
    &\leq 3s\boundkerdiscrinteg + 2s \boundkerinteg,
\end{align*}
due to $l(x') \leq 3s$ if $r(x') \leq 2s$, with Lemma \ref{lemm:unifbndinteg}. The claim now follows as $s = \sqrt{d}\kappa$. 
\end{proof}
We then obtain the first result.

\begin{proposition}[Bound on the Discretization Error]\label{prop:discrrhserr}
Let $\b{u}\in C^0([0,T], L^\infty(\Omega)^d)$ and $\vec{\b{u}} \in C^0([0,T], (\mathbb{R}^d)^{I^\kappa})$, then, for all $t\in [0,T]$,
\begin{equation*}
    \|F(\b{u}_t) - \Kk \vec{F}_\kappa((\vec{\b{u}})_t) \|_{L^\infty(\Omega)^d} \leq \beta \max_{0\leq s \leq t}\|\b{u}(s) - \Kk\vec{\b{u}}(s)\|_{L^\infty(\Omega)^d} + \alpha_\kappa
\end{equation*}
for some $\alpha_\kappa, \beta \geq 0$, and \begin{equation*}
    O(\kappa) \ni \alpha_\kappa \xrightarrow[]{\kappa \searrow 0} 0,\ \ \ 0<\beta < \infty
\end{equation*}
for $\kappa \searrow 0$ and $\beta>0$ independent of  $0<\kappa< \delta/(2\sqrt{d})$ and $t\in [0,T]$.
\end{proposition}
\begin{proof}
Again, we now need to deal with sub-critical bonds first, which we do by defining $e^{ij}_{u}: (\mathbb{R}^d)^{I_\kappa}\times \Omega_j \times \Omega_i \rightarrow \mathbb{R}^d$ by
\begin{align*}
    e^{ij}_{u}(\vec{a},x',x) = \begin{cases}e^{ij}(\vec{a}) & \text{if } \ca{S}^{ij}(\vec{a})\geq S^-_1\\
    e(u,x',x) & \text{else,}
    \end{cases}
\end{align*}
for $\vec{a} \in (\mathbb{R}^d)^{I^\kappa}$ and $u \in L^\infty(\Omega)^d$.
With this, we also define for $x\in \Omega_i$ for all $i$,
\begin{align*}
    &(F_{\kappa,u}^{\text{discr}}(\vec{\b{a}}^\kappa))(x) \\
    &= \int_\Omega \sum_{i\not=j}\chi_{\Omega_j}(x') {\omega_\delta(\|x_j-x_i\|_2)\scalarlaw(\ca{S}^{ij}(\vec{\b{a}}^\kappa(0)))\mu\Big(\sup_{-T\leq s\leq 0}\ca{S}^{ij}(\vec{\b{a}}^\kappa(s))\Big)e^{ij}_{u}(\vec{\b{a}}^\kappa(0),x',x)} \dd x'.
\end{align*}
We then obtain $\Kk\vec{F}_\kappa(\vec{\b{a}}^\kappa) = F_{\kappa,u}^{\text{discr}}(\vec{\b{a}}^\kappa)$ for all $\vec{\b{a}}^\kappa\in C^0([-T,0], (\mathbb{R}^d)^{I^\kappa})$ and $u \in L^\infty(\Omega)^d$.\par

    We pick an $i\in I^\kappa$, then for almost all $x\in \Omega_i$,
    \begin{align}
        &\|(F(\b{u}_t) - \Kk \vec{F}_{\kappa}((\vec{\b{u}})_t)(x)\|_\infty= \|(F(\b{u}_t) - (F_{\kappa,\b{u}(t)}^{\text{discr}}((\vec{\b{u}})_t))(x)\|_\infty\nonumber \\
        &\leq \Big\|\sum_{i\not=j} \int_{\Omega_j} \omega_\delta(\|x'-x\|_2)\scalarlaw(\ca{S}(\b{u}(t), x', x))\mu\Big(\esssup_{0\leq s\leq t}\ca{S}([\Gamma\b{u}](s), x', x) \Big)e(\b{u}(t), x',x) \label{eqn:lastineq} \\
        &\phantom{=====}-\omega_\delta(\|x_j-x_i\|_2)\scalarlaw(\ca{S}^{ij}(\vec{\b{u}}^\kappa(t)))\mu\Big(\sup_{0\leq s\leq t}\ca{S}^{ij}(\vec{\b{u}}^\kappa(s))\Big)e^{ij}_{\b{u}(t)}(\vec{\b{u}}^\kappa(t),x',x) \dd x'\Big\|_\infty \nonumber\\
        &\quad + \int_{\Omega_i}\|\omega_\delta(\|x'-x\|_2)\scalarlaw(\ca{S}(\b{u}(t), x', x))\mu\Big(\esssup_{0\leq s\leq t}\ca{S}([\Gamma\b{u}](s), x', x) \Big)e(\b{u}(t), x',x) \|_\infty \dd x'\nonumber
    \end{align}
We now see that for all $x\in \Omega_i$,
\begin{align*}
    &\int_{\Omega_i}\|\omega_\delta(\|x'-x\|_2)\scalarlaw(\ca{S}(\b{u}(t), x', x))\mu\Big(\esssup_{0\leq s\leq t}\ca{S}([\Gamma\b{u}](s), x', x) \Big)e(\b{u}(t), x',x)\|_\infty \dd x' \\
    &\leq \|\omega_\delta(\|\cdot\|_2)\|_{L^1(B_{\kappa\sqrt{d} }(0))} \scalarlaw^+ =: \alpha_1^\kappa 
\end{align*}
Since $\omega_\delta \in L^1(B_\delta(0))$, we see that $\alpha_1^\kappa \searrow 0$ for $\kappa \searrow 0$. With Lemma \ref{lemm:shellbounds} we see that $\alpha_1^\kappa \in O(\kappa)$. \par

We then derive a bound for the other summand in (\ref{eqn:lastineq}), using Lemma \ref{lemm:unifbndinteg} and Lemma \ref{lemm:auxbnd}, 
\begin{align*}
    &\Big\|\sum_{i\not=j} \int_{\Omega_j} \omega_\delta(\|x'-x\|_2)\scalarlaw(\ca{S}(\b{u}(t), x', x))\mu\Big(\esssup_{0\leq s\leq t}\ca{S}([\Gamma\b{u}](s), x', x) \Big)e(\b{u}(t), x',x) \\
    &\phantom{===}-\omega_\delta(\|x_j-x_i\|_2)\scalarlaw(\ca{S}^{ij}(\vec{\b{u}}^\kappa(t)))\mu\Big(\max_{0\leq s\leq t}\ca{S}^{ij}(\vec{\b{u}}^\kappa(s))\Big)e^{ij}_{\b{u}(t)}(\vec{\b{u}}^\kappa(t),x',x) \dd x'\Big\|_\infty\\
    &\leq \Big\|\sum_{i\not=j} \int_{\Omega_j}  (\omega_\delta(\|x_j-x_i\|_2)-\omega_\delta(\|x'-x\|_2))\scalarlaw(\ca{S}^{ij}(\vec{\b{u}}^\kappa(t)))\mu\Big(\max_{0\leq s\leq t}\ca{S}^{ij}(\vec{\b{u}}^\kappa(s))\Big)e^{ij}_{\b{u}(t)}(\vec{\b{u}}^\kappa(t),x',x) \dd x'\Big\|_\infty \\
    &\quad + \Big\|\sum_{i\not=j} \int_{\Omega_j}  \omega_\delta(\|x'-x\|_2)|\scalarlaw(\ca{S}^{ij}(\vec{\b{u}}^\kappa(t)))-\scalarlaw(\ca{S}(\b{u}(t), x', x))|\mu\Big(\max_{0\leq s\leq t}\ca{S}^{ij}(\vec{\b{u}}^\kappa(s))\Big)e^{ij}_{\b{u}(t)}(\vec{\b{u}}^\kappa(t),x',x) \dd x'\Big\|_\infty \\
    &\quad+ \Big\|\sum_{i\not=j} \int_{\Omega_j} \omega_\delta(\|x'-x\|_2)\mu\Big(\max_{0\leq s\leq t}\ca{S}^{ij}(\vec{\b{u}}^\kappa(s))\Big)\scalarlaw(\ca{S}(\b{u}(t), x', x))(e^{ij}_{\b{u}(t)}(\vec{\b{u}}^\kappa(t), x',x)-e(\b{u}(t), x',x)) \dd x'\Big\|_\infty \\
    &\quad+ \Big\|\sum_{i\not=j} \int_{\Omega_j} \omega_\delta(\|x'-x\|_2)\scalarlaw(\ca{S}(\b{u}(t), x', x))\Big|\mu\Big(\max_{0\leq s\leq t}\ca{S}^{ij}(\vec{\b{u}}^\kappa(s))\Big)\\
    &\phantom{========================00=}-\mu\Big(\esssup_{0\leq s\leq t}\ca{S}([\Gamma\b{u}](s), x',x)\Big)\Big|e(\b{u}(t), x',x)\dd x' \Big\|_\infty \\
    \intertext{Again, by construction, $\scalarlaw(\ca{S}(\b{u}(t), x', x))(e^{ij}_{\b{u}(t)}(\vec{\b{u}}^\kappa(t), x', x)-e(\b{u}(t), x',x))$ is $0$ if $\ca{S}(\b{u}(t), x', x) \in [-1, S^-_1)$ or $\ca{S}^{ij}(\vec{\b{u}}^\kappa(t))\in [-1, S^-_1)$, so apply the bound (\ref{eqn:bndresesem}) with $\ca{S}(\b{u}(t),x',x) \geq S^-_1$,}
    &\leq \scalarlaw^+ \Big\|\sum_{i\not=j}\int_{\Omega_j} |\omega_\delta(\|x_j-x_i\|_2)-\omega_\delta(\|x'-x\|)| \dd x'\Big\|_{L^\infty(\Omega_i)^d}\\
    &\quad+\ell\sum_{i\not=j} \int_{\Omega_j}  \omega_\delta(\|x'-x\|_2)\frac{\sqrt{d}(\sqrt{d}+1)}{\|x'-x\|_2}( 2\|(\Kk\vec{\b{u}}-\b{u})(t)\|_{L^\infty(\Omega)^d}\\
    &\phantom{===========================} + (S^+_1+2)\|x_j-x_i-x'+x\|_\infty)\dd x' \\
    &\quad+ \scalarlaw^+\sum_{j\not=i} \int_{\Omega_j}\omega_\delta(\|x'-x\|) \tfrac{2\sqrt{d}}{(S^-_1+1)\|x'-x\|_2}\Big(2 \|(\Kk\vec{\b{u}} - \b{u})(t)\|_{L^\infty(\Omega)^d}\\
    &\phantom{===========================}+ \|x-x_i-x'+x_j\|_\infty\Big) \dd x' \\
    &\quad+\ell \scalarlaw^+ \sum_{j\not=i} \int_{\Omega_j}\omega_\delta(\|x'-x\|)\Big(\tfrac{(\sqrt{d}+1)\sqrt{d}}{\|x'-x\|_2}\Big)( 2\max_{0\leq s\leq t}\|(\Kk \vec{\b{u}} - \b{u})(s)\|_{L^\infty(\Omega)^d}\\
    &\phantom{===========================}+ (S^+_1+2)\|x_j-x_i-x'+x\|_\infty)   \dd x'\\
    &\leq  \scalarlaw^+ D_\kappa \\
    &\quad + \ell \sqrt{d}(\sqrt{d}+1) \boundkerinteg \Big(2\max_{0\leq s\leq t}\|(\Kk \vec{\b{u}} - \b{u})(s)\|_{L^\infty(\Omega)^d} + (S^+_1+2){\kappa}\Big) \\
    &\quad + 2\sqrt{d}\tfrac{\scalarlaw^+}{(S^-_1+1)} \boundkerinteg \Big(2\max_{0\leq s\leq t}\|(\Kk \vec{\b{u}} - \b{u})(s)\|_{L^\infty(\Omega)^d} + {\kappa}\Big) \\
    &\quad +  2\ell\scalarlaw^+ \sqrt{d}(\sqrt{d}+1)  \boundkerinteg \Big(\max_{0\leq s\leq t}\|(\Kk \vec{\b{u}} - \b{u})(s)\|_{L^\infty(\Omega)^d} + \tfrac{1}{2}(S^+_1 + 2){\kappa}\Big),
\end{align*}
where we made use of 
\begin{equation*}
    \max_{-T\leq s\leq 0}\ca{S}^{ij}(\vec{\b{u}}^\kappa(s))=  \esssup_{-T\leq s\leq 0}\ca{S}^{ij}(\vec{\b{u}}^\kappa(s))
\end{equation*}
and
\begin{equation*}
    \esssup_{0\leq s \leq t}\|[\b{u} - \Kk\vec{\b{u}}^{\kappa}](s)\|_{L^\infty(\Omega)^d} = \max_{0\leq s\leq t}\|[\b{u} - \Kk\vec{\b{u}}^{\kappa}](s)\|_{L^\infty(\Omega)^d}
\end{equation*}
due to the continuity of $\b{u}$ and $\Kk\vec{\b{u}}$ in $t\in[0,T]$.
\end{proof}

The next step will be to transport the semi-discrete equations onto $L^\infty(\Omega)^d$, made possible by the following lemma.

\begin{lemma}[Transport Lemma]\label{lemm:transport} Let $\vec{\b{u}}\in C^1([0,T], (\mathbb{R}^d)^{I^\kappa})$, then $\Kk\vec{\b{u}} \in C^1([0,T], L^\infty(\Omega)^d)$ and
\begin{equation*}
    \tfrac{d}{dt}(\Kk \vec{\b{u}})(t) = \Kk \dot{\vec{\b{u}}}(t)
\end{equation*}
as well as
\begin{equation} \label{eqn:transp2}
    \Kk \vec{\b{u}}(t)- \Kk {\vec{\b{u}}}(0) = \int_{0}^t \Kk \dot{\vec{\b{u}}}(s) \dd s,
\end{equation}
for all $t\in [0,T]$.
\end{lemma}
\begin{proof}
    This follows from $\Kk$ being a bounded operator for each $0<\kappa$.
\end{proof}

With this we now prove the convergence of the solutions as follows. 

\begin{theorem}[Convergence of the Semi-Discrete Solutions]\label{theo:convsemidiscr}
    Let Assumptions \ref{ass:numconv} hold. Let $\kappa_n \searrow 0$ and $\vec{\b{u}}^{\kappa_n}\in C^2([0,T], (\mathbb{R}^d)^{I^{\kappa_n}})$ be the solutions to the semi-discrete problem (\ref{eqn:semdiscrprobl}), as well as $\b{u}\in C^2([0,T], L^\infty(\Omega)^d)$ the solution to the continuous problem (\ref{eqn:conteq}), then
    \begin{equation*}
        \|\b{u} - \Kkf{n}\vec{\b{u}}^{\kappa_n}\|_{C^1([0,T], L^\infty(\Omega)^d)} \xrightarrow[]{n \rightarrow \infty} 0. 
    \end{equation*}
    If further, we have
    \begin{equation} \label{eqn:asssemidiscrconvinputdata}
        \|u_0 - \Kkf{n} \vec{u}_0^{\kappa_n}\|_{L^\infty(\Omega)^d}\in O(\kappa), \|v_0 - \Kkf{n} \vec{v}_0^{\kappa_n}\|_{L^\infty(\Omega)^d} \in O(\kappa),  \|\b{b} - \Kkf{n} \vec{\b{b}}_{\kappa_n}\|_{C^0([0,T], L^\infty(\Omega)^d)} \in O(\kappa), 
    \end{equation}
    then 
    \begin{equation}\label{eqn:convrateressemidisr}
        \|\b{u} - \Kkf{n}\vec{\b{u}}^{\kappa_n}\|_{C^1([0,T], L^\infty(\Omega)^d)} \in O(\kappa)
    \end{equation}
\end{theorem}
\begin{proof}
We define $\Kk^2: (\mathbb{R}^d)^{I^\kappa}\times (\mathbb{R}^d)^{I^\kappa} \rightarrow L^\infty(\Omega)^{2d},\ (\vec{u},\vec{v}) \mapsto (\Kk \vec{u}, \Kk \vec{v})$, similarly \\ $\Pk^2: L^\infty(\Omega)^{2d} \rightarrow (\mathbb{R}^d)^{I^\kappa}\times (\mathbb{R}^d)^{I^\kappa}, (\phi, \psi) \mapsto (\Pk\phi, \Pk\psi)$. Using Lemma \ref{lemm:transport}, we transport the reduced semi-discrete equations (\ref{eqn:semdiscreqred}) to $L^\infty(\Omega)^d\times L^\infty(\Omega)^d$,
\begin{equation*}
    \begin{rcases}
        \tfrac{d}{dt}{\b{u}}_{Red}^{\kappa_n}(t) = \Kkf{n}^2 \vec{G}_{\kappa_n}((\Pkn^2{\b{u}}_{Red}^{\kappa_n})_t, t) \\
        {\b{u}}_{Red}^{\kappa_n}(0) =  \Kkf{n}^2\vec{u}_{0, Red}^{\kappa_n}.
    \end{rcases}
\end{equation*}
Through direct substitution, one can verify that the unique solution $\b{u}_{Red}^{\kappa_n}$ fulfills $\b{u}_{Red}^{\kappa_n} = \Kk^2 \vec{\b{u}}^{\kappa_n}_{Red}$ in $C^1([0,T], L^\infty(\Omega)^{2d})$, where $\vec{\b{u}}_{Red}^{\kappa_n}$ is the unique solution to (\ref{eqn:semdiscreqred}). This follows from Lemma \ref{lemm:transport}, $\Pkn^2 \circ \Kkf{n}^2 = \id\times \id$, and from the fact that $\img\Kkf{n}^2$ is a closed linear subspace of $L^\infty(\Omega)^{2d}$. 
We then have
\begin{align*}
    [\b{u}_{Red} - \b{u}^{\kappa_n}_{Red}](t) &= [\b{u}_{Red} - \b{u}^{\kappa_n}_{Red}](0) + \int_{0}^t \tfrac{d}{dt}[\b{u}_{Red} - \b{u}^{\kappa_n}_{Red}](s) \dd s \\
    &= \begin{pmatrix}
        u_0 -\Kkf{n}\vec{u}_0^{\kappa_n} \\
        v_0 -\Kkf{n} \vec{v}_0^{\kappa_n}
    \end{pmatrix} 
    +\int_{0}^t \begin{pmatrix}
        \dot{\b{u}}(s) - \Kkf{n}\dot{\vec{\b{u}}}^{\kappa_n}(s)\\
        F((\b{u})_s)+ \b{b}(s) - \Kkf{n}\vec{F}_{\kappa_n}(({\vec{\b{u}}}^{\kappa_n})_s) -\Kkf{n}\vec{\b{b}}_{\kappa_n}(s) 
    \end{pmatrix}\dd s \\
    &= \begin{pmatrix}
        u_0 - \Kkf{n}\vec{u}_0^{\kappa_n} \\
        v_0 - \Kkf{n} \vec{v}_0^{\kappa_n}
    \end{pmatrix} 
    + \int_{0}^t \begin{pmatrix}
        0\\
         \b{b}(s)-\Kkf{n}\vec{\b{b}}_{\kappa_n}(s) 
    \end{pmatrix} 
    \dd s\\
    &\quad\quad\quad\quad\quad\quad\quad\quad+\int_{0}^t \begin{pmatrix}
        \dot{\b{u}}(s) - \Kkf{n}\dot{\vec{\b{u}}}^{\kappa_n}(s)\\
        F((\b{u})_s) - \Kkf{n} \vec{F}_{\kappa_n}(({\vec{\b{u}}}^{\kappa_n})_s)
    \end{pmatrix}\dd s.
\end{align*}
This leads to, using Proposition \ref{prop:discrrhserr},
\begin{align*}
    E(t) &:= \|[\b{u}^{\kappa_n}_{Red} - \b{u}_{Red}](t)\|_{L^\infty(\Omega)^{2d}} \\
    &\leq \underbrace{\|u_0 - \Kkf{n} \vec{u}_0^{\kappa_n}\|_{L^\infty(\Omega)^d} +  \|v_0 - \Kkf{n} \vec{v}_0^{\kappa_n}\|_{L^\infty(\Omega)^d} + T\|\b{b} - \Kkf{n} \vec{\b{b}}_{\kappa_n}\|_{C^0([0,T], L^\infty(\Omega)^d)}}_{=: \alpha_2^{\kappa_n}} \\
    &\quad+\int_{0}^t \begin{Vmatrix}
        \|\dot{\b{u}}(s) - \Kkf{n}\dot{\vec{\b{u}}}^{\kappa_n}(s)\|_{L^\infty(\Omega)^d} \\
        \beta \Big(\max_{0\leq \tau \leq s}\|\b{u}(\tau) -\Kkf{n}\vec{\b{u}}(\tau)\|_{L^\infty(\Omega)^d}\Big) + \alpha_{\kappa_n}
    \end{Vmatrix}_\infty \dd s  \\
    &\leq \alpha_2^{\kappa_n} + t \alpha_{\kappa_n} \\
    &\quad + \int_0^{t} \max(1, \beta) \begin{Vmatrix}
    \max_{0\leq \tau \leq s}\|\dot{\b{u}}(\tau) - \Kkf{n}\dot{\vec{\b{u}}}(\tau)\|_{L^\infty(\Omega)^d} \\
    \max_{0\leq \tau \leq s}\|\b{u}(\tau) - \Kkf{n}\vec{\b{u}}(\tau)\|_{L^\infty(\Omega)^d}
    \end{Vmatrix}_\infty\dd s \\
    &\leq  \alpha_2^{\kappa_n} + t \alpha_{\kappa_n}\\
    &\quad + \int_0^{t} \max(1, \beta) 
    \max_{0\leq \tau \leq s}\|{\b{u}}_{Red}(\tau) - \Kkf{n}^2{\vec{\b{u}}}_{Red}(\tau)\|_{L^\infty(\Omega)^{2d}}\dd s.
\end{align*}
Defining 
\begin{align}
    A^{\kappa_n}(t) &:=  \alpha_2^{\kappa_n} + t \alpha_{\kappa_n}\geq 0, \label{eqn:defAkapp}\\
    B&:= \max(1, \beta)\geq 0, \nonumber
\end{align}
we then obtain for all $t\in [0,T]$,
\begin{align*}
    E(t)\leq A^{\kappa_n}(t) + \int_0^{t} B \max_{0\leq \tau \leq s} E(\tau) \dd s,
\end{align*}
with $A^{\kappa_n}$ non-decreasing, such that we further have for $\Delta(t) := \max_{0\leq s \leq t} E(s)$,
\begin{equation*}
    \Delta(t) \leq A^{\kappa_n}(t) + \int_0^{t} B\Delta(s) \dd s.
\end{equation*}
Applying the Gr\"onwall inequality now yields, for all $t\in [0,T]$,
\begin{equation}
\label{eqn:groenonDel}
    0 \leq E(t) \leq \Delta(t) \leq A^{\kappa_n}(t) \exp\big(tB\big).
\end{equation}
Due to Assumptions \ref{ass:numconv}, Proposition \ref{prop:discrrhserr} and (\ref{eqn:defAkapp}), we then obtain
\begin{equation*}
    A^{\kappa_n} \Rightarrow 0
\end{equation*}
as $n \rightarrow \infty$, which using (\ref{eqn:groenonDel}) leads to, for all $t\in [0,T]$,
\begin{equation*}
    0 \leq E(t) \leq\Delta(t) \leq \Delta(T) \searrow 0
\end{equation*}
for $n\rightarrow \infty$. If (\ref{eqn:asssemidiscrconvinputdata}) holds, we immediately obtain $\alpha_2^{\kappa} \in O(\kappa)$ and therefore $\|A^\kappa\|_{C^{0}([0,T])} \in O(\kappa)$ using the rates in Proposition \ref{prop:discrrhserr}, implying (\ref{eqn:convrateressemidisr}) via (\ref{eqn:groenonDel}).
\end{proof}

\section{Convergence of the Fully Discrete System}
\label{sec:ConvFullyDiscr}
We introduce a parameter $N\in\mathbb{N}$ and $h:= T/N$ controlling the discretization with respect to time. The discrete system under consideration in this section is given by a central difference scheme for $n\geq 1$,
\begin{equation}
\label{eqn:discrprob}
    \begin{aligned}
        \vec{\b{u}}^{-1,N}_\kappa& = \vec{u}_{0}^\kappa- h\vec{v}_{0}^\kappa + \tfrac{1}{2}h^2(\vec{F}_\kappa(\IP_1(\vec{\b{u}}^{0,N}_\kappa))+ \vec{\b{b}}_\kappa(0)) \\
        \vec{\b{u}}^{0,N}_\kappa& = \vec{u}_{0}^\kappa \\
        \vec{{\b{u}}}^{n,N}_\kappa &= 2\vec{\b{u}}^{n-1,N}_\kappa - \vec{\b{u}}^{n-2,N}_\kappa +h^2(\vec{F}_\kappa(\IP_{n}(\vec{\b{u}}^{0,...,n-1,N}_\kappa))+ \vec{\b{b}}_\kappa(t_{n-1}))\\
    \end{aligned}
\end{equation}
where $\vec{\b{u}}_\kappa^{0,...n-1,N}$ is the family $(\vec{\b{u}}_\kappa^{n',N})_{0\leq n'\leq n-1}$, and $\IP_n: (\mathbb{R}^{d|I^\kappa|})^{n} \rightarrow \Linftyptw([-T,0], (\mathbb{R}^d)^{I^\kappa})$ given by the piecewise function
\begin{equation*}
    \IP_n((\vec{u}_k)_{0 \leq k \leq n-1}) = t\mapsto \vec{u}_0+\sum_{1\leq k \leq n-1}\chi_{({\color{black}-hk,-h(k-1)}]}(t)(\vec{u}_{n-k}-\vec{u}_0).
\end{equation*}
This function is defined such that $\IP_n((\vec{u}_k)_{0 \leq k \leq n-1})(0) = \vec{u}_{n-1}$.\par
Note that
\begin{equation*}
    \max_{-T \leq \tau\leq 0} \ca{S}^{ij}(\IP_n(\vec{\b{u}}^{0,...n-1,N}_\kappa)(\tau)) = \max_{0\leq k \leq n-1} \ca{S}^{ij}(\vec{\b{u}}^{k,N}_\kappa),
\end{equation*}
which reduces the maximum over $[0,t]$ in the discrete non-local operator $\vec{F}_\kappa$ to a finite maximum that can be easily evaluated in the numerical code. We define $t_n^N := hn$ for $0 \leq n \leq N$ and drop the $N$ from the notation if it is clearly implied by the context. \newline

We now aim to prove the main convergence result as given in Theorem \ref{theo:fullconv} below. We begin with the following definitions.

\begin{definition} We define the following.
\begin{enumerate}
    \item ${\hat{G}}_h: \Linftyptw([-T,0], (\mathbb{R}^d)^{I^\kappa} \times (\mathbb{R}^d)^{I^\kappa}) \times [0,T] \rightarrow (\mathbb{R}^d)^{I^\kappa} \times (\mathbb{R}^d)^{I^\kappa} $ with 
    \begin{align*}
        {\hat{G}}_h((\phi,\psi)^\top, t) = \begin{pmatrix}
        \psi(0)\\
        \vec{F}_\kappa(\tau_{-h}\phi) + \vec{\b{b}}_\kappa(t)
    \end{pmatrix}.
    \end{align*}
    where $\tau_{-h}$ shifts the map by $h$ to the right, extending it by the value at $-T$ on the left. That is, for $\b{a} \in \Linftyptw([-T, 0], X)$ and some Banach Space $X$, we define similarly to the existing shift operator, 
    \begin{equation*}
        [\tau_{-h}\b{a}](s) := \begin{cases}
            \b{a}(s-h) & s\geq -T+h\\
            \b{a}(-T)  & \text{else.}
        \end{cases}
    \end{equation*}
    \item Let $\vec{\b{u}}_{Red}^\kappa$ be the unique solution to (\ref{eqn:semdiscreqred}). We then define the {consistency residual} $K^{N,n}_\kappa$ for $1\leq n \leq N$ as
    \begin{align}\label{eqn:defstabres}
        K^{N,n}_\kappa := \tfrac{1}{h}( \vec{\b{u}}_{Red}^\kappa(t_{n-1})-\vec{\b{u}}_{Red}^\kappa(t_{n}) )+ {\hat{G}}_h((\vec{\b{u}}^\kappa_{Red})_{t_n}, t_{n-1}).
    \end{align}
    \item We then define
\begin{equation}
\label{eqn:redfulldiscrsol}
    \hat{\b{u}}_\kappa^{n,N} := \begin{pmatrix}
        \vec{\b{u}}_\kappa^{n,N} \\
        \frac{\vec{\b{u}}_\kappa^{n,N}-\vec{\b{u}}_\kappa^{n-1,N}}{h}
    \end{pmatrix} =: \begin{pmatrix}
        \vec{\b{u}}_\kappa^{n,N} \\
        \vec{\b{v}}_\kappa^{n,N}
    \end{pmatrix} \in (\mathbb{R}^d)^{I^\kappa}\times(\mathbb{R}^d)^{I^\kappa}
\end{equation}
to equivalently reformulate (\ref{eqn:discrprob}) to 
\begin{equation}
\label{eqn:redfulldiscr}
    0 =  \hat{\b{u}}_\kappa^{n-1,N} -\hat{\b{u}}_\kappa^{n,N}+  h\begin{pmatrix}
         (\hat{\b{u}}_\kappa^{n,N})_2 \\
         \vec{F}_\kappa(\IP_{n}((\hat{\b{u}}^{...n-1,N}_\kappa)_1))+ \vec{\b{b}}_\kappa(t_{n-1})
    \end{pmatrix}.
\end{equation}
\item We define the discretization error at $t_n$ as
\begin{equation}\label{eqn:defediscr}
    e_n := \vec{\b{u}}_{Red}^\kappa (t_n) - \hat{\b{u}}_\kappa^{n,N}.
\end{equation}
\end{enumerate}
\end{definition}

We then obtain our first lemma.
\begin{lemma}[Stability Inequality]\label{lemm:stabineq}
    We have for all $N \in \mathbb{N}$, $1\leq n \leq N$ and all $\delta/\sqrt{d}>\kappa >0$, that
    \begin{align*}
        (1-h)\|e_n\|_\infty& \leq h\|K^{N,n}_\kappa\|_\infty+ \|e_{n-1}\|_\infty \\
       &\quad + hL_{\vec{F}}\|\IP_{n}((\hat{\b{u}}^{...n-1,N}_\kappa)_1)) - ((\vec{\b{u}}_{Red}^\kappa)_1)_{t_{n-1}}\|_{\Linftyptw([-T,0], (\mathbb{R}^d)^{I^\kappa})},
    \end{align*}
    where $\vec{\b{u}}_{Red}^\kappa$ is the unique solution to (\ref{eqn:semdiscreqred}) and $\hat{\b{u}}^{n,N}_\kappa$ as given in (\ref{eqn:redfulldiscrsol}) with the $\b{u}^{n,N}_\kappa$ as given in (\ref{eqn:discrprob}).
\end{lemma}

\begin{proof}
    Subtracting (\ref{eqn:redfulldiscr}) from (\ref{eqn:defstabres}) yields
    \begin{equation*}
            h K^{N,n}_\kappa = e_{n-1} -e_{n} - h\begin{pmatrix}
        (\hat{\b{u}}_\kappa^{n,N})_2 - (\vec{\b{u}}_{Red}^\kappa)_2(t_n) \\
        \vec{F}_\kappa(\IP_{n}((\hat{\b{u}}^{...n-1,N}_\kappa)_1))) - \vec{F}_\kappa(((\vec{\b{u}}_{Red}^\kappa)_1)_{t_{n-1}})
    \end{pmatrix}.
    \end{equation*}
    Rearranging the terms and applying the norm and then Proposition \ref{prop:LScontdiscrrhs} then yields the claim.
\end{proof}

We also need a priori bounds for the semi-discrete solution uniform with respect to $\kappa$.

\begin{lemma}[Uniform A-Priori Bounds for the Semi-Discrete Solutions]\label{lemm:aprioriboundssemidiscrete}
    Let Assumptions \ref{ass:numconv} hold and $\kappa_n \searrow 0$ with $\kappa_n < \frac{\delta}{\sqrt{d}}$ for all $n\in\mathbb{N}$. Define
    \begin{align*}
        R:= \sup_{n\in \mathbb{N}} (\|\vec{u}_0^{\kappa_n}\|_\infty + \|\vec{v}_0^{\kappa_n}\|_\infty + \|\vec{\b{b}}^{\kappa_n}\|_{C^0([0,T], (\mathbb{R}^d)^{I^\kappa})}) < \infty
    \end{align*}
    Then, for all $n\in \mathbb{N}$, $t, t' \in [0,T]$, we have
    \begin{align*}
        \|\vec{\b{u}}^{\kappa_n}_{\text{Red}}(t)\|_\infty\leq M_0:=R(1+T)\,e^{L_{\vec{G}}T},\quad \|\vec{\b{u}}^{\kappa_n}_{\text{Red}}(t)-\vec{\b{u}}^{\kappa_n}_{\text{Red}}(t')\|_\infty\leq M_1 |t-t'|,\quad M_1:=L_{\vec{G}}M_0 + R .
    \end{align*}
\end{lemma}
\begin{proof}
    Since $\Kk$ is an isometry due to $V_i >0$ for all $i\in I^\kappa$, we get that $R<\infty$ since the input data converges as required by Assumptions \ref{ass:numconv}.\newline
    The bound on $\vec{\b{u}}^{\kappa_n}_{\text{Red}}(t)$ is obtained using Gr\"onwall and $\|\vec{G}(\phi,t)\| \leq L_{\vec{G}} \|\phi\| + R$, which also yields the second bound as $\|\dot{\vec{\b{u}}}^{\kappa_n}_{\text{Red}}(t)\|=\|\vec{G}((\vec{\b{u}}^{\kappa_n}_{\text{Red}})_t,t)\|$. 
\end{proof}

We are now ready for the final theorem.

\begin{theorem}[Convergence of the Fully Discrete System]\label{theo:fullconv}
    Let Assumptions \ref{ass:numconv} hold. Further, let $(N_k)_{k\in \mathbb{N}}\subset \mathbb{N}$ be such that $N_k \xrightarrow[]{k\rightarrow \infty} \infty$ and $0<\kappa_k \searrow 0$ as $k\nearrow \infty$. Denote the unique solution to (\ref{eqn:conteq}) in reduced form as $\b{u}_{Red} \in C^1([0,T], L^\infty(\Omega)^d\times L^\infty(\Omega)^d)$, and $(\hat{\b{u}}_{\kappa_k}^{n, N_k})_{0\leq n \leq N_k}$ as given by (\ref{eqn:discrprob}) represented as in (\ref{eqn:redfulldiscrsol}). Then, we have
    \begin{equation*}
        \lim_{k\rightarrow \infty} \max_{0\leq n \leq N_k} \| \b{u}_{Red}(t_n^{N_k}) - \Kkf{k}^2\hat{\b{u}}_{\kappa_k}^{n, N_k}\|_{L^\infty(\Omega)^d\times L^\infty(\Omega)^d} = 0.
    \end{equation*}
    Further, if $\b{b}$ is Lipschitz continuous in $t$, as well as 
    \begin{equation}\label{eqn:assreginputdataconv}
    \begin{aligned}
        &\|\Kkf{k}\vec{\b{b}}_{\kappa_k} - \b{b}\|_{C^0([0,T], L^\infty(\Omega)^d)} \in O(\kappa_k),\ \|u_0 - \Kkf{k}\vec{u}_0\|_{L^\infty(\Omega)^d} \in O(\kappa_k),\\
        &\hspace{3.1cm}\|v_0 - \Kkf{k}\vec{v}_0\|_{L^\infty(\Omega)^d} \in O(\kappa_k)
    \end{aligned}
    \end{equation}
    then we obtain
    \begin{equation*}
       \max_{0\leq n \leq N_k} \| \b{u}_{Red}(t_n^{N_k}) - \Kkf{k}^2\hat{\b{u}}_{\kappa_k}^{n, N_k}\|_{L^\infty(\Omega)^d\times L^\infty(\Omega)^d} \in O(1/N_k+ \kappa_k).
    \end{equation*}\end{theorem}

\begin{proof}
    Define $h:= T/N_k$. We assume without loss of generality that $h<1$. Using the convergence of the input data as assumed in Assumptions \ref{ass:numconv}, we obtain $\kappa',\delta'>0$ for every $\epsilon>0$ such that
\begin{align}
    &\|\vec{\b{b}}_\kappa(t)-\vec{\b{b}}_\kappa(t')\|_\infty \leq \|\Kk\vec{\b{b}}_\kappa(t)-\b{b}(t)\| +\|\b{b}(t)-\b{b}(t')\| +\|\Kk\vec{\b{b}}_\kappa(t')-\b{b}(t')\| \leq \epsilon,\nonumber\\
    &\|\b{u}_{Red}(0) - (\Kk\vec{{u}}_{0}^{\kappa}, \Kk\vec{{v}}_{0}^{\kappa})\|\leq \epsilon, \ \ \  \|\Kk\vec{\b{b}}_\kappa(0)- \b{b}(0)\| \leq \epsilon\label{eqn:inputappr}
\end{align}
for all $0<\kappa < \kappa'$ and $|t'-t| \leq \delta'$. \newline
If (\ref{eqn:assreginputdataconv}) holds, we obtain stronger convergence result,
\begin{equation} \label{eqn:sharperboundsinputdata}
    \begin{aligned}
        &\|\vec{\b{b}}_\kappa(t)-\vec{\b{b}}_\kappa(t')\|_\infty \leq \|\Kk\vec{\b{b}}_\kappa(t)-\b{b}(t)\| +\|\b{b}(t)-\b{b}(t')\| +\|\Kk\vec{\b{b}}_\kappa(t')-\b{b}(t')\|\in O(h+\kappa),\\
        &\|\b{u}_{Red}(0) - (\Kk\vec{{u}}_{0}^{\kappa}, \Kk\vec{{v}}_{0}^{\kappa})\|\in O(\kappa), \ \ \  \|\Kk\vec{\b{b}}_\kappa(0)- \b{b}(0)\| \in O(\kappa)
    \end{aligned}
\end{equation}
The Lipschitz continuity of $\vec{\b{u}}_{\kappa, Red}$ implies that there is a $M^N_\kappa\geq 0$ for all $0<\kappa$ and $0<N$ such that 
\begin{equation}\label{eqn:IPred}
\begin{aligned}
    &\|\IP_{n}((\hat{\b{u}}^{0,...n-1,N}_\kappa)_1) - ((\vec{\b{u}}_{Red}^\kappa)_1)_{t_{n-1}}\|_{\Linftyptw([-T,0], (\mathbb{R}^d)^{I^\kappa})} \\
    &\leq \max_{1\leq k \leq n}\|\IP_{n}((\hat{\b{u}}^{0,...n-1,N}_\kappa)_1) - ((\vec{\b{u}}_{Red}^\kappa)_1)_{t_{n-1}}\|_{\Linftyptw([-hk,-h(k-1)], (\mathbb{R}^d)^{I^\kappa})} \\
    &\leq \max_{0\leq k\leq n-1} \|e_k\|_\infty + \sup_{t\in [0,T]} \sup_{|t-t'|<h}\|\vec{\b{u}}_{Red}^\kappa(t) - \vec{\b{u}}_{Red}^\kappa(t')\| \\
    &\leq M^N_\kappa + \max_{0\leq k \leq n-1} \|e_{k}\|_\infty.
\end{aligned}
\end{equation}
and, with Lemma \ref{lemm:aprioriboundssemidiscrete},
\begin{equation}\label{eqn:Tconv}
    O(h) \ni M^{N_k}_{\kappa_k} \xrightarrow[]{k\rightarrow \infty} 0.
\end{equation}
We further define
\begin{equation*}
    K^{N}_\kappa := \max_{1\leq k \leq N} \|K^{N,k}_\kappa\|_\infty.
\end{equation*}
We now aim to show that $K^N_\kappa$ converges to zero.
To this end, we have for all $t \in [t_{n-1}, t_n]$, $h<\delta'$
\begin{align*}
    &\left\|\begin{pmatrix}
          \dot{\vec{\b{u}}}_{}^\kappa(t)\\
        \vec{F}_\kappa((\vec{\b{u}}^{\kappa})_{t}) + \vec{\b{b}}_\kappa(t) \end{pmatrix} -  {\hat{G}}_h((\vec{\b{u}}^\kappa_{Red})_{t_n}, t_{n-1})\right\|\\
     &=\left\|\begin{pmatrix}
          \dot{\vec{\b{u}}}_{}^\kappa(t)\\
        \vec{F}_\kappa((\vec{\b{u}}^{\kappa})_{t}) + \vec{\b{b}}_\kappa(t)
    \end{pmatrix} - \begin{pmatrix}
        \dot{\vec{\b{u}}}_{}^\kappa(t_{n})\\
        \vec{F}_\kappa((\vec{\b{u}}^{\kappa})_{t_{n-1}}) + \vec{\b{b}}_\kappa(t_{n-1}) 
    \end{pmatrix} \right\|_{\infty}\\
    &\leq \|\dot{\b{u}}_{}^\kappa(t)-\dot{\b{u}}_{}^\kappa(t_{n})\|_{\infty} + \|\vec{F}((\b{u}_{\kappa})_{t})- \vec{F}((\b{u}_{\kappa})_{t_{n-1}}) \|_{\infty} + \|\vec{\b{b}}_\kappa(t)-\vec{\b{b}}_\kappa(t_{n-1})\|_{\infty}\\
    &\leq M_1h+ {L_{\vec{F}}}M_1h + \epsilon
\end{align*}
which implies that
\begin{align*}
    \|K^{N,n}_\kappa\|_\infty &= \left \|\frac{1}{h} \int_{t_{n-1}}^{t_n}\begin{pmatrix}
          \dot{\vec{\b{u}}}^{\kappa}(t)\\
        \vec{F}((\vec{\b{u}}^{\kappa})_{t}) + \vec{\b{b}}_\kappa(t)
    \end{pmatrix} - \begin{pmatrix}
        \dot{\vec{\b{u}}}^{\kappa}(t_{n})\\
        \vec{F}((\vec{\b{u}}^{\kappa})_{t_{n-1}}) + \vec{\b{b}}_\kappa(t_{n-1}) 
    \end{pmatrix}  \dd t \right\|_\infty \\
    &\leq M_1h+ {L_{\vec{F}}}M_1h + \epsilon
\end{align*}
This yields
\begin{equation}\label{eqn:KNconv}
    K^{N_k}_{\kappa_k} \xrightarrow[]{k\rightarrow \infty} 0.
\end{equation}
Note that if (\ref{eqn:assreginputdataconv}) holds, we obtain using (\ref{eqn:sharperboundsinputdata}) that $K^{N_k}_{\kappa_k} \in O(h + \kappa)$. 
We denote the discretization error in $t^{N_k}_n$ as defined in (\ref{eqn:defediscr}) for $N_k$ and $\kappa_k$ as $e^k_n :=  \vec{\b{u}}_{Red}^{\kappa_k}(t_n^{N_k}) - \hat{\b{u}}_{\kappa_k}^{n,N_k}$.
We also note that, using (\ref{eqn:inputappr}) and Proposition \ref{prop:LScontdiscrrhs},
\begin{align*}
    \|e_0^k\|_\infty &\leq \|\vec{u}_{0}^\kappa -\vec{\b{u}}^\kappa(0)\|_\infty + \Big\|\tfrac{\vec{\b{u}}_\kappa^{0,N}-\vec{\b{u}}_\kappa^{-1,N}}{h} -\dot{\vec{\b{u}}}^\kappa(0) \Big\|_\infty \\
    &\leq \|\tfrac{1}{2}h(\vec{F}(\IP_1(\vec{u}_{0}^\kappa))+ \vec{\b{b}}_\kappa(0))\|_\infty \leq \tfrac{1}{2}h (L_{\vec{F}}M_0 + R)
\end{align*} 
This yields
\begin{equation}\label{eqn:e0conv}
    O(h) \ni \|e_0^k\|_\infty \xrightarrow[]{k\rightarrow \infty} 0.
\end{equation}
We now define the final error bound sequence $a_n^k$ such that $a_0^k = \|e_0^k\|_\infty$ and for $n\geq 1$, 
\begin{equation}\label{eqn:defan}
    a_n^k := \frac{1}{1-h}\left(hK^{N_k}_{\kappa_k} + a_{n-1} + hL_{\vec{F}}\left(M^{N_k}_{\kappa_k} + \max_{0\leq s \leq n-1} a_{s}\right)\right).
\end{equation}
Using Lemma \ref{lemm:stabineq} and (\ref{eqn:IPred}), we see that
\begin{equation*}
\|e_n\|_{\infty} \leq a_n^k
\end{equation*}
for all $k$ and $0 \leq n\leq N$.

Since all factors on the right hand side of (\ref{eqn:defan}) are non-negative (we assume $0<h<1$) and $a_0^k\geq 0$, we conclude that $a_n^k \geq a_{n-1}^k$ for $n\geq 1$, therefore the above equation is equivalent to
\begin{equation}\label{eqn:finrecdiscr}
    a_n^k = \frac{1}{1-h}\left(hK^{N_k}_{\kappa_k} +hL_{\vec{F}} M^{N_k}_{\kappa_k} + (1+hL_{\vec{F}})a_{n-1}^k\right).
\end{equation}
For $a^* := - \frac{K^{N_k}_{\kappa_k} +L_{\vec{F}}M^{N_k}_{\kappa_k}}{(1+L_{\vec{F}})}$, we have that
\begin{equation*}
    a^* = \frac{1}{1-h}\left(hK^{N_k}_{\kappa_k} +hL_{\vec{F}} M^{N_k}_{\kappa_k} + (1+hL_{\vec{F}})a^*\right),
\end{equation*}
as well as that $a^{hom}_n := C_k \left(\frac{1+hL_{\vec{F}}}{1-h}\right)^n$ for $n\geq 0$, fulfills
\begin{equation*}
    a_n^{hom} = \frac{1}{1-h}\left((1+hL_{\vec{F}})a_{n-1}^{hom}\right)
\end{equation*}
for all $C_k\in \mathbb{R}$ and $n\geq 1$. Therefore $a_n^{hom} + a^*$ parameterizes all solutions to (\ref{eqn:finrecdiscr}) over $C_k\in \mathbb{R}$.   Given $a_0^k = \|e_0^k\|_\infty$, we then must have 
\begin{equation*}
    C_k = a_0^k - a^*.
\end{equation*}
An explicit formula for the bounds $a_n^k$ is therefore given for all $0 \leq n\leq N^k$ by 
\begin{equation*}
    a_n^k = \left( \|e_0^k\|_\infty+  \frac{K^{N_k}_{\kappa_k} +L_{\vec{F}}M^{N_k}_{\kappa_k}}{(1+L_{\vec{F}})}\right) \left(\frac{1+hL_{\vec{F}}}{1-h}\right)^n - \frac{K^{N_k}_{\kappa_k} +L_{\vec{F}}M^{N_k}_{\kappa_k}}{(1+L_{\vec{F}})}.
\end{equation*}
We then finally have
\begin{equation}\label{eqn:expbound}
    \left(\frac{1+hL_{\vec{F}}}{1-h}\right)^n \leq \left(\frac{1+\tfrac{T}{N_k}L_{\vec{F}}}{1-\tfrac{T}{N_k}}\right)^{N_k} = \left(1+\tfrac{T}{N_k}L_{\vec{F}}\right)^{N_k}\left(1-\tfrac{T}{N_k}\right)^{-N_k}\rightarrow \exp(T(1+L_{\vec{F}}))< \infty
\end{equation}

This finally yields, for all $k\in\mathbb{N}$,
\begin{equation}\label{eqn:convindiscr}
    \begin{aligned}
        &\max_{0\leq n \leq N_k}\|\vec{\b{u}}_{Red}^{\kappa_k}(t_n^{N_k}) - \hat{\b{u}}_{\kappa_k}^{n,N_k}\|_\infty \\
        &\leq \left( \|e_0^k\|_\infty+  \frac{K^{N_k}_{\kappa_k} +L_{\vec{F}}M^{N_k}_{\kappa_k}}{(1+L_{\vec{F}})}\right)\left(\frac{1+\tfrac{T}{N^k}L_{\vec{F}}}{1-\tfrac{T}{N_k}}\right)^{N^k} - \frac{K^{N_k}_{\kappa_k} +L_{\vec{F}}M^{N_k}_{\kappa_k}}{(1+L_{\vec{F}})} \\
        &\quad \xrightarrow[k\nearrow \infty]{(\ref{eqn:expbound}), (\ref{eqn:KNconv}), (\ref{eqn:Tconv}), (\ref{eqn:e0conv})} 0
    \end{aligned}
\end{equation}
Again, if (\ref{eqn:assreginputdataconv}) holds, we further obtain
\begin{align*}
     &\left( \|e_0^k\|_\infty+  \frac{K^{N_k}_{\kappa_k} +L_{\vec{F}}M^{N_k}_{\kappa_k}}{(1+L_{\vec{F}})}\right)\left(\frac{1+\tfrac{T}{N^k}L_{\vec{F}}}{1-\tfrac{T}{N_k}}\right)^{N^k} - \frac{K^{N_k}_{\kappa_k} +L_{\vec{F}}M^{N_k}_{\kappa_k}}{(1+L_{\vec{F}})} \in O(\kappa +h).
\end{align*}
Finally, we have for all $0\leq n \leq N^k$ and all $\kappa_k$, that 
\begin{align*}
    &\max_{0\leq n \leq N_k}\|\b{u}_{Red}(t_n^{N_k}) - \Kk^2\hat{\b{u}}^{n,N_k}_{\kappa_k}\|_{L^\infty(\Omega)^d}\\
    &\leq \|\b{u}_{Red} -  \Kk^2\vec{\b{u}}^{\kappa_k}_{Red}\|_{C^0([0,T], L^\infty(\Omega)^d)} +\max_{0\leq n \leq N_k} \|\vec{\b{u}}_{Red}^{\kappa_k}(t_n^{N_k}) - \hat{\b{u}}_{\kappa_k}^{n,N_k}\|_\infty, 
\end{align*}
which yields the claim via Theorem \ref{theo:convsemidiscr} and (\ref{eqn:convindiscr}). If (\ref{eqn:assreginputdataconv}) holds, Theorem \ref{theo:convsemidiscr} gives us a rate for the first term, resulting in
\begin{align*}
        &\|\b{u}_{Red} -  \Kk^2\vec{\b{u}}^{\kappa_k}_{Red}\|_{C^0([0,T], L^\infty(\Omega)^d)} +\max_{0\leq n \leq N_k} \|\vec{\b{u}}_{Red}^{\kappa_k}(t_n^{N_k}) - \hat{\b{u}}_{\kappa_k}^{n,N_k}\|_\infty \in O(\kappa) + O(\kappa + h)
\end{align*}
\end{proof}

One quantity often of interest in peridynamic simulations is the integrity function given by
\begin{align}\label{eqn:damagefunc}
    [0,1]\ni \gamma_{\b{u}}(x, t) = \frac{1}{\vol(\Omega \cap B_{\delta}(x))}\int_{\Omega \cap B_{\delta}(x)} 
    \mu(\esssup_{0\leq s\leq t} \mathcal{S}(\b{u}(s), x',x)) \dd x' ,
\end{align}
or equivalently, the damage function $d_{\b{u}}(x,t) = 1 - \gamma_{\b{u}}(x,t)$. 

For fully discrete solutions $\vec{\b{u}}^{0,...,N,N}_\kappa$, one can define a natural equivalent given by $\gamma_{\vec{\b{u}}}^{n, N} \in \mathbb{R}^{I^\kappa}$ given by
\begin{align}\label{eqn:damagefuncdiscr}
    [0,1]\ni(\vec{\gamma}_{\vec{\b{u}}}^{n, N})_i = \Big(\sum_{i\not=j} w_{ij}V_j\Big)^{-1}\sum_{i\not=j} w_{ij} V_j 
    \mu(\max_{0\leq k\leq n} \mathcal{S}^{ij}(\vec{\b{u}}^{k,N}_{\kappa})),
\end{align}
with the partial area/volume weights $w_{ij}$ as in Remark \ref{rm:partvolumeweights} or simply $w_{ij}= \chi_{[0,\delta]}(\|x_i-x_j\|_2)$.\par

From Theorem \ref{theo:fullconv} above, we obtain the following Corollary.
\begin{corollary}[Convergence of the Integrity Function]
    Let the Assumptions \ref{ass:numconv}, \ref{ass:quadweights} hold. We also require $\Omega$ to satisfy an interior cone condition. Further, let $(N_k)_{k\in \mathbb{N}}\subset \mathbb{N}$ be such that $N_k \xrightarrow[]{k\rightarrow \infty} \infty$ and $0<\kappa_k \searrow 0$ as $k\nearrow \infty$. Denote the unique solution to (\ref{eqn:conteq}) as $\b{u}_{} \in C^2([0,T], L^\infty(\Omega)^d)$, and $(\vec{\b{u}}_{\kappa_k}^{n, N_k})_{0\leq n \leq N_k}$ as given by (\ref{eqn:discrprob}). Then, we obtain
    \begin{equation*}
        \lim_{k\rightarrow \infty} \sup_{0\leq l \leq N_k} \|\gamma_{\b{u}}(\cdot, t_l)  - \Kk'\vec{\gamma}_{\vec{\b{u}}_{\kappa_k}}^{l, N_k}\|_{L^\infty(\Omega)} = 0. 
    \end{equation*}
    with $\Kk': \mathbb{R}^{I^\kappa} \rightarrow L^\infty(\Omega)$ being the scalar equivalent of $\Kk$ given by $\vec{u} \mapsto \sum_{i} \chi_{\Omega_i}\vec{u}_i$. \par
    If $\b{b}$ is Lipschitz continuous in $t$ and the input data fulfills (\ref{eqn:assreginputdataconv}), then
    \begin{align*}
         \sup_{0\leq l \leq N_k} \|\gamma_{\b{u}}(\cdot, t_l)  - \Kk'\vec{\gamma}_{\vec{\b{u}}_{\kappa_k}}^{l, N_k}\|_{L^\infty(\Omega)}  \in O(1/N_k+\kappa_k).
    \end{align*}
\end{corollary}
\begin{proof}
    We drop the $k$ subscript from $\kappa, N$ for convenience. First, we see that $\vol(\Omega \cap B_{\delta}(x)) \leq \vol(B_\delta(0))$ is bounded from above as well as from below for some $0<c_{ICC}< \vol(\Omega \cap B_{\delta}(x))$ uniformly in $x$, 
    given by the interior cone condition of $\Omega$.\newline
    Since we also have
    \begin{align*}
        |\vol(\Omega \cap B_{\delta}(x)) - \sum_{i\not=j} w_{ij}V_j| \leq \vol(B_{\delta + 3\sqrt{d}\kappa/2}(0)) - \vol(B_{\delta - 3\sqrt{d}\kappa/2}(0)) + \kappa^d \in O(\kappa),
    \end{align*}
    supplying lower bounds for $\sum_{i\not=j} w_{ij}V_j$ in the limit, we obtain
    \begin{equation}\label{eqn:convvoldiff}
        \left|\frac{1}{\vol(\Omega \cap B_{\delta}(x))} - \Big(\sum_{i\not=j} w_{ij}V_j\Big)^{-1}\right| \in O(\kappa)
    \end{equation}
    We then obtain for all $i\in I^\kappa$ and $x\in \Omega_i$ and $0 \leq n \leq N$,
    \begin{equation}
        \begin{aligned}
            &c_{ICC}|\gamma_{\b{u}}(x, t_n)  - (\vec{\gamma}_{\vec{\b{u}}_\kappa}^{n, N})_i|\nonumber\\
            &\leq c_{ICC} \left|\frac{1}{\vol(\Omega \cap B_{\delta}(x))} - \Big(\sum_{i\not=j} w_{ij}V_j\Big)^{-1}\right|\Big(\sum_{i\not=j} w_{ij}V_j\Big) \\
            &+\Big|\int_{\Omega \cap B_{\delta}(x)} \mu(\esssup_{0\leq s\leq t_n} \mathcal{S}(\b{u}(s), x',x)) \dd x' - \sum_{i\not=j} w_{ij} V_j \mu(\max_{0\leq l\leq n} \mathcal{S}^{ij}(\vec{\b{u}}^{l,N}_{\kappa}))\Big| \nonumber\\
        \end{aligned}
    \end{equation}
    Since the $\sum_{j\not=i} V_j w_{ij}$ terms converge as $\kappa \searrow 0$ by Assumption \ref{ass:quadweights}, we have with (\ref{eqn:convvoldiff}) that the first summand vanishes in $O(\kappa)$. We now focus on the second, bounding it by 
    \begin{align}
        &\Big|\sum_{i\not=j} \int_{\Omega_j}  \chi_{[0,\delta]}(\|x-x'\|_2)\Big(\mu(\esssup_{-T\leq s\leq 0} \mathcal{S}(\b{u}_{t_n}(s), x',x)) \nonumber \\
        &\phantom{==============} - \mu(\esssup_{-T\leq s\leq 0} \mathcal{S}^{ij}(\IP^{n+1}(\vec{\b{u}}^{0,...,n,N}_{\kappa})(s)))\Big) \dd x' \Big| \label{eqn:firstsummandconvinteg}\\
        &\quad+ \sum_{i\not=j} \int_{\Omega_j} |w_{ij} - \chi_{[0,\delta]}(\|x-x'\|_2)|\mu(\max_{0\leq l\leq n} \mathcal{S}^{ij}(\vec{\b{u}}^{l,N}_{\kappa}))\dd x'\nonumber + \kappa^d
    \end{align}
    where the $\kappa^d$ factor comes from the $i=j$ cell of the continuous integral. The last summand can be bounded by 
    \begin{align*}
        &\sum_{i\not=j} \int_{\Omega_j} |w_{ij} - \chi_{[0,\delta]}(\|x-x'\|_2)|\dd x'\leq \vol(B_{\delta + 3\sqrt{d}\kappa/2}(0)) - \vol(B_{\delta - 3\sqrt{d}\kappa/2}(0)) \in O(\kappa)
    \end{align*}
    
    Turning our attention to the first summand in (\ref{eqn:firstsummandconvinteg}), we obtain using Lemma \ref{lemm:auxbnd} applicable to the $\IP$ term since all of its values are taken on sets of measure greater $0$,
    \begin{align*}
        &\Big|\sum_{i\not=j} \int_{\Omega_j} \chi_{[0,\delta]}(\|x-x'\|_2)\Big(\mu(\esssup_{-T\leq s\leq 0} \mathcal{S}(\b{u}_{t_n}(s), x',x))  - \mu(\esssup_{-T\leq s\leq 0} \mathcal{S}^{ij}(\IP^{n+1}(\vec{\b{u}}^{0,...,n,N}_{\kappa})(s))) \dd x'\Big) \Big| \\
        &\leq  \Big| \sum_{i\not=j} \int_{\Omega_j}  \Big((\sqrt{d}+1)\tfrac{\sqrt{d}\ell\chi_{[0,\delta]}(\|x-x'\|_2)}{\|x'-x\|_2}\Big)( 2\|\IP^{n+1}(\vec{\b{u}}^{0,...,n,N}_{\kappa})-\b{u}_{t_n}\|_{\Linftyptw([-T,0],L^\infty(\Omega)^d)}\\
        &\qquad+ (S^+_1+2)\sqrt{d}\kappa) \dd x' \Big| \\
        &\leq \Big[(S^+_1+2)(\sqrt{d}+1)d\ell\kappa + 2(\sqrt{d}+1)\sqrt{d}\ell \|\IP^{n+1}(\vec{\b{u}}^{0,...,n,N}_{\kappa})-\b{u}_{t_n}\|_{
    \Linftyptw([-T,0],L^\infty(\Omega)^d)} \Big] \\
        &\quad \times\int_{B_{\delta}(0)} \|\zeta\|_2^{-1} \dd \zeta
    \end{align*}
    Due to $d\geq 2$ the last integral over $\|\zeta\|_2^{-1}$ is finite, as well as
     \begin{align*}
          &\|\IP^{n+1}(\vec{\b{u}}^{0,...,n,N}_{\kappa})-\b{u}_{t_n}\|_{\Linftyptw([-T,0],L^\infty(\Omega)^d)}\\
          &\leq Lh + \max_{0 \leq l \leq n} \| \b{u}(t_l) - \Kk \vec{\b{u}}^{l,N}_\kappa \|_{L^\infty(\Omega)^d}
     \end{align*}
    for the Lipschitz constant $L= \|\dot{\b{u}}\|_{\Linftyptw([0,T], L^\infty(\Omega)^d)} < \infty$ and $h:= T/N$. 
    The result of Theorem \ref{theo:fullconv} now applies.
\end{proof}

\begin{remark}[About the Assumptions, Part II]\label{rm:abouttheassumptionspt2} In order to circumvent the effect of the strong convergence requirements in Assumptions \ref{ass:numconv} as well as to use a more practical norm to measure the convergence, attempts to reformulate the equations on an $L^p(\Omega)^d$ space with $1\leq p<\infty$ are naturally interesting. However, these run into problems due to the $\esssup$ term in the argument of $\mu$, as the spatial $L^p(\Omega)^d$ norm cannot be swapped with the $L^\infty([0,t])$ norm that the essential supremum represents, even for jointly measurable representatives. This causes the Lipschitz continuity proof as given above for the peridynamic operator $F$ to fail. \par
    A mathematically natural approach to circumvent this would be to equivalently rewrite
    \begin{equation*}
        \mu(\esssup_{0\leq s\leq t} \ca{S}(\b{u}(s), x', x))
    \end{equation*}
    as, using $p=\infty$ and the symbolic convention $1/\infty := 0$ for the exponent,
    \begin{equation*}
        \mu(S^+_0 + \tfrac{1}{t_c^{1/p}}\|\max(0, \ca{S}(\b{u}(s),x',x)-S^+_0)\|_{L^p([0,t])}) 
    \end{equation*}
    which naturally also allows for $p<\infty$. The $t_c>0$ parameter is an additional material parameter included to get rid of the $s^{1/p}$ unit of the $\|\cdot\|_{L^p}$ norm due to the integral over $t$ in it and has the same physical unit as the time $t$. For $p=1$ and $t_c=1$ this would yield the damage term as described in \cite{Emmrich2016} and similar to \cite{Lipton2018}, which also uses a time integral. Note that the $1<p<\infty$ then transitions between the damage term in \cite{Emmrich2016} corresponding to $p=1$ and the one in \cite{WP_FRAC, SillMeshFree} corresponding to $p=\infty$. Additionally, the models for $p < \infty$ introduce a time-dependence into how the models accumulate damage, as the maximum is replaced with an integral accumulating stretch over the threshold $S^+_0$ over time. That is, minimally crossing the threshold $S^+_0$ for long enough will result in large damage and, conversely, a large stretch over a sufficiently short time span will not result in significant damage. As noted in \cite{Emmrich2016}, this allows modeling fatigue through cyclic loading, as each cycle adds to the accumulated stretch even if the maximum stretch is left unaffected. The parameter $t_c>0$ can be understood to model the sensitivity to strain, since it is scaling the accumulated damage accordingly. \par
    With this change one could attempt to establish similar convergence results with the assumptions and convergence formulated in the $L^p(\Omega)^d$ norm. Rigorous treatment of these models is however subject to further research. 
\end{remark}

\section{Properties of the Continuous Solution}
\label{sec:PropContSol}
The convergence of the piecewise constant discretization with respect to the $C^0([0,T],L^\infty(\Omega)^d)$ norm naturally implies that the continuous solution must have a certain spatial regularity that we are now aiming to prove explicitly. A similar regularity result for a different model was given in \cite{Emmrich2016}, Theorem 2.2, where global spatial continuity of the input data was used to show global continuity of the motion on $\overline{\Omega}$. \par
To talk about the properties of the displacement in specific points $x$, we need to repose the equations (\ref{eqn:conteq}) on a space that permits point-wise evaluations.
Doing so, repeating the arguments from the $L^\infty(\Omega)^d$ case as described in \cite{WP_FRAC}, we obtain the well-posedness of (\ref{eqn:conteq}) on $C^2([0,T], L^{\infty}_{\mathrm{ptw}}(\Omega; \mathbb{R}^d))$. 

\begin{assumptions}[Assumptions for the Pointwise Well-posedness]\label{ass:contexptw}
    \begin{enumerate}
    \item A horizon $\delta>0$ and a non-increasing measurable weighting function $\omega_\delta: (0,\infty) \rightarrow \mathbb{R}_{\geq 0}$, fulfilling
    \begin{equation}
        \ca{L} :=\int_{\mathbb{R}^d} \frac{\omega_\delta(\|\zeta\|_2)}{\|\zeta\|_2}\dd \zeta<\infty.
    \end{equation}
    \item The external body forces are continuous in time,  $\b{b}\in C^0([0, T], \Linftyptw(\Omega)^d)$ 
    \item The initial data is pointwise $u_0,v_0 \in \Linftyptw(\Omega)^d$
    \item A terminal time $0 <T <\infty$.
    \end{enumerate}
\end{assumptions}

\begin{lemma}[Well-posedness of the Pointwise Equations]
    The right hand side of (\ref{eqn:conteq}) as given in (\ref{eqn:defcontF}) can be redefined as $F^{\text{ptw}}: C^0([-T,0], \Linftyptw(\Omega)^d) \rightarrow \Linftyptw(\Omega)^d$ alongside with $\b{f}$ from (\ref{eqn:defcontfinteg}) to take displacements in $C^0([-T,0], \Linftyptw(\Omega)^d)$, without any changes besides removing the representation map $\Gamma$ since the functions at a time point $t$ are no longer equivalence classes.\par
    Let the Assumptions \ref{ass:contexptw} hold. Then, the corresponding equations to (\ref{eqn:conteq}) posed on \newline $ C^0([-T,0], \Linftyptw(\Omega)^d)$, given by
    \begin{equation}
        \label{eqn:conteqptw}
        \begin{rcases}
            \rho\ddot{\b{u}}_{\text{ptw}}(t) = F^{\text{ptw}}((\b{u}_{\text{ptw}})_t) + \b{b}(t)\\ 
            \b{u}_{\text{ptw}}(0) = u_0 \\
            \dot{\b{u}}_{\text{ptw}}(0) = v_0
        \end{rcases}
    \end{equation} 
    have a unique solution $\b{u}_{\text{ptw}} \in C^2([0,T], \Linftyptw(\Omega)^d)$.
\end{lemma}
\begin{proof}
    The pointwise map $F^{\text{ptw}}$ can be shown to be Lipschitz continuous for some $L_{F^{\text{ptw}}}>0$, 
    \begin{equation*}
        \|F^{\text{ptw}}(\b{u}) - F^{\text{ptw}}(\b{v})\|_{\Linftyptw(\Omega)^d} \leq L_{F^{\text{ptw}}}\|\b{u}-\b{v}\|_{C^0([-T,0], \Linftyptw(\Omega)^d)}
    \end{equation*}
     for all $\b{u},\b{v}\in C^0([-T,0],\Linftyptw(\Omega)^d)$, due to 
     \begin{equation*}
         \esssup_{-T\leq t\leq 0} \|\b{u}(t,x)\|_\infty =  \sup_{-T\leq t\leq 0} \|\b{u}(t,x)\|_\infty
     \end{equation*}
     for all $x\in \Omega$, as well as
     \begin{equation*}
         \sup_{-T\leq t\leq 0} \sup_{x\in\Omega}\|\b{u}(t,x)\|_\infty  = \sup_{x\in\Omega} \sup_{-T\leq t\leq 0}\|\b{u}(t,x)\|_\infty. 
     \end{equation*}
     This is analogous to \cite{WP_FRAC}, Lemma 1.
     The rest of the proof is then analogous to \cite{WP_FRAC}, Theorem 2, obtaining a unique solution using a Picard iteration argument.
\end{proof}
Let $\iota: \Linftyptw(\Omega)^d \rightarrow L^\infty(\Omega)^d$ denote the canonical surjection, we then easily obtain
\begin{equation}\label{eqn:corresp}
    \iota \circ F^{\text{ptw}} = F \circ {\iota}_* 
\end{equation}
on $C^0([-T,0], \Linftyptw(\Omega)^d)$, where ${\iota}_*$ is the post-composition map $\iota_*: C^0([-T, 0], \Linftyptw(\Omega)^d) \rightarrow C^0([-T, 0], L^\infty(\Omega)^d),\ \b{a} \mapsto \iota \circ \b{a}$. This map is also linear and bounded.
Combining these identities, this implies that the solution $\b{u}_{\mathrm{ptw}} \in C^2([0,T]; L^{\infty}_{\mathrm{ptw}}(\Omega)^d)$ of (\ref{eqn:conteqptw}) represents the original solution $\b{u} \in C^2([0,T]; L^{\infty}(\Omega)^d)$, for every $t\in [0,T]$. That is, $\b{u}_{\mathrm{ptw}}(t) \in \Meas(\Omega, \mathbb{R}^d)$ is a representation of $\b{u}(t)\in L^\infty(\Omega)^d$ for all $t\in [0,T]$. \par
We will now study the properties of the solution $\b{u}_{\text{ptw}}$, working towards Proposition \ref{prop:contcontsol} below. First, we obtain for fixed $x, \hat{x} \in \Omega$,
\begin{align*}
    &\b{u}_{\mathrm{ptw}}(t,x) - \b{u}_{\mathrm{ptw}}(t,\hat{x})\\
    &= u_0(x) - u_0(\hat{x})+ \int_{0}^t v_0(x) - v_0(\hat{x})\\
    &\phantom{= u_0(x) - u_0(\hat{x})+ \int_{0}^t}+\int_{0}^s F^{\text{ptw}}( (\b{u}_{\mathrm{ptw}})_\tau )( x) - F^{\text{ptw}}( (\b{u}_{\mathrm{ptw}})_\tau )( \hat{x}) + \b{b}(\tau, x) - \b{b}(\tau, \hat{x}) \dd \tau \dd s.
\end{align*}
Similarly to the discrete case, we also derive
\begin{enumerate}
    \item a bound on differences of the $\ca{S}$, as given by
    \begin{align}\label{eqn:bndcontcontS}
        |\ca{S}(u, x',x)-\ca{S}(u, x',\hat{x})| \leq \begin{cases}
            \tfrac{\sqrt{d}}{\|x'-x\|_2}\Big( \|u(x)-u(\hat{x})\|_\infty + (\ca{S}(u, x',\hat{x})+2)\|x-\hat{x}\|_\infty \Big)\\
            \tfrac{\sqrt{d}}{\|x'-\hat{x}\|_2}\Big( \|u(x)-u(\hat{x})\|_\infty + (\ca{S}(u, x',x)+2)\|x-\hat{x}\|_\infty \Big)
        \end{cases}
    \end{align} 
    \item a bound on the differences of $g$,
    \begin{align}\label{eqn:bndcontcontg}
        &|g(\ca{S}(u,x',x))-g(\ca{S}(u,x',\hat{x}))|\nonumber \\
        &\leq\begin{cases}
            0 & \ca{S}(u,x',x) >S^+_1 \cap \ca{S}(u,x',\hat{x}) > S^+_1\\
            \tfrac{\ell\sqrt{d}}{\|x'-x\|_2}\Big( \|u(x)-u(\hat{x})\|_\infty + (S^+_1+2)\|x-\hat{x}\|_\infty \Big) & \ca{S}(u,x',\hat{x}) \leq S^+_1\\
            \tfrac{\ell\sqrt{d}}{\|x'-\hat{x}\|_2}\Big( \|u(x)-u(\hat{x})\|_\infty + (S^+_1+2)\|x-\hat{x}\|_\infty \Big) &  \ca{S}(u,x',{x}) \leq S^+_1
        \end{cases}
    \end{align}
    \item a bound on the differences of $\mu$, 
    \begin{align}\label{eqn:bndcontcontMu}
        &|\underbrace{\mu(\sup_{-T\leq t \leq 0}\ca{S}(\b{u}(t),x',x))}_{=:A}-\underbrace{\mu(\sup_{-T\leq t \leq 0}\ca{S}(\b{u}(t),x',\hat{x}))}_{=:B}|\nonumber \\
        &\leq\begin{cases}
            0 & A=B=0\\
            \tfrac{\ell\sqrt{d}}{\|x'-x\|_2}\Big(\sup_{-T\leq t\leq 0}\|\b{u}(t,x)-\b{u}(t,\hat{x})\|_\infty + (S^+_1+2)\|x-\hat{x}\|_\infty \Big) & A>0\\
            \tfrac{\ell\sqrt{d}}{\|x'-\hat{x}\|_2}\Big(\sup_{-T\leq t\leq 0} \|\b{u}(t,x)-\b{u}(t,\hat{x})\|_\infty + (S^+_1+2)\|x-\hat{x}\|_\infty \Big) &  B>0
        \end{cases}
    \end{align}
    \item a bound on differences of $e$, assuming $x'-x$, $x'-\hat{x}$ are non-degenerate bonds for $u$ in $x$, given by \begin{align}\label{eqn:bndcontconte}
        |e(u, x', x) -e(u, x', \hat{x})| 
        \leq \begin{cases}
            \tfrac{2\sqrt{d}}{(\ca{S}(u, x',x)+1)\|x'-x\|_2} (\|{u}(x)-{u}(\hat{x})\|_\infty + \|x-\hat{x}\|_\infty) \\
            \tfrac{2\sqrt{d}}{(\ca{S}(u, x',\hat{x})+1)\|x'-\hat{x}\|_2} (\|{u}(x)-{u}(\hat{x})\|_\infty + \|x-\hat{x}\|_\infty).
        \end{cases}
    \end{align}
\end{enumerate}
We also need an analogue of Lemma \ref{lemm:unifbndinteg} for different base points $x$ and $\hat{x}$.

\begin{lemma}[Uniform Bound of Semi-Translated Kernel Integral]\label{lemm:unifbndintegcont} Let $x,\hat{x} \in \Omega$, then
    \begin{align}\label{eqn:unifbndcntintegtransl}
        0 \leq \int_{\Omega} \frac{\omega_\delta(\|x'-x\|_2)}{\|x'-\hat{x}\|_2}\dd x' \leq \boundkerinteg
    \end{align}
\end{lemma}
\begin{proof}
    Intuitively speaking, the integral in (\ref{eqn:unifbndcntintegtransl})  should behave more nicely than the integral in (\ref{eqn:boundcond}), as $\omega_\delta(z)$  just like $1/z$ is non-increasing for $z>0$, meaning that the largest values multiply in the integrand of (\ref{eqn:boundcond}), but are offset in (\ref{eqn:unifbndcntintegtransl}). To make this thought rigorous, we note that for any measurable set $A\subset\mathbb{R}^d$ and $\vol(A)\leq \vol(B_r(0))$ for some $r>0$ we have
    \begin{equation*}
        \int_{A} \frac{1}{\|\zeta\|_2} \dd \zeta \leq \int_{B_r(0)} \frac{1}{\|\zeta\|_2}\dd \zeta.
    \end{equation*}
    That is, the ball around the origin, which is where the singularity is, maximizes the value of the integral over all sets with a given volume $\vol(B_r(0))$. This formalizes the idea of largest values combining if applied to level sets $A=\{x' : \omega_\delta(\|\cdot - x\|_2)>s\}$ for any $s\geq 0$. To use this explicitly, we use the layer cake representation of (\ref{eqn:unifbndcntintegtransl}),  
\begin{align}\label{eqn:ineqhelpunifbndcnttr}
        \int_{\Omega} \frac{\omega_\delta(\|x'-x\|_2)}{\|x'-\hat{x}\|_2}\dd x' \leq \int_0^\infty\int_{\{x'\in\mathbb{R}^d : \omega_\delta(\|x'-x\|_2) > s\}} \frac{1}{\|x'-\hat{x}\|_2}\dd x' \dd s
    \end{align}
   We see that $\vol(\{x'\in\mathbb{R}^d : \omega_\delta(\|x'-x\|_2) > s\}) = \vol(\{x'\in\mathbb{R}^d : \omega_\delta(\|x'-\hat{x}\|_2) > s\})$. Further, the level sets $\{x'\in\mathbb{R}^d : \omega_\delta(\|x'-\hat{x}\|_2) > s\}$ are balls $B_{r_s}(0)$ for some $r_s\geq 0$ dependent on $s\geq 0$, due to the monotonicity and radial symmetry of $\omega_\delta$. With a small translation to allow the assumption $\hat{x}=0$ without loss of generality, we obtain with (\ref{eqn:ineqhelpunifbndcnttr}),
    \begin{equation*}
        \int_0^\infty\int_{\{x'\in\mathbb{R}^d : \omega_\delta(\|x'-x\|_2) > s\}} \frac{1}{\|x'-\hat{x}\|_2}\dd x' \dd s\\
        \leq \int_0^\infty\int_{\{x'\in\mathbb{R}^d : \omega_\delta(\|x'-\hat{x}\|_2) > s\}} \frac{1}{\|x'-\hat{x}\|_2}\dd x' \dd s\\
        =\mathcal{L}.
    \end{equation*}
    The last equality follows from reversing the layer cake representation again. 
\end{proof}

We use the same trick as in the proof of Proposition \ref{prop:discrrhserr} and \ref{prop:LScontdiscrrhs} to deal with degenerate or sub-critical bonds. Defining,
\begin{align*}
    e_x(u,x',\hat{x}) := \begin{cases}
        e(u,x',\hat{x}) & \text{if } \ca{S}(u,x',\hat{x}) \geq S^-_1 \\
        e(u,x',x) & \text{else.}
    \end{cases}
\end{align*}
alongside with, for $\hat{x} \in \Omega$,
\begin{align*}
    F^{\text{ptw}}_x(\b{a})(\hat{x}) = \int_\Omega \omega_\delta(\|x'-\hat{x}\|_2)\scalarlaw(\ca{S}(\b{a}(0),x',\hat{x}))\mu(\esssup_{-T\leq s\leq 0} \ca{S}(\b{a}(s),x',\hat{x})) e_x(\b{a}(0), x',\hat{x}) \dd x' 
\end{align*}
for $x\in\Omega$ and $\b{a} \in C^0([-T,0], L^\infty(\Omega)^d)$.\par
This allows us to apply the inequalities (\ref{eqn:bndcontcontS}), (\ref{eqn:bndcontcontMu}), (\ref{eqn:bndcontcontg}) and (\ref{eqn:bndcontconte}), choosing the right branch as we did in Proposition \ref{prop:LScontdiscrrhs} and \ref{prop:discrrhserr} alongside with Lemma \ref{lemm:unifbndintegcont} to derive a bound for differences of $F$ as follows,
\begin{align*}
    &\|F^{\text{ptw}}( (\b{u}_{\mathrm{ptw}})_t )(x) - F^{\text{ptw}}( (\b{u}_{\mathrm{ptw}})_t )(\hat{x})\|_\infty = \|F^{\text{ptw}}( (\b{u}_{\mathrm{ptw}})_t )(x) - F^{\text{ptw}}_x( (\b{u}_{\mathrm{ptw}})_t )(\hat{x})\|_\infty \\
    &\leq \scalarlaw^+\int_{\Omega} |\omega_\delta(\|x'-x\|_2)-\omega_\delta(\|x'-\hat{x}\|_2)| \dd x' \\
    &\quad+2\sqrt{d}\ell \boundkerinteg( \|\b{u}_{\mathrm{ptw}}(t,x)-\b{u}_{\mathrm{ptw}}(t,\hat{x})\|_\infty + (S^+_1+ 2)\|x-\hat{x}\|_\infty)\\
    &\quad+2\sqrt{d}\ell \scalarlaw^+  \boundkerinteg (\max_{0\leq s\leq t}\|\b{u}_{\mathrm{ptw}}(s,x)-\b{u}_{\mathrm{ptw}}(s,\hat{x})\|_\infty+  (S^+_1+ 2)\|x-\hat{x}\|_\infty) \\
    &\quad+ 4\sqrt{d}\tfrac{\scalarlaw^+}{(S^-_1 +1)} \boundkerinteg(\|\b{u}_{\mathrm{ptw}}(t,x)-\b{u}_{\mathrm{ptw}}(t,\hat{x})\|_\infty+ \|x-\hat{x}\|_\infty).
\end{align*}
For simplicity, define
\begin{equation}\label{eqn:defCICII}
\begin{aligned}
    C_I&:=  4\sqrt{d}\tfrac{\scalarlaw^+}{S^-_1+1}+2\sqrt{d}\ell(S^+_1+2) + 2\sqrt{d}\ell \scalarlaw^+(S^+_1+2), \\
    C_{II}&:= 4\sqrt{d}\tfrac{\scalarlaw^+}{S^-_1+1}+2\sqrt{d}\ell + 2\sqrt{d}\ell \scalarlaw^+
\end{aligned}
\end{equation}
This leads to bounds for the solution $\b{u}_{\text{ptw}}$ for constants $0<C_I, C_{II}<\infty$ independent of $t, x, \hat{x}$ and $x'$,
\begin{align*}
    &\|\b{u}_{\mathrm{ptw}}(t,x) - \b{u}_{\mathrm{ptw}}(t,\hat{x}) \|_\infty \\
    &\leq t\int_0^t\|\b{b}(s,x) - \b{b}(s,\hat{x})\|_{\infty} \dd s + t \|v_0(x) - v_0 (\hat{x})\|_\infty+ \|u_0(x) - u_0 (\hat{x})\|_\infty\\
    &\quad+\tfrac{1}{2}t^2\scalarlaw^+\int_{\Omega} |\omega_\delta(\|x'-x\|_2)-\omega_\delta(\|x'-\hat{x}\|_2)| \dd x'  + \tfrac{1}{2}t^2 C_I\boundkerinteg \|x-\hat{x}\|_\infty\\
    &\quad+ \int_0^t (t-s)C_{II} \boundkerinteg\max_{0\leq \tau\leq s} \|\b{u}_{\mathrm{ptw}}(\tau,x)-\b{u}_{\mathrm{ptw}}(\tau,\hat{x})\|_\infty \dd s
\end{align*}
where in the last integral we used Fubini to rewrite the double integral, allowing us to replace the inner integral over $\tau$ with a factor $t-s$.
With the generalized Gr\"onwall inequality from \cite{YE20071075}, Corollary 2 with $\beta =2$, we obtain the following growth bounds, 
\begin{equation}\label{eqn:fingroenregcontsol}
\begin{aligned}
    &\|\b{u}_{\mathrm{ptw}}(t,x) - \b{u}_{\mathrm{ptw}}(t,\hat{x}) \|  \\
    &\leq\Big(t\int_0^t\|b(s,x) - b(s,\hat{x})\|_{\infty} \dd s + t \|v_0(x) - v_0 (\hat{x})\|_\infty + \|u_0(x) - u_0 (\hat{x})\|_\infty\\
    &\quad +\tfrac{1}{2}t^2\scalarlaw^+\int_{\Omega} |\omega_\delta(\|x'-x\|)-\omega_\delta(\|x'-\hat{x}\|)| \dd x' +\tfrac{1}{2}t^2 C_{I} \boundkerinteg\|x-\hat{x}\|_\infty)\Big) \\
    &\qquad\times\cosh(t\sqrt{C_{II}\boundkerinteg}).
\end{aligned}
\end{equation}

This implies the final proposition.

\begin{proposition}[Continuity of the Continuous Solution] \label{prop:contcontsol}
    Let Assumptions \ref{ass:contexptw} hold. There is a representation $\b{u}_{\mathrm{ptw}}\in C^0([0,T], L^\infty_{\mathrm{ptw}}(\Omega;\mathbb{R}^d))$ of the solution $\b{u} \in C^0([0,T], L^\infty(\Omega)^d)$ of (\ref{eqn:conteq}), characterized by $\b{u}_{\mathrm{ptw}}(t) \in L^\infty_{\text{ptw}}(\Omega; \mathbb{R}^d)$ being a representation of $\b{u}(t) \in L^\infty(\Omega)^d$ for all $t\in [0,T]$, such that 
    \begin{enumerate}
        \item If $u_0, v_0 \in L_{\text{ptw}}^\infty(\Omega;\mathbb{R}^d)$ and $\b{b}\in C^0([0,T], L_{\text{ptw}}^\infty(\Omega;\mathbb{R}^d))$ are continuous in a given $x_0\in \Omega$, the latter for all $t\in [0,T]$, then $\b{u}_{\mathrm{ptw}}(t)$ is continuous in $x_0$ for all $t\in [0,T]$. 
        \item Furthermore, if $v_0, u_0\in L_{\text{ptw}}^\infty(\Omega;\mathbb{R}^d)$, $\b{b}(t) \in L_{\text{ptw}}^\infty(\Omega;\mathbb{R}^d)$ are locally Lipschitz continuous in a given $x_0\in\Omega$, uniform with respect to $t\in [0,T]$, then $\b{u}_{\mathrm{ptw}}(t)$ is locally Lipschitz continuous in $x_0$ as well, with a Lipschitz constant independent of $t\in [0,T]$. If the Lipschitz continuity is global in the input data, then $\b{u}_{\text{ptw}}$ is globally Lipschitz continuous.
    \end{enumerate}
\end{proposition}

\begin{proof}
    We first attend to the weighting function $\omega_\delta$. Again, we will use Lemma \ref{lemm:shellbounds} made applicable by Assumptions \ref{ass:contexptw}, to bound the term
    \begin{align*}
         \int_{\mathbb{R}^d} |\omega_\delta(\|x-x'\|_2) - \omega_\delta(\|\hat{x}-x'\|_2)| \dd x',
    \end{align*}
    We again have for $s:= \|x-\hat{x}\|_2$, $r(x'):= \|x-x'\|_2$,  $l(x'):= \|\hat{x}-x'\|_2$, that
    \begin{align*}
        &\int_{\mathbb{R}^d} |\omega_\delta(\|x-x'\|_2) - \omega_\delta(\|\hat{x}-x'\|_2)| \dd x' \\
        &\leq \int_{\|x-x'\|\geq 2s} \omega_\delta(r(x') - s) - \omega_\delta(r(x') +s) \dd x' + \int_{\|x-x'\|\leq 2s} \omega_\delta(l(x')) + \omega_\delta(r(x')) \dd x' \\
        &\leq 3\cdot 2^{d-1} s \boundkerinteg + 3s \boundkerinteg + 2s \boundkerinteg,
    \end{align*}
    due to $r(x') \leq 2s$ implying $l(x') \leq 3s$. We then obtain using $s = \|x-\hat{x}\|_2$
\begin{align}\label{eqn:BndKernelLS}
        \int_{\mathbb{R}^d} |\omega_\delta(\|x-x'\|_2) - \omega_\delta(\|\hat{x}-x'\|_2)| \dd x' \leq C_d\boundkerinteg\|x-\hat{x}\|_2
    \end{align}
    with $C_d:= 3\cdot 2^{d-1} + 5$.\newline
    
    Now, for 1: The continuity follows from (\ref{eqn:fingroenregcontsol}), (\ref{eqn:BndKernelLS}) and the requirements in 1. We additionally apply the dominated convergence theorem to obtain
    \begin{equation*}
        \int_0^t \|\b{b}(s,x)-\b{b}(s,\hat{x})\|_\infty \dd s \rightarrow 0
    \end{equation*}
    as $\hat{x} \rightarrow x$.\newline
    For 2: The local Lipschitz continuity follows from (\ref{eqn:BndKernelLS}) and from the requirements in 2 applied to (\ref{eqn:fingroenregcontsol}). The existence of a Lipschitz constant $L^*$ such that $L_t \leq L^* <\infty$ for all $0\leq t \leq T$ follows from choosing $t=T$. 
\end{proof}

\begin{remark}[About the Convergence of the Discrete Schemes]\label{rm:DiscrConv} Proposition \ref{prop:contcontsol} explains the convergence result given in Theorem \ref{theo:convsemidiscr}, as all discontinuities present in the final solution have to be already present in the input data. This on the other hand requires the discontinuities in the input data to lie inside the closure of boundaries of all discrete elements $\Omega_i$ for all $\kappa_n$, $n\in \mathbb{N}$ in a given sequence of $\kappa_n \searrow 0$. This is enforced by the strict convergence requirements
\begin{equation*}
    \|\Kk\vec{\b{b}}_\kappa - \b{b}\|_{C^0([0,T], L^\infty(\Omega)^d)}\searrow 0,\  \|v_0 - \Kk \vec{v}_0^\kappa\|_{ L^\infty(\Omega)^d}\searrow 0,\  \|u_0 - \Kk \vec{u}_0^\kappa\|_{ L^\infty(\Omega)^d} \searrow 0.
\end{equation*}
Proposition \ref{prop:contcontsol} indicates that fracture within the bulk appears not as a true discontinuity but rather as a steep but continuous function. In this sense, the model can be understood as a regular approximation of a truly discontinuous fracture model. Note that the horizon $\delta>0$ affects the bounds (\ref{eqn:fingroenregcontsol}) directly, such that they diverge as $\delta \searrow 0$. This would allow the displacement to become steeper as $\delta \searrow 0$. 
We do not claim the convergence of the peridynamic model to a sharp discontinuity model as $\delta \searrow 0$, but merely state that the practically observed localization of damage as $\delta\searrow 0$, which necessarily implies a steeper displacement in the presence of continuity is consistent with the bounds given here. \par
This interpretation is consistent with other works drawing parallels with phase-field approaches, interpreting other peridynamic-type models as hybrid phase-field type models with length parameter $\delta>0$ and a damage field whose evolution follows from the constitutive law (e.g. \cite{COSKUN2025106302, bromley2026blendedapproachevolvingphase}). 
\end{remark}

\begin{remark}[On the Size and $\delta$-Dependence of the Continuity Bounds]\label{rm:bonddiscuss}
The constants in (\ref{eqn:fingroenregcontsol}) result from a worst-case Gr\"onwall argument and carry no quantitative information. For the parameters of Section \ref{sec:numerics} ($E = 100$, $\delta = 0.2$, $T = 0.1$), the argument of the hyperbolic cosine in (\ref{eqn:fingroenregcontsol}) is approximately $3\cdot 10^3$; it remains of order $10^2$ even if $\ell$, $g^+$ and $(S_1^- + 1)^{-1}$ are all set to one, compensating for the fact that these constants are potentially active only on small volumes in the horizon. 
Proposition \ref{prop:contcontsol} is therefore to be understood as a purely qualitative statement. Given a fixed $\delta > 0$, continuity and local Lipschitz continuity of the data are preserved, regardless of damage. In particular, the divergence of the bound as $\delta \searrow 0$ does not imply that solutions become steeper. Such divergence is, however, unavoidable if solutions for fixed data are to localize and develop jumps in the limit $\delta \searrow 0$. Whether and how the actual slopes grow as $\delta \searrow 0$ is beyond the present analysis, as is the resolution required to capture them at a given $\delta$. The latter is further studied numerically in Section \ref{sec:numerics}.
\end{remark}

\section{A Numerical Study of the Convergence}
\label{sec:numerics}
This section serves to illustrate the results of Theorem \ref{theo:fullconv} as well as Proposition \ref{prop:contcontsol} using numerical results. The aim is not to perform a rigorous numerical analysis, but to merely see how the theoretical results are visible in practice. Section \ref{sect:numex} will define the example model used for the numerical studies in (\ref{eqn:defker2d}) and (\ref{eqn:problset}), which will be reused in Section \ref{sect:numexconv}. 

\subsection{Visualization of the Continuity}\label{sect:numex}
We will now choose a specific model for $d=2$ in order to illustrate the theoretical results in Sections \ref{sec:ConvFullyDiscr} and \ref{sec:PropContSol}. First, set  $\Omega = (0,2)\times (0,1)$ and $\omega_\delta$ chosen as a  Lipschitz continuous cone function, given as in \cite{WP_FRAC, Bobaru2015Cracks} by 
\begin{equation}\label{eqn:defker2d}
    \omega_\delta(\|\zeta\|_2) = \begin{cases}
        \tfrac{36 E}{\delta^3\pi }(1 - \tfrac{\|\zeta\|_2}{\delta}) & \|\zeta\|_2\leq \delta \\
        0 & \text{else}
    \end{cases}
\end{equation}
where we choose a Young's modulus of $E=100$, and a horizon $\delta=0.2$. For the external body forces, we define
\begin{align}\label{eqn:problset}
    \b{b}((x,y)^\top,t) = \begin{cases}
        (0,  f\cdot x/2)^\top & y\in [1-\delta, 1] \\
        -(0, f\cdot x/2)^\top & y \in [0, \delta] \\
          0        & \text{else.}
    \end{cases}
\end{align}
with $f=125$. We further use the exact partial area weights defined in \cite{ImprovOnePointQuadAlgs}. Note that $\b{b}$ is constant with respect to time and has a non-trivial spatial discontinuity set $(0,2)\times \{\delta\} \cup (0,2)\times \{1-\delta\}$. To fulfill the Assumptions in \ref{ass:numconv}, it is sufficient to choose $\kappa^{-1} = \delta^{-1}k$ for some $k\in \mathbb{N}$ and $\vec{o}_\kappa = (\kappa/2,\kappa/2)^T$. We choose $\kappa^{-1} = 5 \cdot 64 = 320$, together with  $(\vec{\b{b}}_{\kappa})_i = \b{b}(x_i)$. Further, we set the continuous and discrete initial data to be $0$ everywhere. Finally, we set $G_0=1$ and $S^+_1=\sqrt{\tfrac{5\pi G_0}{9E\delta}}$ (\cite{Bobaru2015Cracks}, Equation 9), $S^+_0=0.98S^+_1$, $S_1^- = -0.999$, $S_0^- = -0.998$. The central difference scheme was run with $T=0.1$ and $N=8000$, leading to $h=1.25\cdot 10^{-5}$. We note the lower bound $S^-_0$ has never been crossed during execution. Figure \ref{fig:displconst} is obtained by drawing each element $\Omega_i$ of the domain in reference configuration displaced by $\vec{\b{u}}_i$ at different time steps.
The color of the elements indicates the presence of damage using the integrity function given by (\ref{eqn:damagefunc}). \par

To illustrate Remark \ref{rm:bonddiscuss}, we can explicitly construct the Lipschitz constants in $x$ where $\b{b}$ vanishes on a neighborhood of $x$, given through (\ref{eqn:fingroenregcontsol})
\begin{equation}\label{eqn:bnd2dexpl}
    L_{x,t}:=\tfrac{1}{2}t^2 \Big(C_d\scalarlaw^++C_{I}\Big)\boundkerinteg \cosh(t\sqrt{ C_{II}\boundkerinteg})
\end{equation}
where $C_{I}, C_{II}$ given via (\ref{eqn:defCICII}), are independent of $t$ and $x$. Then, for small $\delta>0$, $C_I$ and $C_{II}$ are of order $1/\delta$ and $1/\sqrt{\delta}$ respectively. Then we have rapid divergence of $L_{x,t} \rightarrow \infty$ for $\delta \rightarrow 0$ due to $\boundkerinteg = \tfrac{36E}{\delta^2}$ and the $\cosh$ term. As noted in Remark \ref{rm:bonddiscuss}, these bounds are enormous even for small parameters and can therefore not be used to measure quantitative behavior of the solutions. \newline
To illustrate that the slope of the displacement here is visible well below the theoretically conservative estimates, we plotted the numerical solution for the problem given by (\ref{eqn:defker2d}), (\ref{eqn:problset}) in Figure \ref{fig:displconst}. Note the emergence of what appears to be a sharp discontinuity where the input data $\b{b}$ is discontinuous, reflected in the numerical solution by the clean separation of the boundary layers on the top and bottom right on which we impose the external body forces, emphasizing that the results in Proposition \ref{prop:contcontsol} truly only apply locally at points where the input data was sufficiently regular. Further, the material separates in the middle as expected from the symmetrical loading, however the displacement along the center line does not jump once over a middle crack leaving a set of two halves separating, but is instead smoothed across, reflected in the numerical solution rendered here by elements $\Omega_i$ filling the space in between. This is consistent with the continuity implied by Proposition \ref{prop:contcontsol}.\par

\begin{figure}[!ht]
    \centering
    \includegraphics[width=0.49\linewidth, trim=0mm 10mm 10mm 10mm, clip]{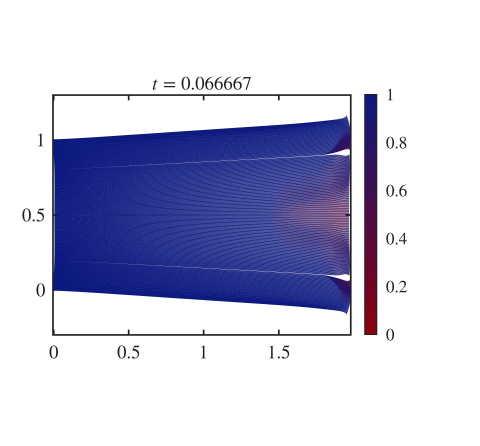}
    \includegraphics[width=0.49\linewidth, trim=0mm 10mm 10mm 10mm, clip]{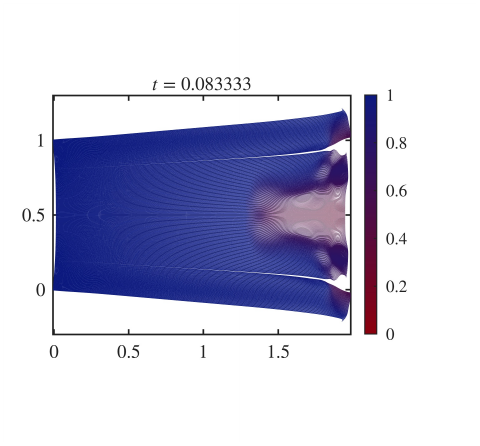}
    \includegraphics[width=0.55\linewidth, trim=0mm 10mm 10mm 10mm, clip]{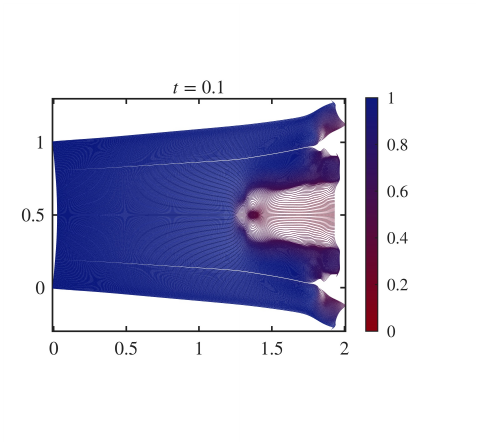}
    \caption{The slope of the displacement in the presence of large damage is visualized using a numerical solution by choosing a large horizon with a small Young's modulus, as such choices decrease the theoretical bounds (\ref{eqn:bnd2dexpl}). The elements $\Omega_i$ making up the reference configuration $\Omega$ are then drawn displaced by $\vec{\b{u}}^\kappa_i$ for different time steps, where $\vec{\b{u}}^\kappa$ denotes the solution to the problem described in (\ref{eqn:defker2d}), (\ref{eqn:problset}). The color of the elements indicates the damage present in the element as given by the discrete version of the integrity function (\ref{eqn:damagefuncdiscr}).}
    \label{fig:displconst}
\end{figure}

\FloatBarrier

\subsection{About the Convergence} \label{sect:numexconv}
We complement the convergence results of Theorem \ref{theo:fullconv} with a numerical study of the two dimensional problem already introduced above in Section \ref{sect:numex} with stronger body forces to increase damage, $f=150$. Note that as we vary the horizon $\delta$ below, we always update $S^+_1=\sqrt{\tfrac{5\pi G_0}{9E\delta}}$ accordingly.  We further chose the external body forces to scale with $\delta$ to keep the overall force times area load the same over all runs with different horizons, 
\begin{align}\label{eqn:problsetscaled}
    \b{b}((x,y)^\top,t) = \begin{cases}
        (0, (5\delta)^{-1} f x/2)^\top & y\in [1-\delta, 1] \\
        -(0,(5\delta)^{-1} f x/2)^\top & y \in [0, \delta] \\
          0        & \text{else.}
    \end{cases}
\end{align}
Throughout, we write
\begin{equation*}    
  m := \delta/\kappa
\end{equation*}
for the number of cells resolving the radius of the horizon. Since no exact solution is available, we compare successive iterates. For $\kappa_{n}>\kappa_{n+1}$ we evaluate
\begin{equation*}
  d_n^t := \big\|\Kkf{n}\vec{\b{u}}^{\kappa_n}(t)-\Kkf{n+1}\vec{\b{u}}^{\kappa_{n+1}}(t) \big\|_{L^p(\Omega)^d},  
\end{equation*}
for $p \in \{2, \infty\}$ using solutions of the fully discrete equations (\ref{eqn:discrprob}) for large $N=16000$, neglecting the error due to the discretization with respect to time. \newline
Assuming $\|u^{\kappa}-u\|\le C\kappa^{q}$, the triangle inequality gives $d_n\le C(\kappa_n^q+\kappa_{n+1}^q)$, so, for a geometric sequence $\kappa_{n+1}=\kappa_n/2$ the right hand side is $C(1+2^{-q})\kappa_n^{q}$. Using this, we report the experimental order of convergence using
\begin{equation}\label{eqn:eoc}
  \EOC_n^t := \frac{\log(d_{n}^t/d_{n+1}^t)}{\log(\kappa_{n}/\kappa_{n+1})},
\end{equation}
which recovers the order $q$ if the error $d_n$ behaves like $C'\kappa^q$.  \newline
We evaluate the experimental order of convergence at three distinct points in time, $t_1=T/3$, $t_2=2T/3$, $t_3=T$. 

\begin{table}[htbp]
  \centering
  \begin{tabular}{@{}l c c c c c @{}}
    \toprule
    \multicolumn{6}{@{}l}{$p=\infty$}\\
    $\delta$ & $\kappa^{-1}$ & $m$ & $t_1$ & $t_2$ & $t_3$\\
    \midrule
    $1/5$  & $40,\dots,320$ & $8$--$64$ & (0.768, 1.129) & (0.751, 0.896) & (0.887, 1.708)\\
    $1/10$ & $40,\dots,320$ & $4$--$32$ & (0.738, 0.944) & (1.120, 1.569) & (0.666, 0.008)\\
    $1/20$ & $80,\dots,640$ & $4$--$32$ & (0.685, 0.262) & (0.568, 1.082) & (0.656,-0.120)\\
    \bottomrule
    \noalign{\vspace{0.5cm}}
    \toprule
    \multicolumn{6}{@{}l}{$p=2$}\\
    $\delta$ & $\kappa^{-1}$ & $m$ & $t_1$ & $t_2$ & $t_3$\\
    \midrule
    $1/5$  & $40,\dots,320$ & $8$--$64$ & (0.998, 1.002) & (1.083, 1.003) & (1.249, 1.027)\\
    $1/10$ & $40,\dots,320$ & $4$--$32$ & (1.025, 1.002) & (1.535, 1.118) & (1.460, 1.218)\\
    $1/20$ & $80,\dots,640$ & $4$--$32$ & (1.079, 1.034) & (1.170, 1.152) & (1.891, 0.829)\\
    \bottomrule
  \end{tabular}
  \caption{Experimental order of convergence in the $L^\infty(\Omega)^d$ and $L^2(\Omega)^d$ norm, given in pairs $(\EOC_0^{t_k},\EOC_1^{t_k})$ using the formula (\ref{eqn:eoc}), evaluated at three different time points $t_1=T/3, t_2=2T/3, t_3=T=0.1$ and for different horizons $\delta$.}
  \label{tab:eoc}
\end{table}

\begin{figure}[!ht]
    \centering
    \includegraphics[width=\textwidth]{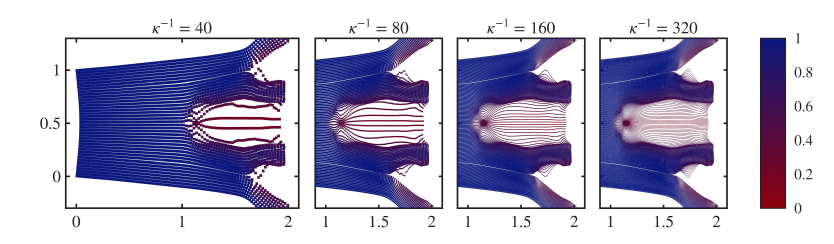}
    \includegraphics[width=\textwidth]{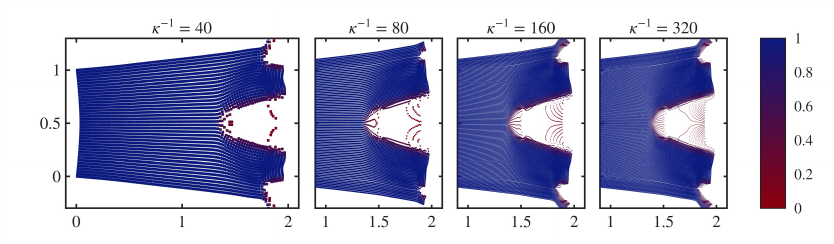}
    \includegraphics[width=\textwidth]{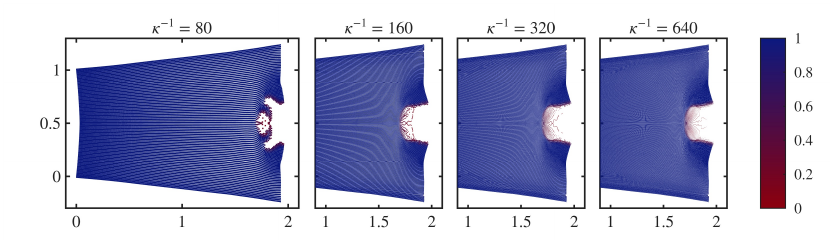}
    \caption{Plots of the visualized displacement for the different horizons and increasing spatial resolutions at time $t=T=0.1$. The horizon used in the top row is $\delta=1/5$, in the middle $\delta=1/10$, and on the bottom row $\delta=1/20$. The higher resolution plots have been cropped to allow better visual comparison of the area with damage. These plots in particular illustrate how the $L^\infty(\Omega)^2$ norm, unlike the $L^2(\Omega)^2$ norm, fails to capture the fact that the damage region itself remains fairly stable and the error localizes with increasing resolution. This explains the gap between the $p=2$ and $p=\infty$ rates in Table \ref{tab:eoc}. As expected, the $\delta = 1/20$ run required higher resolutions to reach the final damage shape as can be seen by comparing the lower left low spatial resolution plot with $\kappa^{-1}=80$ to the lower right high resolution plot with $\kappa^{-1}=640$. Once again, the elements filling in the gap produced by the crack as the resolution is increased are consistent with the continuity of the continuous solution as described by Proposition \ref{prop:contcontsol}. }
    \label{fig:convvisual}
\end{figure}

Table \ref{tab:eoc} for $p=\infty$ is consistent with what the constants derived in (\ref{eqn:bnd2dexpl}) suggest, the necessary resolution of the horizon $\delta$ given by $m$, required to capture the slope of the continuous solution in the discrete setting well enough to reach the asymptotic linear rate $O(\kappa +1/N)$ as proven in Theorem \ref{theo:fullconv}, increases with shrinking $\delta$ in the presence of damage. This is suggested by the declining $\EOC$ values at $t=t_3$ as $\delta$ shrinks ($\EOC_0$ for $\delta=1/5$ vs $\EOC_1$ for $\delta=1/10, 1/20$; different $\EOC$ indices were chosen to keep $m$ constant). We note that those numbers do not contradict the theorems, they merely suggest that the scheme is still pre-asymptotic. As such, the shape of the damage distribution might still be dependent on $\kappa$ or cracks might still appear as discontinuous with elements along the crack border showing spurious behavior. \newline
The table for $p=2$ reveals that in line with Remark \ref{rm:abouttheassumptionspt2}, using norms that factor in the volume on which an error occurs yield much better values for the $\EOC$. This indicates that most errors driving up the $L^\infty(\Omega)^d$ norm localize on small volumes. To make this visible, the solutions are rendered in Figure \ref{fig:convvisual}. 

The overall results still suggest that while the scheme is convergent, the model requires very large resolutions to resolve the displacement near damage in the $L^\infty$-norm for this example. Due to performance issues inherently implied by choosing large $m$, this makes adaptive schemes naturally relevant.

\FloatBarrier
\section{Homogeneous Dirichlet Volume Constraint}
\label{sec:HomogDirich}

To impose volume constraints on the equations, we chose a naive method here due to its simplicity. This method simply sets the displacement to $0$ on a set $\Omega_I$ adjacent but disjoint with $\Omega$. We note that issues with this approach include artificial boundary effects near the interface between $\Omega$ and $\Omega_I$. For details, see \cite{PrudTreatBnd, Le2018, novelBCinPD}. Alternative approaches are given, e.g. \cite{zhaoGenFNM, novelBCinPD}.

For the analysis, we will then rewrite the external body forces as a  force term $\b{b}^D$ equal to the force the constrained portion of $\Omega\cup \Omega_I$ exerts on $\Omega$. This is done to avoid the complications of inner and outer approximations caused by the $L^\infty(\Omega)^d$ norm. First, denote $\Omega_I$ as the open and bounded set disjoint to $\Omega$ on which we wish to impose the volume constraints. We then define $\b{b}^D: C^0([-T, 0], L^\infty(\Omega)^d) \rightarrow L^\infty(\Omega)^d$ as
\begin{align*}
    \b{b}^D(\b{u})(x) :=\int_{\Omega_I} \omega_\delta(\|x'-x\|_2)\scalarlaw(\mathcal{S}(\hat{\b{u}}(0),x',x))\mu\Big(\esssup_{-T\leq s\leq 0}\mathcal{S}(\hat{\b{u}}(s),x',x)\Big)e(\hat{\b{u}}(0),x',x) \dd x',
\end{align*}
where the hat is an extension operator,
\begin{align*}
    \textasciicircum: L^\infty(\Omega)^d \rightarrow L^\infty(\Omega\cup\Omega_I)^d, u \mapsto \hat{u} := \begin{cases}
        u(x)  & x\in \Omega \\
        0 & \text{else},
    \end{cases} 
\end{align*}
where we naturally extend the operators $\ca{S},e$ to include terms from $\Omega_I$. Note that the definition of $\b{b}^D(\b{u})$ allows the bonds crossing the interface between $\Omega_I$ and $\Omega$ to break. \par
If $\dist(x, \Omega_I) >\delta$, then $\b{b}^D(\b{u}_t)(x) = 0$ for all $\b{u} \in C^0([0,T], L^\infty(\Omega)^d)$ and $t \in [0,T]$. \par

This force term admits a natural discretization. First, we begin by defining $\hat{I}^\kappa$ as the set of indices ${{i}}$ with ${i}\in \mathbb{Z}^d$ such that $\hat{V}_{{i}}:=\vol(B_{{i}} \cap \Omega_I)\not=0$. 
As the displacement may occupy different values on $B_j \cap \Omega$ and $B_j \cap \Omega_I$ for the same index $j$, we define the total index set for the discrete displacement $I^\kappa \dot{\cup} \hat{I}^\kappa$ as the disjoint union of $I^\kappa$ and $\hat{I}^\kappa$. This is a set constructed to contain both $I^\kappa$ and $\hat{I}^\kappa$ as disjoint subsets. As a consequence, indices $j\in\mathbb{Z}^d$ with $\vol(\Omega \cap B_j)\not=0$ as well as $\vol(\Omega_I \cap B_j)\not= 0$ are in $I^\kappa \dot{\cup} \hat{I}^\kappa$ twice, once as an element of $I^\kappa$, written without a hat as usual, and once as the copy stemming from $j \in \hat{I}^\kappa$ denoted by $\hat{j} \in I^\kappa \dot\cup \hat{I}^\kappa$, the hat distinguishing between the instances. In particular for $j\in \hat{I}^\kappa\cup I^\kappa$, $j\not=\hat{j}$. In this sense, $\hat{}
:\hat{I}^\kappa\cap I^\kappa \rightarrow I^\kappa \dot \cup \hat{I}^\kappa$ can be understood as a mapping of an index in the intersection onto its copy in $I^\kappa \dot\cup \hat{I}^\kappa$ stemming from $\hat{I}^\kappa$. We then extend the discrete displacement $\vec{\b{u}}(t) \in (\mathbb{R}^d)^{I^\kappa}$ to $\vec{\hat{\b{u}}}(t) \in (\mathbb{R}^d)^{I^\kappa \dot{\cup} \hat{I}^\kappa}$ with $\vec{\hat{\b{u}}}_{\hat{j}}(t):= 0$ for all ${j}\in \hat{I}^\kappa$. This duplication naturally reflects the two different values the displacement might take, specifically, $\vec{\b{u}}_{j}^\kappa$ and $\vec{\b{u}}_{\hat{j}}^\kappa$. Note that the individual sets $I^\kappa$ and $\hat{I}^\kappa$ are left unaffected and are still subsets of $\mathbb{Z}^d$.\par
To construct the discretization, we define $\vec{\b{b}}^D_\kappa: \Linftyptw([-T,0],  (\mathbb{R}^d)^{I^\kappa}) \rightarrow (\mathbb{R}^d)^{I^\kappa}$ using
\begin{equation*}
    (\vec{\b{b}}^D_\kappa)_i(\vec{\b{u}}) := \sum_{\substack{{j} \in \hat{I}^\kappa \\ {j} \not= i}} \hat{V}_{{j}}\omega_\delta(\|x_{i}- x_{{j}}\|_2)\scalarlaw(\mathcal{S}^{i\hat{j}}(\vec{\hat{\b{u}}}(0)))\mu\Big(\sup_{-T\leq s\leq 0}\mathcal{S}^{i\hat{j}}(\vec{\hat{\b{u}}}(s))\Big)e^{i\hat{j}}(\vec{\hat{\b{u}}}(0)).
\end{equation*}
We also extend the discrete bond-stretch and direction of motion functions $\ca{S}^{i\hat{j}}$, $e^{i\hat{j}}$ to take indices $\hat{j}$, for $j \in \hat{I}^\kappa$.\par

We note here that $\hat{V}_{{j}} \not = V_j$ for $j\in I^\kappa \cap \hat{I}^\kappa$. This is such that elements $B_i$ intersecting both $\Omega_I$ and $\Omega$ contribute to the equations in (\ref{eqn:semidiscrprobldirich}) twice, once as $V_j$ in $F$ and once as $\hat{V}_{{j}}$ in $\vec{\b{b}}^D_\kappa$. With $\hat{\Omega}_i := \Omega_I \cap B_i$, we then have $\hat{\Omega}_i \dot{\cup} \Omega_i = B_i \cap (\Omega_I \dot\cup \Omega)$, implying $\hat{V}_i+V_i = \vol(B_i \cap (\Omega_I\dot\cup \Omega))$. 

The continuous equations (\ref{eqn:conteq}) then become
\begin{equation}
\label{eqn:conteqdirich}
    \begin{rcases}
        \rho\ddot{\b{u}}(t) = F(\b{u}_t) + \b{b}(t) + \b{b}^D(\b{u}_t)\\ 
        \b{u}(0) = u_0 \\
        \dot{\b{u}}(0) = v_0
    \end{rcases}.
\end{equation}
Similarly, the semi-discrete equations (\ref{eqn:semdiscrprobl}) become
\begin{equation}
\label{eqn:semidiscrprobldirich}
    \begin{rcases}
        \tfrac{d^2}{dt^2}\vec{\b{u}}^\kappa(t) = \vec{F}_\kappa((\vec{\b{u}}^\kappa)_t) + \vec{\b{b}}_\kappa(t) + \vec{\b{b}}^D_\kappa((\vec{\b{u}}^\kappa)_t) \\
        \ \ \vec{\b{u}}^\kappa(0) = \vec{u}_0^\kappa\\
        \tfrac{d}{dt}\vec{\b{u}}^\kappa(0) = \vec{v}_0^\kappa
    \end{rcases}.
\end{equation}
The fully discrete system (\ref{eqn:discrprob}) then naturally becomes
\begin{equation}
\label{eqn:discrprobdirich}
    \begin{aligned}
        \vec{\b{u}}^{-1,N}_\kappa& = \vec{u}_{0}^\kappa- h\vec{v}_{0}^\kappa + \tfrac{1}{2}h^2(\vec{F}_\kappa(\IP_1(\vec{\b{u}}^{0,N}_\kappa))+ \vec{\b{b}}_\kappa(0)+\vec{\b{b}}^D_\kappa(\IP_1(\vec{\b{u}}^{0,N}_\kappa))) \\
        \vec{\b{u}}^{0,N}_\kappa& = \vec{u}_{0}^\kappa \\
        \vec{{\b{u}}}^{n,N}_\kappa &= 2\vec{\b{u}}^{n-1,N}_\kappa - \vec{\b{u}}^{n-2,N}_\kappa +h^2(\vec{F}_\kappa(\IP_{n}(\vec{\b{u}}^{...n-1,N}_\kappa))+ \vec{\b{b}}_\kappa(t_{n-1})+  \vec{\b{b}}^D_\kappa(\IP_{n}(\vec{\b{u}}^{...n-1,N}_\kappa))).
    \end{aligned}
\end{equation}

 To obtain the same convergence results as without the boundary conditions, we now verify the relevant properties of this new force term. Since it is  structurally derived from $F$, it is no surprise that similar results follow for it and the proofs carry over. 
 
\begin{proposition}[Lipschitz Bounds for the Dirichlet Force]
    For $\b{u},\b{v} \in C^0([-T, 0], L^\infty(\Omega)^d)$
    \begin{equation*}
        \|\b{b}^D(\b{u}) - \b{b}^D(\b{v})\|_{L^\infty(\Omega)^d} \leq L_D \|\b{u}-\b{v}\|_{C^0([-T,0], L^\infty(\Omega)^d)}.
    \end{equation*}
Similarly, for $\vec{\b{u}},\vec{\b{v}} \in C^0([-T, 0], (\mathbb{R}^d)^{I^\kappa})$,
\begin{equation*}
    \|\vec{\b{b}}^D_\kappa(\vec{\b{u}}) - \vec{\b{b}}^D_\kappa(\vec{\b{v}})\|_{\infty} \leq L_{{D}} \|\vec{\b{u}}-\vec{\b{v}}\|_{C^0([-T,0], (\mathbb{R}^d)^{I^\kappa})},
\end{equation*}
for some $L_D>0$ independent of $0 < \kappa \leq \delta/(2\sqrt{d})$.
\end{proposition}
\begin{proof}
These proofs are structurally identical to Proposition \ref{prop:LScontdiscrrhs} and \cite{WP_FRAC}, Lemma 1. 
\end{proof}
Both of these properties imply the well-posedness of (\ref{eqn:conteqdirich}) and (\ref{eqn:semidiscrprobldirich}). Note that the well-posedness of (\ref{eqn:conteqdirich}) follows already from \cite{WP_FRAC}. We will however need this bound explicitly later.

\begin{proposition}[Convergence of the Discretization of the Dirichlet Force]\newline
Let $\b{u}\in C^0([0,T], L^\infty(\Omega)^d)$ and $\vec{\b{u}} \in C^0([0,T], (\mathbb{R}^d)^{I^\kappa})$, then, for all $t\in [0,T]$,
\begin{equation*}
    \|\b{b}^D(\b{u}_t) - \Kk \vec{\b{b}}^D_\kappa((\vec{\b{u}})_t) \|_{L^\infty(\Omega)^d} \leq \beta^D \Big(\max_{0\leq s \leq t}\|\b{u}(s) - \Kk\vec{\b{u}}(s)\|_{L^\infty(\Omega)^d}\Big) + \alpha^D_\kappa
\end{equation*}
for some $\alpha_\kappa^D, \beta^D \geq 0$, and \begin{equation*}
    O(\kappa) \ni \alpha_\kappa^D \xrightarrow[]{\kappa \searrow 0} 0,\ \ \ \beta^D < \infty
\end{equation*}
for $\kappa \searrow 0$ and $\beta^D$ independent of $0<\kappa\leq \delta/(2\sqrt{d})$.
\end{proposition}
\begin{proof} The proof is structurally the same as Proposition \ref{prop:discrrhserr}.
\end{proof}

From this we obtain the convergence of the solutions of the semi-discrete schemes (\ref{eqn:semidiscrprobldirich}) to the solution of (\ref{eqn:conteqdirich}) just like we did in Theorem \ref{theo:convsemidiscr}. \par

\begin{corollary}[Convergence of the Semi-Discrete Systems with Homogeneous Dirichlet Boundary Conditions] Let Assumptions \ref{ass:numconv} hold and $\Omega_I$ be an open and bounded set disjoint with $\Omega$. Let $\kappa_n \searrow 0$ and $\vec{\b{u}}^{\kappa_n}\in C^2([0,T], (\mathbb{R}^d)^{I^{\kappa_n}})$ be the solutions to the semi-discrete problem (\ref{eqn:semidiscrprobldirich}), as well as $\b{u}\in C^2([0,T], L^\infty(\Omega)^d)$ the solution to the continuous problem (\ref{eqn:conteqdirich}), then
    \begin{equation*}
        \|\b{u} - \Kkf{n}\vec{\b{u}}^{\kappa_n}\|_{C^1([0,T], L^\infty(\Omega)^d)} \xrightarrow[]{n \rightarrow \infty} 0. 
    \end{equation*}
    If further the rates in (\ref{eqn:asssemidiscrconvinputdata}) hold, we have
    \begin{equation}
        \|\b{u} - \Kkf{n}\vec{\b{u}}^{\kappa_n}\|_{C^1([0,T], L^\infty(\Omega)^d)} \in O(\kappa)
    \end{equation}
    \qedwhite

\end{corollary}

The convergence of the fully discrete scheme as in Theorem \ref{theo:fullconv} then follows again by treating the $\b{b}^D$ terms like a second $F$ term.
First,
${\hat{G}}_h^D: C^0([-T,0], (\mathbb{R}^d)^{I^\kappa} \times (\mathbb{R}^d)^{I^\kappa}) \times [0,T] \rightarrow (\mathbb{R}^d)^{I^\kappa} \times (\mathbb{R}^d)^{I^\kappa} $ with 
\begin{align*}
    {\hat{G}}_h^D((\phi,\psi)^\top, t) = \begin{pmatrix}
    \psi(0)\\
    \vec{F}_\kappa(\tau_{-h}\phi) + \vec{\b{b}}_\kappa^D(\tau_{-h}\phi) + \vec{\b{b}}_\kappa(t)
\end{pmatrix}.
\end{align*}
Let $\vec{\b{u}}_{Red}^\kappa$ be the unique solution to the reduced system of (\ref{eqn:semidiscrprobldirich}). We then define the {consistency residual} $\hat{K}^{N,n}$ as
\begin{align}\label{eqn:defstabresdirich}
    \hat{K}^{N,n}_\kappa := \tfrac{1}{h}( \vec{\b{u}}_{Red}^\kappa(t_{n-1})-\vec{\b{u}}_{Red}^\kappa(t_{n}) )+ {\hat{G}}_h^D((\vec{\b{u}}^\kappa_{Red})_{t_n}, t_{n-1}).
\end{align}

Lemma \ref{lemm:stabineq} then becomes the following.

\begin{lemma}[Stability Inequality with the Dirichlet Term]\label{lemm:stabineqdirich}
    We have for all $N\in\mathbb{N}$, $1\leq n \leq N$ and all $\delta/\sqrt{d}> \kappa >0$, that
    \begin{align*}
        (1-h)\|e_n\|_\infty& \leq h\|\hat{K}^{N,n}_\kappa\|_\infty+ \|e_{n-1}\|_\infty \\
       &\quad + h(L_{\vec{F}}+L_{{D}})\|\IP_{n}((\hat{\b{u}}^{0,...n-1,N}_\kappa)_1)) - ((\vec{\b{u}}_{Red}^\kappa)_1)_{t_{n-1}}\|_{\Linftyptw([-T,0], (\mathbb{R}^d)^{I^\kappa})},
    \end{align*}
    where $\vec{\b{u}}_{Red}^\kappa$ is the unique solution to (\ref{eqn:semidiscrprobldirich}) in reduced form, and $\hat{\b{u}}^{n,N}_\kappa$ again derived from the solutions $\vec{\b{u}}^{n,N}_\kappa$ from (\ref{eqn:discrprobdirich}), just like we did in (\ref{eqn:redfulldiscrsol}).\qedwhite
\end{lemma}

The rest of the proof then goes similarly as in Theorem \ref{theo:fullconv}.

\begin{corollary}[Convergence of the Fully Discrete Scheme with Boundary Conditions] \label{corol:fullconvdirich}
    Let Assumptions \ref{ass:numconv} hold. Let $\Omega_I$ be open, bounded and disjoint with $\Omega$. Further, let $(N_k)_{k\in \mathbb{N}}\subset \mathbb{N}$ be such that $N_k \xrightarrow[]{k\rightarrow \infty} \infty$ and $0<\kappa_k \searrow 0$. Denote the unique solution to (\ref{eqn:conteqdirich}) in reduced form as $\b{u}_{Red} \in C^1([0,T], L^\infty(\Omega)^d\times  L^\infty(\Omega)^d)$, and the solution constructed using (\ref{eqn:discrprobdirich}) expressed like in (\ref{eqn:redfulldiscrsol}) as $(\hat{\b{u}}_{\kappa_k}^{n, N_k})_{0\leq n \leq N_k}$. Then, we have
    \begin{equation*}
        \lim_{k\rightarrow \infty} \max_{0\leq n \leq N_k} \| \b{u}_{Red}(t_n^{N_k}) - \Kkf{k}^2\hat{\b{u}}_{\kappa_k}^{n, N_k}\|_{L^\infty(\Omega)^d\times L^\infty(\Omega)^d} = 0.
    \end{equation*}
    Further, if $\b{b}$ is Lipschitz continuous in $t$, as well as 
    \begin{align}
        &\|\Kkf{k}\vec{\b{b}}_{\kappa_k} - \b{b}\|_{C^0([0,T], L^\infty(\Omega)^d)} \in O(\kappa_k),\ \|u_0 - \Kkf{k}\vec{u}_0\|_{L^\infty(\Omega)^d} \in O(\kappa_k),\\
        &\hspace{3.1cm}\|v_0 - \Kkf{k}\vec{v}_0\|_{L^\infty(\Omega)^d} \in O(\kappa_k)
    \end{align}
    then we obtain
    \begin{equation*}
       \max_{0\leq n \leq N_k} \| \b{u}_{Red}(t_n^{N_k}) - \Kkf{k}^2\hat{\b{u}}_{\kappa_k}^{n, N_k}\|_{L^\infty(\Omega)^d\times L^\infty(\Omega)^d} \in O(1/N_k+ \kappa_k).
    \end{equation*} \qedwhite
\end{corollary}

\section{Acknowledgments}
This research was funded by the Deutsche Forschungsgemeinschaft (DFG, German Research Foundation) - 377472739/GRK 2423/2-2023. The authors are very grateful for this support. 

The authors declare that they have no conflict of interest.

\bibliographystyle{plainurl}
\addcontentsline{toc}{section}{References}
\bibliography{biblio}

@article{SillRef,
  author  = "S.A. Silling",
  title   = "Reformulation of elasticity theory for
discontinuities and long-range forces",
  journal = "Journal of the Mechanics and Physics of Solids",
  year    = 2000,
  volume  = "48",
  pages   = "175--209"
}

@Article{Silling2007,
author={Silling, S. A.
and Epton, M.
and Weckner, O.
and Xu, J.
and Askari, E.},
title={Peridynamic States and Constitutive Modeling},
journal={Journal of Elasticity},
year={2007},
month={Aug},
day={01},
volume={88},
number={2},
pages={151-184},
issn={1573-2681},
doi={10.1007/s10659-007-9125-1},
url={https://doi.org/10.1007/s10659-007-9125-1}
}

@article{SillMeshFree,
title = {A meshfree method based on the peridynamic model of solid mechanics},
journal = {Computers \& Structures},
volume = {83},
number = {17},
pages = {1526-1535},
year = {2005},
note = {Advances in Meshfree Methods},
issn = {0045-7949},
doi = {https://doi.org/10.1016/j.compstruc.2004.11.026},
url = {https://www.sciencedirect.com/science/article/pii/S0045794905000805},
author = {S.A. Silling and E. Askari}
}

@article{CPDSOURCE,
title = {Continuum-kinematics-inspired peridynamics. {M}echanical problems},
journal = {Journal of the Mechanics and Physics of Solids},
volume = {131},
pages = {125-146},
year = {2019},
issn = {0022-5096},
doi = {https://doi.org/10.1016/j.jmps.2019.06.016},
url = {https://www.sciencedirect.com/science/article/pii/S0022509619300456},
author = {A. Javili and A.T. McBride and P. Steinmann}
}

@Article{CPDFRAC,
title = {Dynamic fracture with continuum-kinematics-based peridynamics},
journal = {AIMS Materials Science},
volume = {9},
number = {6},
pages = {791-807},
year = {2022},
issn = {2372-0484},
doi = {10.3934/matersci.2022049},
url = {https://www.aimspress.com/article/doi/10.3934/matersci.2022049},
author = {Kai Friebertshäuser and Christian Wieners and Kerstin Weinberg},
}

@Article{Partmann2023,
author={Partmann, Kai
and Wieners, Christian
and Weinberg, Kerstin},
title={Continuum-kinematics-based peridynamics and phase-field approximation of non-local dynamic fracture},
journal={International Journal of Fracture},
year={2023},
month={Dec},
day={01},
volume={244},
number={1},
pages={187-200},
issn={1573-2673},
doi={10.1007/s10704-023-00726-7},
url={https://doi.org/10.1007/s10704-023-00726-7}
}

@Article{lipton2015cohesive,
author={Lipton, Robert},
title={Cohesive Dynamics and Brittle Fracture},
journal={Journal of Elasticity},
year={2016},
month={Aug},
day={01},
volume={124},
number={2},
pages={143-191},
issn={1573-2681},
doi={10.1007/s10659-015-9564-z},
url={https://doi.org/10.1007/s10659-015-9564-z}
}

@Article{lipton2020nonlocal,
author={Lipton, Robert P.
and Jha, Prashant K.},
title={Nonlocal elastodynamics and fracture},
journal={Nonlinear Differential Equations and Applications NoDEA},
year={2021},
month={Mar},
day={17},
volume={28},
number={3},
pages={23},
issn={1420-9004},
doi={10.1007/s00030-021-00683-x},
url={https://doi.org/10.1007/s00030-021-00683-x}
}

@Inbook{Lipton2019DynamicBF,
author="Lipton, Robert
and Said, Eyad
and Jha, Prashant K.",
title="Dynamic Brittle Fracture from Nonlocal Double-Well Potentials: A State-Based Model",
bookTitle="Handbook of Nonlocal Continuum Mechanics for Materials and Structures",
year="2019",
publisher="Springer International Publishing",
address="Cham",
pages="1265--1291",
isbn="978-3-319-58729-5",
doi="10.1007/978-3-319-58729-5_33",
url="https://doi.org/10.1007/978-3-319-58729-5_33"
}

@article{JAVILI2025104563,
title = {A nonlinear correspondence model for three-dimensional continuum-kinematics-inspired peridynamics},
journal = {Mechanics Research Communications},
volume = {150},
pages = {104563},
year = {2025},
issn = {0093-6413},
doi = {https://doi.org/10.1016/j.mechrescom.2025.104563},
url = {https://www.sciencedirect.com/science/article/pii/S009364132500196X},
author = {A. Javili and E. Ekiz and P. Steinmann}
}

@Article{Javili2025,
author={Javili, Ali
and Ekiz, Ekim
and Steinmann, Paul},
title={A Geometrically Nonlinear Correspondence Model for Continuum-Kinematics-Inspired Peridynamics},
journal={Journal of Peridynamics and Nonlocal Modeling},
year={2025},
month={Mar},
day={19},
volume={7},
number={1},
pages={6},
issn={2522-8978},
doi={10.1007/s42102-025-00129-3},
url={https://doi.org/10.1007/s42102-025-00129-3}
}

@Article{Ha2010,
author={Ha, Youn Doh
and Bobaru, Florin},
title={Studies of dynamic crack propagation and crack branching with peridynamics},
journal={International Journal of Fracture},
year={2010},
month={Mar},
day={01},
volume={162},
number={1},
pages={229-244},
issn={1573-2673},
doi={10.1007/s10704-010-9442-4},
url={https://doi.org/10.1007/s10704-010-9442-4}
}

@article{KHOSRAVANI2022103899,
title = {On the use of peridynamics in fracture of ultra-high performance concrete},
journal = {Mechanics Research Communications},
volume = {123},
pages = {103899},
year = {2022},
issn = {0093-6413},
doi = {https://doi.org/10.1016/j.mechrescom.2022.103899},
url = {https://www.sciencedirect.com/science/article/pii/S0093641322000532},
author = {Mohammad Reza Khosravani and Kai Friebertshäuser and Kerstin Weinberg}
}

@Article{Partmann2024,
author={Partmann, Kai
and Dienst, Manuel
and Weinberg, Kerstin},
title={Peridynamic computations of wave propagation and reflection at material interfaces},
journal={Archive of Applied Mechanics},
year={2024},
month={Sep},
day={01},
volume={94},
number={9},
pages={2405-2416},
issn={1432-0681},
doi={10.1007/s00419-024-02646-x},
url={https://doi.org/10.1007/s00419-024-02646-x}
}

@Article{Bobaru2015Cracks,
author={Bobaru, Florin
and Zhang, Guanfeng},
title={Why do cracks branch? {A} peridynamic investigation of dynamic brittle fracture},
journal={International Journal of Fracture},
year={2015},
month={Nov},
day={01},
volume={196},
number={1},
pages={59-98},
issn={1573-2673},
doi={10.1007/s10704-015-0056-8},
url={https://doi.org/10.1007/s10704-015-0056-8}
}

@Article{Willberg2023,
author={Willberg, Christian
and Hesse, Jan-Timo
and Heinecke, Falk},
title={Peridynamic Simulation of a Mixed-Mode Fracture Experiment in PMMA Utilizing an Adaptive-Time Stepping for an Explicit Solver},
journal={Journal of Peridynamics and Nonlocal Modeling},
year={2023},
month={Jun},
day={01},
volume={5},
number={2},
pages={205-228},
issn={2522-8978},
doi={10.1007/s42102-021-00079-6},
url={https://doi.org/10.1007/s42102-021-00079-6}
}

@Article{Diehl2019,
author={Diehl, Patrick
and Prudhomme, Serge
and L{\'e}vesque, Martin},
title={A Review of Benchmark Experiments for the Validation of Peridynamics Models},
journal={Journal of Peridynamics and Nonlocal Modeling},
year={2019},
month={Apr},
day={01},
volume={1},
number={1},
pages={14-35},
issn={2522-8978},
doi={10.1007/s42102-018-0004-x},
url={https://doi.org/10.1007/s42102-018-0004-x}
}

@article{BellidoWellposed,
    author = {Bellido, Jos\'{e} C. and Mora-Corral, Carlos},
    title = {Existence for Nonlocal Variational Problems in Peridynamics},
    journal = {SIAM Journal on Mathematical Analysis},
    volume = {46},
    number = {1},
    pages = {890-916},
    year = {2014},
    doi = {10.1137/130911548},
    URL = {https://doi.org/10.1137/130911548 }
}

@article{CocliteWellpsdNONLINPD,
title = "Wellposedness of a nonlinear peridynamic model",
author = "Coclite, {Giuseppe Maria} and Serena Dipierro and Francesco Maddalena and Enrico Valdinoci",
year = "2019",
month = jan,
doi = "10.1088/1361-6544/aae71b",
language = "English",
volume = "32",
pages = "1--21",
journal = "Nonlinearity",
issn = "0951-7715",
publisher = "IOP Publishing",
number = "1",
}

@Article{WP_FRAC,
author={Du, Qiang
and Tao, Yunzhe
and Tian, Xiaochuan},
title={A Peridynamic Model of Fracture Mechanics with Bond-Breaking},
journal={Journal of Elasticity},
year={2018},
month={Aug},
day={01},
volume={132},
number={2},
pages={197-218},
issn={1573-2681},
doi={10.1007/s10659-017-9661-2},
url={https://doi.org/10.1007/s10659-017-9661-2}
}

@article{Mengesha2014,
author={Mengesha, Tadele
and Du, Qiang},
title={Nonlocal Constrained Value Problems for a Linear Peridynamic {N}avier Equation},
journal={Journal of Elasticity},
year={2014},
month={Jun},
day={01},
volume={116},
number={1},
pages={27-51},
issn={1573-2681},
doi={10.1007/s10659-013-9456-z},
url={https://doi.org/10.1007/s10659-013-9456-z}
}

@article{Meng2014BB, title={The bond-based peridynamic system with {D}irichlet-type volume constraint}, volume={144}, DOI={10.1017/S0308210512001436}, number={1}, journal={Proceedings of the Royal Society of Edinburgh: Section A Mathematics}, author={Mengesha, Tadele and Du, Qiang}, year={2014}, pages={161–186}}

@Article{Emmrich2016,
author={Emmrich, Etienne
and Puhst, Dimitri},
title={A Short Note on Modeling Damage in Peridynamics},
journal={Journal of Elasticity},
year={2016},
month={Apr},
day={01},
volume={123},
number={2},
pages={245-252},
issn={1573-2681},
doi={10.1007/s10659-015-9550-5},
url={https://doi.org/10.1007/s10659-015-9550-5}
}

@article{OnWellPsdLinPD,
author = {Etienne Emmrich and Olaf Weckner},
title = {{On the well-posedness of the linear peridynamic model and its convergence towards the Navier equation of linear elasticity}},
volume = {5},
journal = {Communications in Mathematical Sciences},
number = {4},
publisher = {International Press of Boston},
pages = {851 -- 864},
year = {2007},
}

@article{nonlocveccalclong,
    author = {Du, Qiang and Gunzburger, Max and Lehoucq, R. B. and Zhou, Kun},
    title = {A NONLOCAL VECTOR CALCULUS, NONLOCAL VOLUME-CONSTRAINED PROBLEMS, AND NONLOCAL BALANCE LAWS},
    journal = {Mathematical Models and Methods in Applied Sciences},
    volume = {23},
    number = {03},
    pages = {493-540},
    year = {2013},
    doi = {10.1142/S0218202512500546},
    URL = {https://doi.org/10.1142/S0218202512500546}
}

@article{DuConv,
 ISSN = {00361429},
 URL = {http://www.jstor.org/stable/41149024},
 author = {Kun Zhou and Qiang Du},
 journal = {SIAM Journal on Numerical Analysis},
 number = {5},
 pages = {1759--1780},
 publisher = {Society for Industrial and Applied Mathematics},
 title = {MATHEMATICAL AND NUMERICAL ANALYSIS OF LINEAR PERIDYNAMIC MODELS WITH NONLOCAL BOUNDARY CONDITIONS},
 urldate = {2025-12-10},
 volume = {48},
 year = {2010}
}

@article{Emmrich2013WellposednessOT,
  title={Well-posedness of the peridynamic model with {L}ipschitz continuous pairwise force function},
  author={Etienne Emmrich and Dimitri Puhst},
  journal={Communications in Mathematical Sciences},
  year={2013},
  volume={11},
  pages={1039-1049},
  url={https://api.semanticscholar.org/CorpusID:56261348}
}

@article{EmmWeakMeasu,
    author = {Emmrich, Etienne and Puhst, Dimitri},    
    doi = {10.1088/0951-7715/28/1/285},
    url = {https://dx.doi.org/10.1088/0951-7715/28/1/285},
    year = {2014},
    month = {dec},
    publisher = {IOP Publishing},
    volume = {28},
    number = {1},
    pages = {285},
    title = {Measure-valued and weak solutions to the nonlinear peridynamic model in nonlocal elastodynamics},
    journal = {Nonlinearity},
}

@article{bromley2026blendedapproachevolvingphase,
title = {A blended approach for evolving phase fields using peridynamics: Cyclic loading in quasi-brittle fracture},
journal = {Journal of the Mechanics and Physics of Solids},
volume = {216},
pages = {106771},
year = {2026},
issn = {0022-5096},
doi = {https://doi.org/10.1016/j.jmps.2026.106771},
url = {https://www.sciencedirect.com/science/article/pii/S0022509626002723},
author = {Hayden Bromley and Robert Lipton}
}

@article{ConvStudiesInMeshFree,
title = {Convergence studies in meshfree peridynamic simulations},
journal = {Computers \& Mathematics with Applications},
volume = {71},
number = {11},
pages = {2432-2448},
year = {2016},
note = {Proceedings of the conference on Advances in Scientific Computing and Applied Mathematics. A special issue in honor of Max Gunzburger’s 70th birthday},
issn = {0898-1221},
doi = {https://doi.org/10.1016/j.camwa.2015.12.021},
url = {https://www.sciencedirect.com/science/article/pii/S0898122115005891},
author = {Pablo Seleson and David J. Littlewood}
}

@article{TianDuAnalysisApprox,
author = {Tian, Xiaochuan and Du, Qiang},
title = {Analysis and Comparison of Different Approximations to Nonlocal Diffusion and Linear Peridynamic Equations},
journal = {SIAM Journal on Numerical Analysis},
volume = {51},
number = {6},
pages = {3458-3482},
year = {2013},
doi = {10.1137/13091631X},
URL = {  https://doi.org/10.1137/13091631X}
}

@article{CHEN20111237,
title = {Continuous and discontinuous finite element methods for a peridynamics model of mechanics},
journal = {Computer Methods in Applied Mechanics and Engineering},
volume = {200},
number = {9},
pages = {1237-1250},
year = {2011},
issn = {0045-7825},
doi = {https://doi.org/10.1016/j.cma.2010.10.014},
url = {https://www.sciencedirect.com/science/article/pii/S0045782510002926},
author = {X. Chen and Max Gunzburger}
}

@article{FinElemApproxNonLoc,
title = {Finite element approximation of nonlocal dynamic fracture models},
journal = {Discrete and Continuous Dynamical Systems - B},
volume = {26},
number = {3},
pages = {1675-1710},
year = {2021},
issn = {1531-3492},
doi = {10.3934/dcdsb.2020178},
url = {https://www.aimsciences.org/article/id/7e1359b9-fe44-49f6-be2f-f88830a1746d},
author = {P. K. Jha and R. Lipton}
}

@Article{SillConv,
author={Silling, S. A.
and Lehoucq, R. B.},
title={Convergence of Peridynamics to Classical Elasticity Theory},
journal={Journal of Elasticity},
year={2008},
month={Oct},
day={01},
volume={93},
number={1},
pages={13-37},
issn={1573-2681},
doi={10.1007/s10659-008-9163-3},
url={https://doi.org/10.1007/s10659-008-9163-3}
}

@Article{Bellido2015,
author={Bellido, Jos{\'e} C.
and Mora-Corral, Carlos
and Pedregal, Pablo},
title={Hyperelasticity as a {$\Gamma$}-limit of peridynamics when the horizon goes to zero},
journal={Calculus of Variations and Partial Differential Equations},
year={2015},
month={Oct},
day={01},
volume={54},
number={2},
pages={1643-1670},
issn={1432-0835},
doi={10.1007/s00526-015-0839-9},
url={https://doi.org/10.1007/s00526-015-0839-9}
}

@article{Bellido2020BondBasedPD,
  title={Bond-Based Peridynamics Does Not Converge to Hyperelasticity as the Horizon Goes to Zero},
  author={Jos{\'e} Carlos Bellido and Javier Cueto and Carlos Mora‐Corral},
  journal={Journal of Elasticity},
  year={2020},
  volume={141},
  pages={273 - 289},
  url={https://api.semanticscholar.org/CorpusID:210116629}
}

@article{ImprovOnePointQuadAlgs,
title = {Improved one-point quadrature algorithms for two-dimensional peridynamic models based on analytical calculations},
journal = {Computer Methods in Applied Mechanics and Engineering},
volume = {282},
pages = {184-217},
year = {2014},
issn = {0045-7825},
doi = {https://doi.org/10.1016/j.cma.2014.06.016},
url = {https://www.sciencedirect.com/science/article/pii/S0045782514002023},
author = {Pablo Seleson}
}

@article{liphabbcohesivenumerical,
 author = {Jha, Prashant K and Lipton, Robert},
 journal = {SIAM Journal on Numerical Analysis},
 number = {2},
 pages = {906--941},
 publisher = {SIAM},
 title = {Numerical analysis of nonlocal fracture models in {H}{\"o}lder space},
 volume = {56},
 year = {2018}
}

@article{COSKUN2025106302,
title = {Quasi-brittle fracture: The blended approach},
journal = {Journal of the Mechanics and Physics of Solids},
volume = {205},
pages = {106302},
year = {2025},
issn = {0022-5096},
doi = {https://doi.org/10.1016/j.jmps.2025.106302},
url = {https://www.sciencedirect.com/science/article/pii/S0022509625002789},
author = {Semsi Coskun and Davood Damircheli and Robert Lipton}
}

@article{BobarFlorConv1DPD,
author = {Bobaru, Florin and Yang, Mijia and Alves, Leonardo Frota and Silling, Stewart A. and Askari, Ebrahim and Xu, Jifeng},
title = {Convergence, adaptive refinement, and scaling in {1D} peridynamics},
journal = {International Journal for Numerical Methods in Engineering},
volume = {77},
number = {6},
pages = {852-877},
doi = {https://doi.org/10.1002/nme.2439},
url = {https://onlinelibrary.wiley.com/doi/abs/10.1002/nme.2439},
year = {2009}
}

@article{XiaoQiangAsymptCompatRobDiscr,
author = {Tian, Xiaochuan and Du, Qiang},
title = {Asymptotically Compatible Schemes for Robust Discretization of Parametrized Problems with Applications to Nonlocal Models},
journal = {SIAM Review},
volume = {62},
number = {1},
pages = {199-227},
year = {2020},
doi = {10.1137/19M1296720},
URL = {https://doi.org/10.1137/19M1296720 }}

@misc{pflug2026convergencediscretizedlinearstatic,
      title={On The Convergence of the Discretized Linear Static State-Based Peridynamic Equations}, 
      author={Lukas Pflug and Michael Stingl and Max Zetzmann},
      year={2026},
      eprint={2603.02916},
      archivePrefix={arXiv},
      primaryClass={math.NA},
      url={https://arxiv.org/abs/2603.02916}, 
}

@Inbook{Seleson2018,
author="Seleson, Pablo
and Littlewood, David J.",
title="Numerical Tools for Improved Convergence of Meshfree Peridynamic Discretizations",
bookTitle="Handbook of Nonlocal Continuum Mechanics for Materials and Structures",
year="2018",
publisher="Springer International Publishing",
address="Cham",
pages="1--27",
isbn="978-3-319-22977-5",
doi="10.1007/978-3-319-22977-5_39-1",
url="https://doi.org/10.1007/978-3-319-22977-5_39-1"
}

@article{MACEK20071169,
title = {Peridynamics via finite element analysis},
journal = {Finite Elements in Analysis and Design},
volume = {43},
number = {15},
pages = {1169-1178},
year = {2007},
issn = {0168-874X},
doi = {https://doi.org/10.1016/j.finel.2007.08.012},
url = {https://www.sciencedirect.com/science/article/pii/S0168874X07001035},
author = {Richard W. Macek and Stewart A. Silling}
}

@article{PARKS2008777,
title = {Implementing peridynamics within a molecular dynamics code},
journal = {Computer Physics Communications},
volume = {179},
number = {11},
pages = {777-783},
year = {2008},
issn = {0010-4655},
doi = {https://doi.org/10.1016/j.cpc.2008.06.011},
url = {https://www.sciencedirect.com/science/article/pii/S0010465508002348},
author = {Michael L. Parks and Richard B. Lehoucq and Steven J. Plimpton and Stewart A. Silling}
}

@article{TRASK2019151,
title = {An asymptotically compatible meshfree quadrature rule for nonlocal problems with applications to peridynamics},
journal = {Computer Methods in Applied Mechanics and Engineering},
volume = {343},
pages = {151-165},
year = {2019},
issn = {0045-7825},
doi = {https://doi.org/10.1016/j.cma.2018.08.016},
url = {https://www.sciencedirect.com/science/article/pii/S004578251830402X},
author = {Nathaniel Trask and Huaiqian You and Yue Yu and Michael L. Parks}
}

@article{NIAZI2021102855,
title = {Crack nucleation in brittle and quasi-brittle materials: A peridynamic analysis},
journal = {Theoretical and Applied Fracture Mechanics},
volume = {112},
pages = {102855},
year = {2021},
issn = {0167-8442},
doi = {https://doi.org/10.1016/j.tafmec.2020.102855},
url = {https://www.sciencedirect.com/science/article/pii/S0167844220304316},
author = {Sina Niazi and Ziguang Chen and Florin Bobaru}
}

@article{JhaNumericalCO,
title = {Numerical convergence of finite difference approximations for state based peridynamic fracture models},
journal = {Computer Methods in Applied Mechanics and Engineering},
volume = {351},
pages = {184-225},
year = {2019},
issn = {0045-7825},
doi = {https://doi.org/10.1016/j.cma.2019.03.024},
url = {https://www.sciencedirect.com/science/article/pii/S0045782519301537},
author = {Prashant K. Jha and Robert Lipton}
}

@book{anainbanachsp,
    author = {Tuomas Hytönen and Jan van Neerven and Mark Veraar and Lutz Weis},
    title = {Analysis in Banach Spaces. Volume I: Martingales and Littlewood-Paley Theory},
    publisher = {Springer},
    address = {Cham},
    year = {2016},
    url = {https://doi.org/10.1007/978-3-319-48520-1}
}

@Article{Scabbia2023,
author={Scabbia, Francesco
and Zaccariotto, Mirco
and Galvanetto, Ugo},
title={Accurate computation of partial volumes in 	{3D} peridynamics},
journal={Engineering with Computers},
year={2023},
month={Feb},
day={01},
volume={39},
number={1},
pages={959-991},
issn={1435-5663},
doi={10.1007/s00366-022-01725-3},
url={https://doi.org/10.1007/s00366-022-01725-3}
}

@article{LIU2024109115,
title = {On the one-point quadrature discretization in peridynamics: A novel perspective from {M}onte {C}arlo integration},
journal = {Computer Physics Communications},
volume = {298},
pages = {109115},
year = {2024},
issn = {0010-4655},
doi = {https://doi.org/10.1016/j.cpc.2024.109115},
url = {https://www.sciencedirect.com/science/article/pii/S0010465524000389},
author = {Hengjie Liu and Ziguang Chen}
}

@article{PrudTreatBnd,
title = {On the treatment of boundary conditions for bond-based peridynamic models},
journal = {Computer Methods in Applied Mechanics and Engineering},
volume = {372},
pages = {113391},
year = {2020},
issn = {0045-7825},
doi = {https://doi.org/10.1016/j.cma.2020.113391},
url = {https://www.sciencedirect.com/science/article/pii/S0045782520305764},
author = {Serge Prudhomme and Patrick Diehl},
}

@article{Le2018,
author={Le, Q. V.
and Bobaru, F.},
title={Surface corrections for peridynamic models in elasticity and fracture},
journal={Computational Mechanics},
year={2018},
month={Apr},
day={01},
volume={61},
number={4},
pages={499-518},
issn={1432-0924},
doi={10.1007/s00466-017-1469-1},
url={https://doi.org/10.1007/s00466-017-1469-1}
}

@article{zhaoGenFNM,
    author = {Jiangming Zhao and Siavash Jafarzadeh and Ziguang Chen and Florin Bobaru},
    title = {Enforcing local boundary conditions in peridynamic models of diffusion with singularities and on arbitrary domains},
    journal = {Engineering with Computers},
    year = 2024,
    doi={10.1007/s00366-024-01995-z}
}

@article{novelBCinPD,
author = {Scabbia, Francesco and Zaccariotto, Mirco and Galvanetto, Ugo},
title = {A novel and effective way to impose boundary conditions and to mitigate the surface effect in state-based Peridynamics},
journal = {International Journal for Numerical Methods in Engineering},
volume = {122},
number = {20},
pages = {5773-5811},
doi = {https://doi.org/10.1002/nme.6773},
url = {https://onlinelibrary.wiley.com/doi/abs/10.1002/nme.6773},
year = {2021}
}

@article{TupekRadovitzky2014,
title = {An extended constitutive correspondence formulation of peridynamics based on nonlinear bond-strain measures},
journal = {Journal of the Mechanics and Physics of Solids},
volume = {65},
pages = {82-92},
year = {2014},
issn = {0022-5096},
doi = {https://doi.org/10.1016/j.jmps.2013.12.012},
url = {https://www.sciencedirect.com/science/article/pii/S0022509614000027},
author = {M.R. Tupek and R. Radovitzky}
}

@article{GuanGunzNonlocStabAcc,
author = {Guan, Qingguang and Gunzburger, Max},
title = {Stability and accuracy of time-stepping schemes and dispersion relations for a nonlocal wave equation},
journal = {Numerical Methods for Partial Differential Equations},
volume = {31},
number = {2},
pages = {500-516},
doi = {https://doi.org/10.1002/num.21931},
url = {https://onlinelibrary.wiley.com/doi/abs/10.1002/num.21931},
year = {2015}
}

@Article{Lipton2018,
author={Lipton, Robert
and Said, Eyad
and Jha, Prashant},
title={Free Damage Propagation with Memory},
journal={Journal of Elasticity},
year={2018},
month={Dec},
day={01},
volume={133},
number={2},
pages={129-153},
issn={1573-2681},
doi={10.1007/s10659-018-9672-7},
url={https://doi.org/10.1007/s10659-018-9672-7}
}

@article{YE20071075,
title = {A generalized {G}ronwall inequality and its application to a fractional differential equation},
journal = {Journal of Mathematical Analysis and Applications},
volume = {328},
number = {2},
pages = {1075-1081},
year = {2007},
issn = {0022-247X},
doi = {https://doi.org/10.1016/j.jmaa.2006.05.061},
url = {https://www.sciencedirect.com/science/article/pii/S0022247X06005956},
author = {Haiping Ye and Jianming Gao and Yongsheng Ding}
}

\appendix
\section{Appendix}
\subsection{Proofs for Section \ref{sec:SemiDiscrSystem}}
\begin{lemma}[Uniform Bounds on Discrete Integrals with Partial Area Weights]\label{lemm:alt34weights}
For all $i\in I^\kappa$ and $x\in \Omega_i$,
\begin{align*}
        &0\leq \sum_{j\not=i} \int_{\Omega_j}\frac{\omega_\delta^{ij}}{\|x'-x\|_2} \dd x' \leq \boundkerdiscrinteg,\\
        &0\leq \sum_{j\not=i} V_j \frac{\omega_\delta^{ij}}{\|x_j-x_i\|_2}\leq \boundkerdiscrinteg,
\end{align*}
for a constant $\boundkerdiscrinteg>0$ independent of $\delta/\sqrt{d}> \kappa>0$.
\end{lemma}
\begin{proof}
We deal with these proofs by using the Assumptions \ref{ass:quadweights}, more precisely, the fact that $\omega^{ij}_\delta$ and $\omega_\delta(\|x_j-x_i\|_2)$ can only differ on the volume $B_{\delta + \sqrt{d}\kappa}(x_i)\backslash B_{\delta - \sqrt{d}\kappa}(x_i)$. This follows as the only $\Omega_j$ on which the values differ must fulfill $ \delta + \sqrt{d}/2\kappa \geq \|x_i-x_j\|_2 \geq \delta - \sqrt{d}/2\kappa$, therefore $\Omega_j \subset B_{\delta + \sqrt{d}\kappa}(x_i)\backslash B_{\delta - \sqrt{d}\kappa}(x_i)$ due to $\diam(B_j) = \sqrt{d}\kappa$. We then obtain
\begin{align*}
    &\Bigg|\sum_{j\not=i} \int_{\Omega_j}\frac{\omega_\delta^{ij}}{\|x'-x\|_2} \dd x' - \sum_{j\not=i} \int_{\Omega_j}\frac{\omega_\delta(\|x_i-x_j\|_2)}{\|x'-x\|_2} \dd x'\Bigg| \\
    &\leq \omega_\delta(\delta - 
\sqrt{d}\kappa) \vol(B_{\delta + \sqrt{d}\kappa}(0)\backslash B_{\delta - \sqrt{d}\kappa}(0)) \frac{1}{\delta - \sqrt{d}\kappa} \xrightarrow[]{\kappa \searrow 0} 0,
\end{align*}
independently of $i\in I^\kappa$.
The convergence follows from the fact that $\omega_\delta$ is non-increasing and finite on $(0,\delta]$. Using the triangle inequality and Lemma \ref{lemm:unifbndinteg}, yields the claim. \par
The other integral follows similarly, as again, (2) of Assumptions \ref{ass:quadweights} implies that if $w^{ij}$ and $\omega_\delta(\|x_i-x_j\|_2)$ differ, then $\|x_i-x_j\|\geq \delta - \sqrt{d}\kappa/2$.
\end{proof}

\begin{lemma}[Error Bound on the Discrete Weighting Function with Partial Area Weights] \label{lemm:alt39weights}
Let $i\in I^\kappa$, then, for all $x\in \Omega_i$,
\begin{equation*}
   \sum_{i\not=j}\int_{\Omega_j} |\omega_\delta^{ij}-\omega_\delta(\|x'-x\|_2)| \dd x' \leq D_\kappa,
\end{equation*}
with $0<D_\kappa< \infty$ independent of $x$ and $i$ as well as $ O(\kappa) \ni D_\kappa \searrow 0$ as $\kappa \searrow 0$, with $0< \kappa < \delta/\sqrt{d}$.\end{lemma}
\begin{proof} We use a similar argument as in Lemma \ref{lemm:alt34weights}, reducing the argument to bounding the error term 
    \begin{align*}
        &\sum_{i\not=j}\int_{\Omega_j} |\omega_\delta^{ij}-\omega_\delta(\|x_j-x_i\|_2)| \dd x' \\
        &\leq \omega_\delta(\delta - 
\sqrt{d}\kappa) \vol(B_{\delta + \sqrt{d}\kappa}(0)\backslash B_{\delta - \sqrt{d}\kappa}(0))  \in O(\kappa)
    \end{align*}
    The claim now follows from the triangle inequality and Lemma \ref{lemm:discrweighterrbnd}.
\end{proof}

\end{document}